\documentclass[11pt,reqno]{amsart}
\usepackage[a4paper,margin=25mm]{geometry}
\usepackage[T1]{fontenc}
\usepackage{lmodern,microtype,amsmath,amssymb,amsthm,mathtools,booktabs,mathrsfs,longtable,array}
\usepackage{needspace}
\usepackage{hyperref}
\hypersetup{hidelinks,
 pdftitle={An explicit polynomial counterexample to Connes' embedding conjecture},
 pdfauthor={Jiaqi Wang and Lihong Zhi},
 pdfsubject={Trace-positive noncommutative polynomials and finite group certificates}}
\numberwithin{equation}{section}
\newtheorem{theorem}{Theorem}[section]
\newtheorem{proposition}[theorem]{Proposition}
\newtheorem{lemma}[theorem]{Lemma}
\newtheorem{corollary}[theorem]{Corollary}
\theoremstyle{definition}

\theoremstyle{remark}
\newtheorem{remark}[theorem]{Remark}
\newcommand{\C}{\mathbb C}
\newcommand{\R}{\mathbb R}
\newcommand{\Q}{\mathbb Q}
\newcommand{\Z}{\mathbb Z}
\newcommand{\F}{\mathbb F}

\newcommand{\tr}{\operatorname{tr}}
\newcommand{\Tr}{\operatorname{Tr}}
\newcommand{\Ad}{\operatorname{Ad}}

\newcommand{\gc}[2]{\{#1,#2\}}
\newcommand{\EL}{\operatorname{EL}}
\newcommand{\SL}{\operatorname{SL}}

\newcommand{\rank}{\operatorname{rank}}
\newcommand{\spec}{\operatorname{spec}}
\newcommand{\HS}{\mathrm{HS}}
\newcommand{\cE}{\mathcal E}
\newcommand{\eps}{\varepsilon}
\newcommand{\ip}[2]{\langle #1,#2\rangle}
\DeclareMathOperator{\sgn}{sgn}
\newcommand{\cyc}{\mathrel{\sim_{\rm cyc}}}
\newcommand{\HC}{\mathcal H}
\title[An explicit polynomial counterexample to Connes' embedding conjecture]{An explicit polynomial counterexample to Connes' embedding conjecture}

\author[Jiaqi Wang]{Jiaqi Wang}
\address{State Key Laboratory of Mathematical Sciences, Academy of Mathematics
and Systems Science, Chinese Academy of Sciences, Beijing 100190, China;
University of Chinese Academy of Sciences, Beijing 100049, China}
\email{jiaqiwang@amss.ac.cn}
\author[Lihong Zhi]{Lihong Zhi}
\address{State Key Laboratory of Mathematical Sciences, Academy of Mathematics
and Systems Science, Chinese Academy of Sciences, Beijing 100190, China;
University of Chinese Academy of Sciences, Beijing 100049, China}
\email{lzhi@mmrc.iss.ac.cn}
\date{}
\begin{document}

\begin{abstract}
We construct an explicit Hermitian polynomial in six selfadjoint
variables,
\[
 f=-1+\omega+\omega^*
   +M\sum_{j=1}^{15}(2-u_j-u_j^*),
\]
with integer coefficients, degree $72$, and exactly $33$ monomials.
Here $\omega,u_1,\ldots,u_{15}$ are specified words, $*$ reverses
words, and $M$ is a specified positive integer.
Its normalized trace is at least $3/4$ on selfadjoint matrix
contraction tuples of every dimension, but equals $-1$ at a
specified tuple of selfadjoint unitaries in a group von Neumann
algebra. Thus $f$ is a counterexample to the algebraic formulation
of Connes' embedding conjecture.

We also construct a Hermitian quartic in $37$ selfadjoint variables,
with coefficients in $\Z[i]=\Z+i\Z$ and $387$ monomials, attaining
the same trace bounds without norm restrictions on selfadjoint
matrix inputs. Degree four is minimal in this unrestricted setting.
Finally, an encoding in two selfadjoint variables yields a
counterexample on the contraction domain with integer coefficients
and degree at most $48$; two variables are minimal.
\end{abstract}
\maketitle

\section{Introduction}\label{sec:intro}

Connes' embedding conjecture, formulated in
1976~\cite[p.~105]{Connes}, asked whether every $\mathrm{II}_1$
factor with a separable predual embeds into an ultrapower of the
hyperfinite $\mathrm{II}_1$ factor. Equivalently, it asks whether
finitely many joint moments of selfadjoint contractions in such
a factor can be approximated arbitrarily well by normalized traces of matrix
contractions~\cite[Proposition~3.17]{KS}.
This approximation viewpoint connects the conjecture with
Voiculescu's work on random matrix limits~\cite{Voiculescu}
and Kirchberg's equivalences involving tensor products of
$C^*$-algebras and the QWEP conjecture~\cite{Kirchberg}.

The connection with Tsirelson's problem was established by
Junge and collaborators~\cite{Junge} and Fritz~\cite{Fritz}.
Following Slofstra's nonclosure theorem~\cite{Slofstra},
Ji, Natarajan, Vidick, Wright, and Yuen obtained a negative
answer to Connes' embedding conjecture through
$\mathrm{MIP}^*=\mathrm{RE}$~\cite{JNVWY}.
Subsequent developments include Lin's theorem
$\mathrm{MIP}^{\mathrm{co}}=\mathrm{coRE}$ and alternative proof of
$\mathrm{MIP}^*=\mathrm{RE}$~\cite{Lin}, and Manzoor's construction
of a countable probability measure preserving equivalence relation
whose von Neumann algebra is not Connes embeddable~\cite{Manzoor}.

An algebraic formulation due to Klep and
Schweighofer~\cite{KS} asks whether every Hermitian polynomial
with nonnegative trace on all selfadjoint matrix contraction
tuples admits, after any positive scalar perturbation, a
certificate in the contraction quadratic module modulo
commutators. Related tracial moment and positivity formulations
were developed by Hadwin~\cite{Hadwin} and
R\u{a}dulescu~\cite{Rad04}; see also the real formulation
in~\cite{BDKS} and Ozawa's unified treatment~\cite{OzawaCEC}.

The negative solution of the embedding conjecture implies the
existence of polynomials separating matrix traces from traces
on finite von Neumann algebras. The corresponding existence
statement for unrestricted matrix tuples is recorded
in~\cite[Proposition~6.4]{KSV}.
Our purpose is to give such a polynomial explicitly, including
its coefficients, a negative group substitution, and bounds
valid in every matrix dimension. The construction uses the
group obstruction of Kun and Thom~\cite{KunThom}, Thom's
normalization argument~\cite{Thom}, and the spectral and
commutant results of Alekseev, Liu, and Thom~\cite{ALT}.

\subsection{Main results}

For $N\ge1$, let $\mathcal A_N=\C\langle X_1,\ldots,X_N\rangle$
be the free $*$-algebra on selfadjoint variables. The involution
conjugates coefficients and reverses words; a polynomial $p$
is Hermitian if $p^*=p$. Write $\Z\langle X_1,\ldots,X_N\rangle$
for the subalgebra with integer coefficients, and $p(A)$ for
evaluation at a tuple $A=(A_1,\ldots,A_N)$.
Let $M_d=M_d(\C)$, let $\Tr$ be the usual matrix trace, and put
$\tr_d=d^{-1}\Tr$. We use $\|\cdot\|$ for the operator norm.
A selfadjoint contraction tuple satisfies $A_j^*=A_j$ and
$\|A_j\|\le1$ for every $j$.

For the words $\omega,u_1,\ldots,u_{15}$ in
\eqref{eq:signalword}--\eqref{eq:penaltywords}, define
\begin{equation}\label{eq:penaltyscheme}
 f=-1+\omega+\omega^*
   +M\sum_{j=1}^{15}(2-u_j-u_j^*),
 \qquad M=2^{\,2^{55227}+54}.
\end{equation}
This is the same integer $M$ as in \eqref{eq:globalconstants}.
Each penalty $2-u_j-u_j^*$ has nonnegative trace on selfadjoint
matrix contraction tuples. Appendix~\ref{app:nativeexpanded} gives
the complete expansion of $f$ in the six variables.

Let $\Lambda$ be the group defined in \eqref{eq:s2:Lambda}.
Its group von Neumann algebra $L(\Lambda)$ has canonical trace
$\tau_\Lambda$, which is one on the identity group element and
zero on every other group element. Let $a=(a_1,\ldots,a_6)$ be
the tuple of selfadjoint unitaries in $L(\Lambda)$ specified in
\eqref{eq:s2:involution-tuple}.

\begin{theorem}\label{thm:main}\label{thm:native}\label{csimp:r3:nativethm}
The polynomial $f$ in \eqref{eq:penaltyscheme} is Hermitian,
has integer coefficients, degree $72$, and exactly $33$ monomials.
For every $d\ge1$ and every selfadjoint contraction tuple
$A\in M_d^6$,
\begin{equation}\label{eq:matrixpositive}
 \tr_d f(A)\ge\frac34,
\end{equation}
whereas the specified group tuple satisfies
\begin{equation}\label{eq:negative}
 \tau_\Lambda f(a)=-1.
\end{equation}
\end{theorem}

The negative evaluation is proven in
Proposition~\ref{prop:s2:negative-evaluation}, and the matrix
bound is established in Theorem~\ref{thm:matrixpositivity}.

The main text develops this polynomial. Appendix~\ref{sec:quartic}
uses a reduction following R\u{a}dulescu~\cite{Rad99,CapRad}
to construct a Hermitian quartic in $37$ selfadjoint variables,
with $387$ monomials and coefficients in the Gaussian integers
$\Z[i]=\{a+ib:a,b\in\Z\}$. It has the same trace bounds as $f$
without norm restrictions on the matrix inputs; degree four is
minimal for such unrestricted separation.
Appendix~\ref{app:fewvariables} encodes this quartic in two
selfadjoint variables, giving a polynomial with integer
coefficients and degree at most $48$ whose trace is positive on
matrix contraction pairs and negative at a specified contraction
pair in a finite tracial von Neumann algebra. Two variables are
minimal for separation on contractions.

\subsection{The algebraic consequence}

The contraction quadratic module in $\mathcal A_N$ is
\begin{equation}\label{eq:cubemoduleintro}
 \mathcal M_N=
 \left\{\sum_\ell a_\ell^*a_\ell+
 \sum_{j=1}^N\sum_\ell b_{j\ell}^*(1-X_j^2)b_{j\ell}:
 a_\ell,b_{j\ell}\in\mathcal A_N\right\},
\end{equation}
where all sums are finite. We write $[a,b]=ab-ba$ for the
additive commutator and $[\mathcal A_N,\mathcal A_N]_{\rm lin}$
for the linear span of these commutators. Two polynomials are
cyclically equivalent if their difference belongs to this span.

The algebraic formulation of Connes' embedding conjecture asserts
that, for every $N$ and every Hermitian $p\in\mathcal A_N$,
the condition
\begin{equation}\label{eq:algebraicpositive}
 \tr_d p(A)\ge0
 \quad\text{for every }d\ge1\text{ and every selfadjoint
 contraction tuple }A\in M_d^N
\end{equation}
implies
\begin{equation}\label{eq:algebraicconnes}
 p+\eps\in\mathcal M_N+
 [\mathcal A_N,\mathcal A_N]_{\rm lin}
 \qquad\text{for every }\eps>0.
\end{equation}
The equivalence with the embedding conjecture is proved
in~\cite[Theorem~3.18]{KS}, with the real and complex formulations
related in~\cite{BDKS}. The converse implication from
\eqref{eq:algebraicconnes} to \eqref{eq:algebraicpositive} follows
by taking matrix traces and letting $\eps$ tend to zero.

Theorem~\ref{thm:main} gives the explicit failure
\begin{equation}\label{eq:introexclusion}
 f+\eps\notin\mathcal M_6+
 [\mathcal A_6,\mathcal A_6]_{\rm lin}
 \qquad(0\le\eps<1).
\end{equation}
Indeed, evaluation at $a$ gives nonnegative trace on
$\mathcal M_6$ and zero trace on commutators, whereas
$\tau_\Lambda(f(a)+\eps)=-1+\eps<0$; see
Corollary~\ref{cor:module}.
The positive perturbation and the contraction constraints are
essential to this formulation. For example, the noncommutative
Motzkin polynomial in~\cite[Example~4.4]{KS} is not cyclically
equivalent to a sum of Hermitian squares, but every positive
scalar perturbation is cyclically equivalent to an element
of the contraction quadratic module.

\subsection{Construction and proof strategy}\label{subsec:priorwork}

Starting from the groups of Kun and
Thom~\cite[Theorem~E]{KunThom}, we use
\begin{equation}\label{eq:groups}
\begin{split}
 H&=\EL_3(\F_5[\xi_1,\xi_2,\xi_3]),\\
 G&=\EL_3(\F_5[\xi_1^{\pm1},\xi_2^{\pm1},\xi_3^{\pm1}])
       \rtimes\SL_3(\Z),\\
 \Gamma&=G*_H\bigl(H\times\langle v\mid v^2=1\rangle\bigr).
\end{split}
\end{equation}
Here $\F_5=\Z/5\Z$ and the $\xi_i$ are commuting indeterminates.
For a unital ring $R$, the group $\EL_3(R)$ is generated by
$I_3+bE_{ij}$ with $b\in R$ and $i\ne j$, where $I_3$ is the
identity matrix and $E_{ij}$ has its only nonzero entry, equal to
one, in position $(i,j)$. The group $\SL_3(\Z)$ consists of
integer matrices of determinant one. Its action on the Laurent
polynomial ring is given by $\xi^n\mapsto\xi^{Un}$ for
$U\in\SL_3(\Z)$ and $n\in\Z^3$, where
$\xi^n=\xi_1^{n_1}\xi_2^{n_2}\xi_3^{n_3}$, and acts on matrix
entries accordingly. We view $H$ as a subgroup of $G$ by
$g\mapsto(g,I_3)$. The amalgamated free product $*_H$ identifies
the two copies of $H$, so $\Gamma$ adjoins an involution $v$
commuting with every element of $H$.

The group obstruction comes from conjugation that maps $H$
into a proper subgroup of itself. In matrix algebras, a unitary
conjugation that maps a unital $*$-subalgebra into itself must
map it onto itself, because the two algebras have the same
finite dimension. The compression and centralization relations
then force the commutator defining the signal to vanish.
In the group model, the corresponding commutator has nonzero
trace of its squared modulus. Section~\ref{sec:polynomial}
preserves this distinction while adjoining further involutions
to obtain $\Lambda$ and the six coordinates.

Thom~\cite[Theorems~1.2--1.3 and \S5.2]{Thom} and
Alekseev, Liu, and Thom~\cite[Theorem~B and Corollary~C]{ALT}
establish the group obstruction through normalization and
relative commutants in tracial matrix ultraproducts.
Our analytic task is to obtain a finite tolerance for the
same rigidity conclusion that is independent of the matrix
dimension. Section~\ref{sec:certificates} obtains spectral
bounds from the finite subgroup relations and the integer
certificate of Netzer and Thom~\cite{NT}, then constructs
algebras describing approximate commutation with the corresponding inputs.
Section~\ref{sec:reconstruction} proves the quantitative
normalization estimate needed to reverse approximate inclusions.

The analytic argument uses $97$ full relations. The six
coordinates encode fifteen retained equations, from which
Section~\ref{sec:gate} recovers all full relations with explicit
error bounds. These fifteen equations differ from the fifteen
integer relations used to transfer the certificate of Netzer and Thom
in Lemma~\ref{lem:Kshort}.
For arbitrary selfadjoint matrix contractions, a small total
penalty controls the distance to a tuple of selfadjoint unitaries
and the defects of its relations. The resulting reflection
estimate makes the trace of the signal close to $1$.
If the total penalty is larger, its contribution multiplied
by $M$ gives the lower bound directly. At the specified group
tuple, every penalty vanishes and the signal has trace $-1$.

\subsection{Notation and preliminaries}\label{sec:prelim}

An element $u$ of a unital $*$-algebra is unitary if
$u^*u=uu^*=1$. A symmetry is a selfadjoint unitary, equivalently
an element $s$ satisfying $s=s^*$ and $s^2=1$.
We write $\mathcal U(d)$ for the unitary group of $M_d$.
For words or operators $u,v$, write
$\gc{u}{v}=uvu^*v^*$; for unitary elements, this is the group
commutator.

For a finite von Neumann algebra with normalized faithful trace
$\tau$, set
\[
 |x|=(x^*x)^{1/2},\qquad
 \|x\|_{p,\tau}=\bigl(\tau(|x|^p)\bigr)^{1/p}
 \quad(p=2,4).
\]
Unless a trace is specified, matrix $p$-norms use $\tr_d$. Thus
\[
 \|x\|_2^2=\frac{\Tr(x^*x)}d,\qquad
 \langle x,y\rangle=\tr_d(y^*x),\qquad
 \|x\|_{\HS}=\bigl(\Tr(x^*x)\bigr)^{1/2},
\]
where $\|\cdot\|_{\HS}$ is the unnormalized Hilbert--Schmidt norm.
Auxiliary rectangular matrices retain the denominator $d$.
On a corner $e_iM_de_i$ with nonzero projection $e_i$, we write
$\tau_i(x)=\tr_d(x)/\tr_d(e_i)$ when using the normalized corner
trace. Function $L^p$ norms use the stated probability weights.

For a polynomial $p=\sum_w c_ww$ in distinct words, define its
coefficient norm by
\[
 \|p\|_{\mathrm{coeff},1}=\sum_w|c_w|.
\]
Each coefficient estimate specifies the alphabet: either the formal
complex symbols and their adjoints in an auxiliary dictionary, or
the primitive selfadjoint variables. This coefficient norm differs
from function $L^1$ and matrix norms.

For a linear map $\Phi$ between matrix spaces, set
\begin{equation}\label{eq:mapnorms}
 \|\Phi\|_{\infty\to2}
 =\sup_{\|x\|\le1}\|\Phi(x)\|_2,
 \qquad
 \|\Phi\|_{2\to2}
 =\sup_{\|x\|_2\le1}\|\Phi(x)\|_2.
\end{equation}
Also set
$\|\Phi\|_{\infty\to\infty}=\sup_{\|x\|\le1}\|\Phi(x)\|$.
All suprema include nonselfadjoint inputs.
For domain $M_d$ with trace $\tr_d$,
$\|\Phi\|_{\infty\to2}\le\|\Phi\|_{2\to2}$.
The mixed norm is written $\|\Phi\|_{\infty,2}$ in~\cite{PSS}.

For a unital $*$-subalgebra $A\subseteq M_d$, write
\[
 A_1=\{a\in A:\|a\|\le1\},\qquad
 A'=\{x\in M_d:xa=ax\text{ for all }a\in A\}
\]
for its operator norm unit ball and its commutant in $M_d$, and
put $\mathcal U(A)=A\cap\mathcal U(d)$.
The trace-preserving conditional expectation $E_A:M_d\to A$ is
the orthogonal projection onto $A$ for the above inner product.
It satisfies $E_A(bxc)=bE_A(x)c$ for $b,c\in A$ and $x\in M_d$,
and is contractive in both the operator norm and the normalized
$2$-norm; see~\cite[Theorem~6.2.2]{Peterson}.

For $u\in\mathcal U(d)$, put $a=E_A(u)$.
Since $A$ is finite dimensional, the partial isometry in the polar
decomposition of $a$ extends to a unitary $v\in\mathcal U(A)$
with $a=v|a|$. Every such completion satisfies
\begin{equation}\label{eq:unitaryrounding}
 \|u-v\|_2\le\sqrt2\,\|u-E_A(u)\|_2.
\end{equation}
Indeed, $\|a\|\le1$ and $u-a$ is orthogonal to $A$, so
\[
 \|a-v\|_2^2=\tr_d\bigl((1-|a|)^2\bigr)
 \le1-\|a\|_2^2=\|u-a\|_2^2.
\]
Since $a-v\in A$, orthogonality gives
$\|u-v\|_2^2=\|u-a\|_2^2+\|a-v\|_2^2
\le2\|u-a\|_2^2$.
See also~\cite[Theorem~2.2]{Higham}.

The paper is organized as follows. Section~\ref{sec:polynomial}
introduces the group model, derives the finite relations, and
constructs $f$ and its negative group evaluation.
Section~\ref{sec:certificates} establishes spectral estimates
and reconstructs algebras describing approximate commutation.
Section~\ref{sec:reconstruction} proves quantitative normalization.
Section~\ref{sec:gate} combines these results with relation recovery
and rounding to prove matrix positivity. Appendices~\ref{sec:quartic}
and~\ref{app:fewvariables} give the quartic construction and its
encoding in two selfadjoint variables, respectively;
Appendix~\ref{app:nativeexpanded} lists the $33$ monomials of $f$.

\section{Finite relations and the separating polynomial}\label{sec:polynomial}

The construction has three stages. We first identify elementary matrices
and integer actions in the group model and record the finite relations
used in the analytic proof. We then exploit symmetries of these elements
to encode the required relations using six involutions. Finally, we
replace those involutions by free selfadjoint variables and form the
polynomial from one signal word and fifteen relation penalties.
Throughout this section, $\mathcal I=\{1,2,3\}$.

\subsection{The group model and its basic elements}
\label{subsec:exactsolutions}
Recall the groups from \eqref{eq:groups}:
\[
\begin{gathered}
 H=\operatorname{EL}_3(\mathbb F_5[\xi_1,\xi_2,\xi_3]),
\\
 G=\operatorname{EL}_3(\mathbb F_5[\xi_1^{\pm1},\xi_2^{\pm1},
 \xi_3^{\pm1}])\rtimes\operatorname{SL}_3(\mathbb Z),
\\
 \Gamma=G*_H\bigl(H\times\langle v\mid v^2=1\rangle\bigr).
\end{gathered}
\]
Here $\mathbb F_5=\mathbb Z/5\mathbb Z$, the $\xi_i$ are commuting
indeterminates, and $\operatorname{EL}_3(R)$ is generated by
$e_{ij}(b)=I_3+bE_{ij}$ for $i\ne j$ and $b\in R$.
Here $I_3$ is the identity matrix and $E_{ij}$ has its only nonzero
entry, equal to $1$, in position $(i,j)$.
The action of $A\in\operatorname{SL}_3(\mathbb Z)$ sends
$\xi^n$ to $\xi^{An}$, where
$\xi^n=\xi_1^{n_1}\xi_2^{n_2}\xi_3^{n_3}$ and $n\in\mathbb Z^3$.
The amalgam identifies the two copies of $H$ and adjoins an involution
$v$ commuting with every element of $H$.

We use a subscript zero for fixed group elements.
Our construction starts from the following five elementary
matrices in $H$:
\begin{equation}\label{eq:s2:basic-group-roots}
 x_0=e_{12}(1),\qquad y_0=e_{23}(1),\qquad
 z_{i,0}=e_{31}(\xi_i)\quad(i\in\mathcal I).
\end{equation}
These belong to the standard root subgroups corresponding
to the matrix positions $(1,2)$, $(2,3)$, and $(3,1)$.
We refer to them as the five basic root elements.
The choice of these elements is motivated by the three finite subgroups
\[
 \langle x_0\rangle,\qquad \langle y_0\rangle,\qquad
 \langle z_{1,0},z_{2,0},z_{3,0}\rangle,
\]
which have orders $5,5,125$, respectively.
Each pair generates a finite group with central commutators.
The structure of these groups is used to establish the spectral
estimate in Section~\ref{sec:certificates}.
We do not assume that the five selected elements generate all of $H$.

Their commutators give seven further root elements:
\begin{equation}\label{eq:s2:derived-group-roots}
\begin{aligned}
 c_0&=\gc{x_0}{y_0}=e_{13}(1),\\
 d_{i,0}&=\gc{x_0}{z_{i,0}}=e_{32}(-\xi_i),&
 e_{i,0}&=\gc{y_0}{z_{i,0}}=e_{21}(\xi_i)
 \quad(i\in\mathcal I).
\end{aligned}
\end{equation}
The six integer action elements are
\begin{equation}\label{eq:s2:group-actions}
 U_{ij,0}=(I_3,I_3+E_{ij})\in G\qquad(i\ne j).
\end{equation}
Conjugation by $U_{ij,0}$ fixes $x_0,y_0$ and $z_{k,0}$ for $k\ne j$,
and replaces $\xi_j$ by $\xi_i\xi_j$. In particular,
\[
 U_{ij,0}z_{j,0}U_{ij,0}^{-1}=e_{31}(\xi_i\xi_j).
\]
The elementary matrix on the right can be expressed
as a word in the five basic elements:
\begin{equation}\label{eq:s2:group-targets}
 d_{i,0}^{-1}e_{j,0}d_{i,0}e_{j,0}^{-1}
 =e_{31}(\xi_i\xi_j)\qquad(i<j).
\end{equation}
These formulas explain the compression relations introduced below.

The obstruction comes from the difference between forward and inverse
conjugation. Put $u=U_{12,0}$, $h=z_{2,0}$, and $W=uvu^{-1}$.
Then
\[
 uhu^{-1}=e_{31}(\xi_1\xi_2)\in H,\qquad
 u^{-1}hu=e_{31}(\xi_1^{-1}\xi_2)\notin H.
\]
Writing $k=u^{-1}hu$, we have
\[
 u^{-1}(h^{-1}WhW)u=k^{-1}vkv.
\]
The right-hand side is a reduced nonidentity word in the amalgam
$\Gamma$. The canonical group trace therefore gives
\begin{equation}\label{eq:s2:negative-signal}
 \tau_\Gamma(h^*WhW)=0,\qquad
 \tau_\Gamma([h,W]^*[h,W])=2.
\end{equation}
We will preserve this nonidentity word while adjoining the symmetries
needed for the six coordinates.

\subsection{The full relations used in the proof}
\label{subsec:fullrelations}

We now extract the finite relations used by the analytic argument.
Each family comes from an explicit feature of the group model:
elementary-matrix multiplication, the integer actions, or the
centralization built into the amalgam.

To formulate these relations for matrix solutions, let
\[
 x,y,z_1,z_2,z_3,c,d_1,d_2,d_3,e_1,e_2,e_3,
 \qquad U_{ij}\ (i\ne j),\qquad v
\]
be unitary entries, with $v$ a symmetry.
The corresponding fixed group elements carry a subscript zero.
In particular, put
\[
 \mathcal B=\{x,y,z_1,z_2,z_3\}.
\]
At the group substitution, these five entries take the values
$x_0,y_0,z_{1,0},z_{2,0},z_{3,0}$ specified in
\eqref{eq:s2:basic-group-roots}.
For arbitrary unitary tuples, the entries $c,d_i,e_i$ are
initially separate; the crossing equations below constrain
them to be the appropriate commutators.

The relationships among the elementary matrices follow from
the multiplication rule
\[
 E_{ij}E_{k\ell}=\delta_{jk}E_{i\ell},
\]
where $\delta_{jk}$ is the Kronecker delta. In particular,
\[
 \gc{x_0}{y_0}=c_0,\qquad
 \gc{x_0}{z_{i,0}}=d_{i,0},\qquad
 \gc{y_0}{z_{i,0}}=e_{i,0}.
\]
For example, the second identity is equivalent to
$x_0z_{i,0}=d_{i,0}z_{i,0}x_0$.
The same multiplication rule gives the required commutations.
For instance, $E_{32}E_{31}=E_{31}E_{32}=0$ implies
$[d_{i,0},z_{j,0}]=0$.
Finally, since the coefficient ring has characteristic five
and $E_{ij}^2=0$ for $i\ne j$,
\[
 (I_3+bE_{ij})^5=I_3.
\]

These identities motivate the following relations on the
unitary entries:
\begin{equation}\label{eq:s2:root-commutations}
\begin{gathered}
 {}[z_i,z_j]=[d_i,d_j]=[e_i,e_j]=0\qquad(i<j),\\
 [c,x]=[c,y]=0,\\
 [d_i,x]=[e_i,y]=0\qquad(i\in\mathcal I),\\
 [d_i,z_j]=[e_i,z_j]=0\qquad(i,j\in\mathcal I),
\end{gathered}
\end{equation}
together with
\begin{equation}\label{eq:s2:root-crossings}
\begin{gathered}
 xy=cyx,\qquad xz_i=d_i z_i x,\qquad
 yz_i=e_i z_i y\qquad(i\in\mathcal I),\\
 g^5=1\qquad(g\in\mathcal B).
\end{gathered}
\end{equation}
There are $35$ commutations, seven crossings, and five power
relations, giving $47$ equations.

The relations among the integer actions follow from
elementary matrix identities in $\operatorname{SL}_3(\mathbb Z)$,
since their group values are
\[
 U_{ij,0}=(I_3,I_3+E_{ij}).
\]
For example,
\[
 (I_3+E_{12})(I_3+E_{23})
 =(I_3+E_{13})(I_3+E_{23})(I_3+E_{12})
\]
gives $U_{12,0}U_{23,0}=U_{13,0}U_{23,0}U_{12,0}$.
Matrices in the same row or column commute, and multiplication
in the corresponding two-coordinate plane gives the braid
identities. We therefore impose
\begin{equation}\label{eq:s2:integer-relations}
\begin{aligned}
 U_{ij}U_{ik}&=U_{ik}U_{ij},&
 U_{ji}U_{ki}&=U_{ki}U_{ji}
 &&(j<k,\ \{i,j,k\}=\mathcal I),\\
 U_{ij}U_{jk}&=U_{ik}U_{jk}U_{ij}
 &&&&(i,j,k\text{ distinct}),\\
 U_{ij}U_{ji}^*U_{ij}&=U_{ji}^*U_{ij}U_{ji}^*
 &&&&(i<j).
\end{aligned}
\end{equation}
These are six commutations, six crossings, and three braid
relations. These fifteen integer equations are used by the
spectral certificate; they are distinct from the fifteen
retained relations defining the polynomial penalties.

The compression relations follow from the action of $U_{ij,0}$
on the coefficient variables:
\[
\xi_j\longmapsto\xi_i\xi_j,\qquad
\xi_k\longmapsto\xi_k\quad(k\ne j).
\]
It therefore fixes $x_0,y_0,z_{k,0}$ for $k\ne j$, while
\[
 U_{ij,0}z_{j,0}U_{ij,0}^{-1}
 =e_{31}(\xi_i\xi_j).
\]
By \eqref{eq:s2:group-targets}, for $i<j$ this target equals
\[
 e_{31}(\xi_i\xi_j)
 =d_{i,0}^{-1}e_{j,0}d_{i,0}e_{j,0}^{-1}.
\]
This explains the compression target words
\begin{equation}\label{eq:s2:full-targets}
 \pi_{ij}=d_i^*e_jd_i e_j^*\qquad(i<j).
\end{equation}

The centralization relations come from the definition
\[
 \Gamma=G*_H\bigl(H\times\langle v\mid v^2=1\rangle\bigr):
\]
the involution $v$ commutes with every element of $H$,
including the five selected elements.
Thus the compression and centralization equations are
\begin{equation}\label{eq:s2:compression-relations}
\begin{gathered}
 {}[U_{ij},x]=[U_{ij},y]=0\qquad(i\ne j),\\
 [U_{ij},z_k]=0\qquad(i\ne j,\ k\ne j),\\
 U_{ij}z_j=\pi_{\min(i,j),\max(i,j)}U_{ij}
 \qquad(i\ne j),\\
 [v,g]=0\qquad(g\in\mathcal B).
\end{gathered}
\end{equation}
For each of the six actions, there are five compression
equations, giving thirty in total. Together with the five
centralizations, the full system consists of
\[
 47+15+30+5=97
\]
equations.

For an exact unitary matrix solution, let
$\mathcal A\subseteq M_d(\mathbb C)$ be the unital
$*$-algebra generated by $\mathcal B$.
These equations are satisfied by the group model and are
selected to support the analytic estimates. They are not
asserted to present the whole group.

Their effect on an exact matrix solution is elementary.
Let $\mathcal A\subseteq M_d(\mathbb C)$ be the unital
$*$-algebra generated by $\mathcal B$.
The crossing equations express $c,d_i,e_i$ as words in $\mathcal B$,
so the targets $\pi_{ij}$ also belong to $\mathcal A$.
The compression equations imply
\[
 U_{12}\mathcal A U_{12}^*\subseteq\mathcal A.
\]
Both algebras have the same finite dimension, so equality holds.
Since $v$ commutes with $\mathcal A$, it commutes with
$U_{12}^*z_2U_{12}$. Consequently,
\[
 [z_2,U_{12}vU_{12}^*]=0.
\]

Sections~\ref{sec:certificates}--\ref{sec:gate} establish a
quantitative version of this conclusion for approximate
solutions, with bounds independent of the matrix dimension.
We next explain how symmetries of the group model allow the
required full relations to be recovered from fifteen relations
encoded by six involutions.

\subsection{Six involutions and the retained relations}\label{subsec:coordinates}

We first realize the required symmetries in a group and explain the six
coordinates. We then introduce the corresponding free variables.

To construct the cyclic permutation and reflection, let $J_0$
be the antidiagonal permutation matrix and set
$D_{\mathrm{sign}}=\operatorname{diag}(1,-1,1)$.
 The automorphism
\[
\vartheta(X)=D_{\mathrm{sign}}J_0(X^{-1})^T J_0D_{\mathrm{sign}}
\]
is an involution. It preserves $H$ and the Laurent elementary group,
and commutes with substitutions of the Laurent variables. Enlarge the
group model to
\begin{equation}\label{eq:s2:enlarged-groups}
\begin{aligned}
 G'&=\Bigl(\operatorname{EL}_3(
 \mathbb F_5[\xi_1^{\pm1},\xi_2^{\pm1},\xi_3^{\pm1}])
 \rtimes\operatorname{GL}_3(\mathbb Z)\Bigr)\rtimes\langle\vartheta\rangle,\\
 H'&=(H\rtimes S_3)\rtimes\langle\vartheta\rangle,\\
 \Gamma'&=G'*_{H'}\bigl(H'\times\langle v\mid v^2=1\rangle\bigr).
\end{aligned}
\end{equation}
Here $S_3$, the permutation group of $\mathcal I$, is represented by permutation
matrices acting on the exponent columns. The involution $v$ centralizes
$H'$.

Write $P_{(123)}$ and $P_{(23)}$ for these permutation matrices, and set
\begin{equation}\label{eq:s2:group-symmetries}
 z_0=z_{1,0},\qquad a_0=U_{12,0},\qquad
 q_0=\vartheta P_{(23)},\qquad t_0=P_{(123)}^2v.
\end{equation}
Then
\[
 q_0x_0q_0=y_0,\qquad q_0z_0q_0=z_0^{-1},\qquad
 q_0t_0q_0=t_0^{-1},
\]
and
\[
 t_0^2=P_{(123)},\qquad t_0^3=v,\qquad t_0^6=1.
\]
Thus one element $t_0$ encodes both the cyclic permutation of the
three root indices and the centralizing involution.

The negative signal persists in this enlarged model. Indeed,
\[
 a_0^{-1}z_{2,0}a_0=e_{31}(\xi_1^{-1}\xi_2)\notin H',
\]
because its action and $\vartheta$ components are trivial but its
matrix entry is not polynomial. With $W_0=a_0va_0^{-1}$, the word
$z_{2,0}^{-1}W_0z_{2,0}W_0$ is conjugate to a reduced alternating
word in $\Gamma'$, so it is nonidentity by the normal-form theorem
for amalgamated products~\cite[Chapter~I, \S1]{Serre}.
This verifies the normal-form
obstruction directly in $\Gamma'$; no embedding of $\Gamma$ into
$\Gamma'$ is required.

We next construct an involution that interchanges $x_0$ and $z_0$.
Set $H_0=\langle x_0,y_0,z_0\rangle$.
The matrix
\[
 P_{\mathrm{swap}}=
 \begin{pmatrix}
 1&0&0\\
 0&0&-\xi_1\\
 0&-\xi_1&0
 \end{pmatrix}
\]
defines an inverse-transpose automorphism
$X\mapsto P_{\mathrm{swap}}(X^{-1})^T P_{\mathrm{swap}}^{-1}$
of $H_0$. Together with the restriction of $\vartheta$, its action is
\[
\begin{array}{c|ccc}
 &x_0&y_0&z_0\\ \hline
 \text{first transposition}&y_0&x_0&z_0^{-1}\\
 \text{second transposition}&z_0&y_0^{-1}&x_0
\end{array}
\]
Both maps are involutions, and
\[
 \bigl(D_{\mathrm{sign}}J_0P_{\mathrm{swap}}^{-T}\bigr)^3
 =-\xi_1^{-2}I_3.
\]
Their product therefore has cube equal to the identity as an
automorphism, giving an $S_3$-action on $H_0$.
Conjugation by $q_0$ is the first transposition, since $P_{(23)}$
fixes $\xi_1$. Moreover $q_0\notin H_0$, so
$\langle H_0,q_0\rangle$ is the corresponding semidirect product
with a group of order two. Form
\begin{equation}\label{eq:s2:Lambda}
 \Lambda=\Gamma'*_{\langle H_0,q_0\rangle}(H_0\rtimes S_3),
\end{equation}
identifying $q_0$ with the first transposition in the second factor.
Let $s_0$ be the second transposition. Then
\[
 s_0z_0s_0=x_0,\qquad (q_0s_0)^3=1.
\]
The normal-form theorem~\cite[Chapter~I, \S1]{Serre}
gives an injection $\Gamma'\hookrightarrow
\Lambda$, so the nonidentity signal word remains nonidentity.

We now encode the required group elements using six involutions.
Put $D_a=(I_3,\allowbreak\operatorname{diag}(1,-1,-1))$.
This involution inverts
$a_0$, while $q_0$ inverts both $t_0$ and $z_0$. Consequently
\begin{equation}\label{eq:s2:involution-tuple}
 (a_1,\ldots,a_6)
 =(s_0,\ a_0D_a,\ D_a,\ t_0q_0,\ z_0q_0,\ q_0)
\end{equation}
is a tuple of six involutions. The two coordinates $a_0D_a,D_a$
recover $a_0$ by multiplication; the three coordinates
$t_0q_0,z_0q_0,q_0$ recover $t_0,z_0$ using the common reflection
$q_0$; and $s_0$ recovers $x_0$ from $z_0$ by conjugation.

We now introduce six free selfadjoint variables $X_1,\ldots,X_6$,
corresponding in order to \eqref{eq:s2:involution-tuple}, and define
the word abbreviations
\begin{equation}\label{csimp:r3:readout}
 s=X_1,\quad q=X_6,\quad a=X_2X_3,\quad t=X_4X_6,\quad
 z=X_5X_6,\quad x=szs.
\end{equation}
At the displayed group tuple these words recover
$s_0,q_0,a_0,t_0,z_0,x_0$, respectively. No additional free variables
are introduced by the following dictionary:
\begin{equation}\label{csimp:r17:eq:dict}
\begin{gathered}
 r=t^2,\quad v=t^3,\quad z_1=z,\quad z_2=rzr^*,\quad z_3=r^*zr,\\
 a_+=rar^*,\quad b=qaq,\quad d=\gc{x}{z},\quad \ell=rq,\\
 p=d^*(\ell d^*\ell)d(\ell d^*\ell)^*,\quad W=ava^*.
\end{gathered}
\end{equation}

In the group model, conjugation by $r$ cyclically permutes
$z_1,z_2,z_3$, and $a,b,a_+$ have values
$U_{12,0},U_{13,0},U_{23,0}$, respectively.
The word $\ell$ has value $\vartheta P_{(12)}$, and
\[
 d\longmapsto e_{32}(-\xi_1),\qquad
 \ell d^*\ell\longmapsto e_{21}(\xi_2),\qquad
 p\longmapsto e_{31}(\xi_1\xi_2).
\]
Thus $p$ is the target of the compression by $a$.
For approximate matrix solutions, its comparison with the full target
$\pi_{12}$ is proved in Section~\ref{sec:gate}; it is not a word identity.

For any substitution of selfadjoint unitaries for the $X_j$, the
words $s,q,\ell$ are symmetries, and $qtq=t^*$, $qzq=z^*$.
If also $t^6=1$, then $r^3=v^2=1$ and $W$ is a symmetry.
These identities are not assumed for arbitrary matrix contraction
inputs of the final polynomial.

The fifteen retained relations, corresponding in order to the
penalty words $u_1,\ldots,u_{15}$ defined below, are
\begin{equation}\label{csimp:r3:fifteen}
\begin{gathered}
 t^6=1,\quad(qs)^3=1,\quad [t,x]=[a,x]=0,\quad[z,z_2]=0,\\
 [d,x]=[d,z_2]=0,\quad z^5=1,\quad[a,b]=[b,a_+]=0,\\
 aa_+=ba_+a,\quad[a,z]=[a,z_3]=0,\quad az_2a^*=p,\quad[v,z]=0.
\end{gathered}
\end{equation}
Their origin is now explicit. The first two relations describe the
finite symmetries. The commutation relations among the root elements and fifth power are elementary
matrix identities over $\mathbb F_5$. The relations
$[a,b]=[b,a_+]=0$ and $aa_+=ba_+a$ are one row commutation,
one column commutation, and one crossing among
$U_{12,0},U_{13,0},U_{23,0}$. The three relations involving the
action of $a$ on $z,z_2,z_3$ express that $U_{12,0}$ fixes
$\xi_1,\xi_3$ and sends $\xi_2$ to $\xi_1\xi_2$.
Finally, $v$ centralizes the basic roots; $[t,x]=[a,x]=0$ records
that the indicated symmetries and action fix the constant root $x_0$.
Every relation in \eqref{csimp:r3:fifteen} holds at the group tuple
\eqref{eq:s2:involution-tuple}.

The connection with the full system is supplied by the additional
word assignments
\begin{equation}\label{eq:s2:full-readouts}
\begin{gathered}
 y=qxq,\quad c=\gc{x}{y},\quad
 d_i=\gc{x}{z_i},\quad e_i=\gc{y}{z_i}\quad(i\in\mathcal I),\\
 (U_{12},U_{13},U_{23},U_{31},U_{21},U_{32})
 =(a,b,rar^*,r^*ar,rbr^*,r^*br).
\end{gathered}
\end{equation}
In particular, the five entries of $\mathcal B$ are now words in the six
variables, and the seven root crossings hold identically on unitary
substitutions. The other full relations are recovered by the finite
calculations in Section~\ref{sec:gate}. More precisely, on a selfadjoint
unitary tuple with $t^6=(qs)^3=1$ exactly, errors at most $\varepsilon$
in the retained relations give errors at most $3\cdot2^{13}\varepsilon$
in all 97 full relations. The proof passes through an intermediate
system of seventeen equations and compares the compression targets.
The same section supplies the rounding needed to make the two
finite-order relations exact. This separates the purpose of the two
systems: the 97 relations support the analytic estimates, while the
fifteen retained relations determine the polynomial penalties.

\subsection{The explicit polynomial and its negative evaluation}
\label{subsec:explicitpolynomial}

We work in $\mathbb Z\langle X_1,\ldots,X_6\rangle$, with $X_j^*=X_j$.
The signal is motivated by the group word $z_2^*Wz_2W$; the penalty
words come from the fifteen retained relations.
Write the $i$th relation in \eqref{csimp:r3:fifteen} as $L_i=R_i$,
expressing each commutation as an equality of products.
In $z_2^*Wz_2W$ and the relators $L_iR_i^*$, substitute
\eqref{csimp:r3:readout}--\eqref{csimp:r17:eq:dict}, expand powers
and adjoints, and repeatedly delete adjacent pairs $X_jX_j$.
This gives the signal word $\omega$. For each relator, further remove
matching first and last letters repeatedly to obtain $u_i$; no cyclic
rotation is made. For example, the crossing $aa_+=ba_+a$ gives
$\gc{a}{a_+}b^*$, and the compression $az_2a^*=p$ gives $az_2a^*p^*$.
These operations specify literal words, rather than identities in the
free algebra.

The explicit signal word is
\begin{equation}\label{eq:signalword}
\begin{aligned}
\omega
 &=X_4X_6X_4X_5(X_6X_4)^2
   X_2X_3(X_4X_6)^3X_3X_2\\
 &\quad\cdot
   (X_4X_6)^2X_5X_4X_6X_4
   X_2X_3(X_4X_6)^3X_3X_2.
\end{aligned}
\end{equation}
The penalty words are listed below; powers of parenthesized words
denote repeated concatenation.
\begingroup
\allowdisplaybreaks[1]
\begin{align}
u_1&=(X_4X_6)^6,\notag\\
u_2&=(X_6X_1)^3,\notag\\
u_3&=X_4X_6X_1X_5X_6X_1X_6X_4X_1X_6X_5X_1,\notag\\
u_4&=X_2X_3X_1X_5X_6X_1X_3X_2X_1X_6X_5X_1,\notag\\
u_5&=X_5(X_6X_4)^2(X_6X_5X_4X_6X_4)^2
      X_5(X_6X_4)^2,\notag\\
u_6&=X_1X_5X_6(X_1X_6X_5)^2
      (X_1X_5X_6)^3(X_1X_6X_5)^2,\notag\\
u_7&=(X_1X_5X_6)^2(X_1X_6X_5)^2
      (X_4X_6)^2X_5X_4X_6X_4\notag\\
   &\quad\cdot
      (X_5X_6X_1)^2(X_6X_5X_1)^2
      X_4X_6X_4X_5(X_6X_4)^2,\notag\\
u_8&=(X_5X_6)^5,\notag\\
u_9&=(X_2X_3X_6)^2(X_3X_2X_6)^2,\notag\\
u_{10}
   &=\bigl(X_6X_2X_3(X_6X_4)^2\bigr)^2
     \bigl(X_6X_3X_2(X_6X_4)^2\bigr)^2,\notag\\
u_{11}
   &=X_2X_3(X_4X_6)^2X_2X_3(X_6X_4)^2X_3X_2
     \bigl((X_4X_6)^2X_3X_2X_6\bigr)^2,\notag\\
u_{12}
   &=X_2X_3X_5X_6X_3X_2X_6X_5,\notag\\
u_{13}
   &=X_2X_3(X_6X_4)^2X_5(X_6X_4)^2X_6X_3X_2
     (X_6X_4)^2X_6X_5(X_4X_6)^2,\notag\\
u_{14}
   &=X_2X_3(X_4X_6)^2X_5X_4X_6X_4X_3X_2
     \bigl(X_4X_6X_4(X_5X_6X_1)^2
                       (X_6X_5X_1)^2\bigr)^2\notag\\
   &\quad\cdot
     \bigl(X_4X_6X_4(X_1X_5X_6)^2
                       (X_1X_6X_5)^2\bigr)^2,\notag\\
u_{15}
   &=\bigl((X_4X_6)^3X_5\bigr)^2.\label{eq:penaltywords}
\end{align}
\endgroup

On a tuple $A$ of selfadjoint unitary matrices, adjacent equal letters
cancel, while removing matching endpoints conjugates the evaluated
word by a unitary. Consequently,
\[
 \omega(A)=z_2(A)^*W(A)z_2(A)W(A),\qquad
 \|1-u_i(A)\|_2=\|L_i(A)-R_i(A)\|_2\quad(1\le i\le15).
\]
On arbitrary selfadjoint matrix inputs, the expressions
\eqref{eq:signalword}--\eqref{eq:penaltywords} are evaluated as written,
without assuming $X_j^2=1$. Their adjoints are obtained by reversing
the letters.

Set
\begin{equation}\label{eq:globalconstants}
 J=2^{55226},\qquad M=2^{2J+54},
\end{equation}
and define
\begin{equation}\label{eq:nativefinal}
 f=-1+\omega+\omega^*+M\sum_{j=1}^{15}(2-u_j-u_j^*).
\end{equation}
Each penalty has nonnegative trace on selfadjoint matrix contraction
tuples, since every evaluated word is a contraction.
The 32 words $\omega,\omega^*,u_1,u_1^*,\ldots,u_{15},u_{15}^*$
are pairwise distinct. Each signal word has length~36, and the
penalty words have maximum length~72 and collectively contain all
six variables. Thus $f$ is Hermitian with integer coefficients,
degree~72, and exactly 33 monomials, including the constant $30M-1$.
Its total expanded letter count is~716. The complete expansion appears
in Appendix~\ref{app:nativeexpanded}.

\begin{proposition}\label{prop:s2:negative-evaluation}
The group involutions $a_1,\ldots,a_6$ in
\eqref{eq:s2:involution-tuple}, regarded as selfadjoint unitaries in
$L(\Lambda)$, satisfy
\[
 u_i(a_1,\ldots,a_6)=1\quad(1\le i\le15),\qquad
 \tau_\Lambda f(a_1,\ldots,a_6)=-1.
\]
\end{proposition}
\begin{proof}
The readout words at this tuple have the group values constructed
above and satisfy all fifteen retained relations. Adjacent deletions
preserve their values, and endpoint deletions preserve the identity
value of a relator. Thus every penalty vanishes.
The signal word evaluates to $z_{2,0}^{-1}W_0z_{2,0}W_0$,
which is nonidentity in $\Gamma'$ and remains nonidentity under
$\Gamma'\hookrightarrow\Lambda$. Its canonical trace and that of
its inverse are zero. Substituting into \eqref{eq:nativefinal}
gives the asserted trace $-1$.
\end{proof}

The rest of the main text proves the uniform matrix trace bound.
The full relations provide the spectral and normalization estimates;
the recovery argument transfers these estimates to the retained
relations, and rounding connects their defects to the polynomial
penalties on arbitrary selfadjoint contractions.

\section{Spectral certificates and correction}\label{sec:certificates}

We now pass from the group relations of Section~\ref{sec:polynomial}
to unitary matrices satisfying those relations approximately.
The change of viewpoint is to study matrices that nearly commute
with the given inputs. Matrices satisfying a fixed commutator-error
bound need not form an algebra. Our aim is to approximate them by
an actual unital $*$-subalgebra whose entire operator-norm unit ball
nearly commutes with the inputs, with estimates independent of the
matrix dimension.

Following \cite[Section~2.1, (2.1)]{ALT}, for integers $d,m\ge1$ and
unitaries $v_1,\ldots,v_m\in\mathcal U(d)$, define the averaging
operator $T$ and its commutator energy by
\begin{equation}\label{eq:energy}
 T(x)=\frac{x}{2}+\frac1{4m}\sum_{j=1}^m
 (v_jxv_j^*+v_j^*xv_j),\qquad
 \cE(x)=\frac1{4m}\sum_{j=1}^m\|[v_j,x]\|_2^2.
\end{equation}
Thus $\cE(x)=0$ precisely when $x$ commutes with every $v_j$,
and small energy measures approximate commutation.
The map $T$ is unital, completely positive, trace preserving, and
selfadjoint on $L^2(M_d,\tr_d)$. The term $x/2$ ensures that its
spectrum lies in $[0,1]$. Writing $\mathcal L=I-T$, where $I$ is
the identity map, we have
\begin{equation}\label{eq:energyid}
 \cE(x)=\ip{\mathcal Lx}{x},\qquad
 \cE(x)\le\|x\|_2^2.
\end{equation}
In particular, the fixed points of $T$ form the commutant of the
input tuple. The seminorm $\sqrt{\cE}$ is subadditive and satisfies
\[
 \sqrt{\cE(xy)}
 \le\sqrt{\cE(x)}+\sqrt{\cE(y)}
 \qquad(\|x\|,\|y\|\le1).
\]
Conditional expectations and induced map norms are as in
Section~\ref{sec:prelim}, particularly \eqref{eq:mapnorms}.

The spectral input to our argument is an inequality of the form
\[
 \ip{(\mathcal L^2-\lambda\mathcal L)y}{y}
 \ge-C\eps
 \qquad(\|y\|\le1),
\]
where $0<\lambda<1$ and $C\ge0$ are fixed rational parameters,
and $\eps\ge0$ bounds the relevant relation defects.
When $\eps=0$, this inequality excludes spectrum in $(0,\lambda)$.
For positive $\eps$, we use the stated quadratic-form bound on
the operator-norm unit ball; an exact spectral gap for the
original tuple is not assumed.

Our main conclusion is that this inequality allows us to
reconstruct an approximate commutant. For fixed $m,\lambda,C$
and any prescribed positive tolerance, sufficiently small $\eps$
yields a unital $*$-subalgebra $A\subseteq M_d$ with two properties.
For every contraction $x$, the distance $\|x-E_Ax\|_2$ is bounded
by a constant multiple of $\sqrt{\cE(x)}$ plus the prescribed
tolerance. Every contraction in $A$ also commutes with each
$v_j$ up to that tolerance. Here $E_A$ is the trace-preserving
conditional expectation onto $A$. Both the required smallness
of $\eps$ and the resulting bounds are independent of $d$.

Sections~\ref{sec:field} and~\ref{subsec:K} establish the spectral
inequality for two input lists. For the basic elements
$x,y,z_1,z_2,z_3$, we use the ordered 375-entry word list
$\mathscr L_H$ defined in \eqref{eq:Hlist}. It comes from the
three finite subgroups in the group model, with repetitions
chosen to give each subgroup equal weight.
For the integer actions, we use the six unitary entries $U_{ab}$
from Section~\ref{subsec:fullrelations}. Their energies are
\begin{equation}\label{eq:specialenergies}
 \mathcal E_H(x)
 =\frac1{4\cdot375}\sum_{w\in\mathscr L_H}\|[w,x]\|_2^2,
 \qquad
 \mathcal E_K(x)
 =\frac1{24}\sum_{a\ne b}\|[U_{ab},x]\|_2^2.
\end{equation}
Each word in $\mathscr L_H$ is evaluated at the given unitary
entries; these matrices need not generate finite groups.
The finite-field relations and the integer-action certificate
give, respectively, the parameters
\begin{equation}\label{eq:actualinputs}
 (m,\lambda,C)=(375,1/96,2^{80})
 \quad\text{and}\quad
 (6,1/144,26).
\end{equation}

Section~\ref{subsec:spectralcorrection} uses these spectral bounds
to reduce commutator energy by taking powers of $T$, correct
the unitary tuple, and construct an approximately idempotent
averaging map. Section~\ref{sec:algebra} reconstructs an actual
unital $*$-subalgebra from this map and combines the estimates
to prove Theorem~\ref{thm:interface}.

Section~\ref{sec:reconstruction} then proves a general normalization
theorem for algebras satisfying these approximate-commutant
estimates. In Section~\ref{sec:gate}, we apply the construction
separately to the basic elements and the integer actions.
The two resulting algebras, together with the compression relations,
satisfy the hypotheses of the normalization theorem. Its conclusion
and the centralization relations yield the commutator estimate
needed for matrix positivity.

\subsection{The spectral estimate over \texorpdfstring{$\F_5$}{F5}}\label{sec:field}
The subgroups
\[
 \langle x_0\rangle,\qquad
 \langle y_0\rangle,\qquad
 \langle z_{1,0},z_{2,0},z_{3,0}\rangle
\]
generated by the fixed elements in \eqref{eq:s2:basic-group-roots}
have orders $5,5,125$, respectively.
Their normal forms determine the following ordered list of words
in the corresponding unitary entries:
\begin{equation}\label{eq:Hlist}
\begin{aligned}
 \mathscr L_H=\bigl(
 &(x^a)_{0\le a\le4}\text{ repeated }25\text{ times};\\
 &(y^a)_{0\le a\le4}\text{ repeated }25\text{ times};\\
 &(z_1^az_2^bz_3^c)_{0\le a,b,c\le4}\bigr).
\end{aligned}
\end{equation}
The lists are concatenated in the displayed order, with exponents
in increasing lexicographic order. Each subgroup contributes
125 entries and therefore has the same weight in the average.
The resulting list has 375 entries of length at most twelve
in $x,y,z_1,z_2,z_3$. It is evaluated at arbitrary unitary entries,
which need not generate finite groups.
To transfer the finite-group spectral estimate to these substitutions,
we first control the fifth powers of the seven commutator entries.

For a unitary $u\in M_d$, write $\Ad(u)(x)=uxu^*$ for
conjugation by $u$. This map preserves $\|\cdot\|_2$.
Its powers denote repeated composition:
\[
 \Ad(u)^k(x)=u^kx(u^*)^k
 \qquad(x\in M_d,\ k\ge1).
\]

\begin{lemma}\label{lem:centralpowers}
Suppose the twelve entries $x,y,z_i,c,d_i,e_i$ $(1\le i\le3)$
are unitaries and the relations
\eqref{eq:s2:root-commutations}-\eqref{eq:s2:root-crossings}
have defects of at most $\varepsilon\ge0$ in $\|\cdot\|_2$. Then
\[
 \|q^5-1\|_2\le17\varepsilon
 \qquad(q\in\{c,d_1,d_2,d_3,e_1,e_2,e_3\}).
\]
\end{lemma}
\begin{proof}
For $q=c,d_i,e_i$, choose $(a,b)=(x,y),(x,z_i),(y,z_i)$,
respectively. The corresponding relations give
\[
 \|\Ad(a)(b)-qb\|_2\le\varepsilon,\qquad
 \|\Ad(a)(q)-q\|_2\le\varepsilon,\qquad
 \|a^5-1\|_2\le\varepsilon.
\]
Apply $\Ad(a)$ five times, replacing $\Ad(a)(b)$ by $qb$
and each factor $\Ad(a)(q)$ by $q$. Since conjugation
preserves $\|\cdot\|_2$ and all factors are unitary, the
additional errors are at most
$\varepsilon,2\varepsilon,\ldots,5\varepsilon$. Hence
\[
 \|\Ad(a)^5(b)-q^5b\|_2
 \le(1+2+3+4+5)\varepsilon
 =15\varepsilon.
\]
Moreover,
\[
 \|\Ad(a)^5(b)-b\|_2
 =\|a^5b(a^*)^5-b\|_2
 \le2\|a^5-1\|_2
 \le2\varepsilon.
\]
Since $b$ is unitary, the triangle inequality now gives
\[
 \|q^5-1\|_2=\|q^5b-b\|_2
 \le15\varepsilon+2\varepsilon=17\varepsilon.\qedhere
\]
\end{proof}

\begin{proposition}\label{prop:Htransfer}
Suppose the twelve entries are unitaries and the 47 equations
\eqref{eq:s2:root-commutations}-\eqref{eq:s2:root-crossings}
have normalized $2$-norm defects of at most $\varepsilon$. For the
Laplacian $\mathcal L=I-T$ associated with $\mathscr L_H$, we have
\begin{equation}\label{eq:Hquadratic}
 \langle(\mathcal L^2-\mathcal L/96)y,y\rangle
 \ge-2^{80}\varepsilon\qquad(\|y\|\le1).
\end{equation}
\end{proposition}

\begin{proof}
We first work with the exact group elements. The indices $1,2,3$
refer to the subgroups
\[
 \langle x_0\rangle,\qquad
 \langle y_0\rangle,\qquad
 \langle z_{1,0},z_{2,0},z_{3,0}\rangle,
\]
respectively. Their normalized averages in $\mathbb C[H]$ are
\[
 P_1=\frac15\sum_{a=0}^4x_0^a,\qquad
 P_2=\frac15\sum_{a=0}^4y_0^a,\qquad
 P_3=\frac1{125}\sum_{a,b,c=0}^4z_{1,0}^az_{2,0}^bz_{3,0}^c.
\]
For $1\le i<j\le3$, let $F_{ij}$ be the subgroup generated
by the $i$th and $j$th subgroups. Explicitly,
\[
\begin{aligned}
 F_{12}&=\langle x_0,y_0\rangle,\\
 F_{13}&=\langle x_0,z_{1,0},z_{2,0},z_{3,0}\rangle,\\
 F_{23}&=\langle y_0,z_{1,0},z_{2,0},z_{3,0}\rangle.
\end{aligned}
\]
There are only three unordered pairs of distinct subgroups;
reversing their order gives the same generated group.
For each pair, $P_i,P_j\in\mathbb C[F_{ij}]$, we establish
the spectral estimate for $2-P_i-P_j$ in this finite group algebra.
The commutator $c_0$ is central in $F_{12}$, the elements
$d_{1,0},d_{2,0},d_{3,0}$ are central in $F_{13}$, and
$e_{1,0},e_{2,0},e_{3,0}$ are central in $F_{23}$.
Thus, every element of $F_{ij}$ can be written with powers
from the first subgroup, followed by powers from the second
subgroup, and then the central commutators.
All these generators have order five, so their exponents
can be chosen in $\{0,1,2,3,4\}$.
There are at most seven generators in these normal forms,
giving $|F_{ij}|\le5^7$ and a maximum word length of
$4\cdot7=28$.

Regard $\mathbb C[F_{ij}]$ as acting by left translation on
functions on $F_{ij}$. The averages $P_i,P_j$ are orthogonal
projections. The representation calculation in
\cite[Section~4]{EJ} gives
\[
 \spec(2-P_i-P_j)
 \subseteq\{0,1,2,1-5^{-1/2},1+5^{-1/2}\},
\]
where $\spec$ denotes the set of eigenvalues in this representation.
Since $17/32<1-5^{-1/2}$, the element
\[
 R_{ij}=(2-P_i-P_j)^2-\frac{17}{32}(2-P_i-P_j)
\]
is positive and has operator norm less than $3$.
Its positive square root belongs to $\mathbb C[F_{ij}]$.
Writing $\|\cdot\|_{\ell^1}$ for the sum of the absolute values
of the group coefficients, Cauchy's inequality gives
\begin{equation}\label{eq:Hsquarebudget}
 \|R_{ij}^{1/2}\|_{\ell^1}^2
 \le |F_{ij}|\,\|R_{ij}\|
 \le3\cdot5^7.
\end{equation}

We next transfer this identity to the given unitary matrices.
A word of length $n$, allowing inverses, can be reduced to normal
form using at most $2^{20}(n+1)^4$ of the stated relations and the
derived fifth power relations. Indeed, replacing inverse letters
by fourth powers leaves at most $4n$ letters. At most $16n^2$
crossings insert at most $64n^2$ central letters. Ordering these
letters and reducing their powers modulo five requires fewer
than the stated number of replacements.
Each replacement has defect at most $17\varepsilon$ by
Lemma~\ref{lem:centralpowers} and the hypotheses.
Multiplication by unitary factors preserves these error bounds.

For each $g\in F_{ij}$, let $w_g$ denote its normal word evaluated
at the given unitary entries. Define
\[
 \Phi(g)=\frac12\bigl(\Ad(w_g)+\Ad(w_{g^{-1}}^*)\bigr)
\]
and extend linearly to $\mathbb C[F_{ij}]$.
This map preserves adjoints, where the adjoint of an operator
on matrices is taken in $L^2(M_d,\tr_d)$.
Each $\Phi(g)$ is contractive in both the operator norm and
$\|\cdot\|_2$.

Products of two normal words have length at most $56$.
The preceding reduction bound, together with
$\|\Ad(u)(Y)-\Ad(v)(Y)\|_2\le2\|u-v\|_2$ for unitaries $u,v$
and contractions $Y$, gives
\begin{equation}\label{eq:phidefect}
 \|(\Phi(g)\Phi(h)-\Phi(gh))Y\|_2
 \le 2\cdot17\cdot2^{20}(57^4+29^4)\varepsilon
 \le2^{52}\varepsilon
 \qquad(\|Y\|\le1).
\end{equation}
The normal words within each original subgroup are chosen
consistently across the two groups containing it.
Thus $Q_i=\Phi(P_i)$ is independent of that choice of group.
The repetitions in \eqref{eq:Hlist} give
\[
 \mathcal L=\frac{3I-Q_1-Q_2-Q_3}{6},
\]
where $I$ is the identity map on $M_d$.

Extend \eqref{eq:phidefect} by linearity to group algebra elements.
Since
$\Phi(R_{ij}^{1/2})^*\Phi(R_{ij}^{1/2})$ is positive,
\eqref{eq:Hsquarebudget} controls the error in transferring
$R_{ij}=(R_{ij}^{1/2})^*R_{ij}^{1/2}$.
Transferring $(2-P_i-P_j)^2$ contributes at most
$16\cdot2^{52}\varepsilon$, because
$\|2-P_i-P_j\|_{\ell^1}\le4$. Consequently,
\[
 \left\langle
 \left((2I-Q_i-Q_j)^2-\frac{17}{32}(2I-Q_i-Q_j)\right)Y,Y
 \right\rangle
 \ge-2^{52}(16+3\cdot5^7)\varepsilon
 \qquad(\|Y\|\le1).
\]
Similarly, $P_i^2=P_i$ and $\|P_i\|_{\ell^1}=1$ give
\[
 \|(Q_i^2-Q_i)Y\|_2\le2^{52}\varepsilon
 \qquad(\|Y\|\le1).
\]

Finally, direct expansion gives the identity
\begin{equation}\label{eq:Hidentity}
\begin{aligned}
 36(\mathcal L^2-\mathcal L/96)
 ={}&\sum_{i<j}
 \left((2I-Q_i-Q_j)^2-\frac{17}{32}(2I-Q_i-Q_j)\right)\\
 &-\sum_{i=1}^3(Q_i^2-Q_i).
\end{aligned}
\end{equation}
Summing the preceding bounds and dividing by $36$, we obtain
\[
 \left\langle(\mathcal L^2-\mathcal L/96)Y,Y\right\rangle
 \ge-\frac{2^{52}}{12}(17+3\cdot5^7)\varepsilon
 \ge-2^{80}\varepsilon
 \qquad(\|Y\|\le1),
\]
as required.
\end{proof}

\subsection{The integer certificate}\label{subsec:K}

This subsection establishes the spectral estimate for the six
unitary action entries $U_{ij}\in\mathcal U(d)$, $1\le i\ne j\le3$.
Let $T$ be their averaging operator from \eqref{eq:energy}, with
$m=6$, and put $\mathcal L=I-T$. We prove that if the fifteen
integer relations in \eqref{eq:s2:integer-relations} have defects
at most $\eps\ge0$ in $\|\cdot\|_2$, then
\begin{equation}\label{eq:Kquadratic}
 \ip{(\mathcal L^2-\mathcal L/144)Y}{Y}
 \ge-26\eps
 \qquad(Y\in M_d,\ \|Y\|\le1).
\end{equation}
This supplies the input $(m,\lambda,C)=(6,1/144,26)$ for
Section~\ref{subsec:spectralcorrection}. We obtain the estimate
from an exact certificate of Netzer and Thom, controlling the
error introduced when the group relations hold only approximately.

Set $K=\SL_3(\Z)$ and write
\[
 v_{ij}=I_3+E_{ij}\qquad(1\le i\ne j\le3),
\]
where $E_{ij}$ is the matrix unit. Order these six generators
as $12,13,21,23,31,32$. 
Let $F_6$ be the free group on six formal generators mapped
to the corresponding $v_{ij}$. For each equation $L=R$ in \eqref{eq:s2:integer-relations},
replace the entries $U_{ij}$ by the corresponding formal
generators of $F_6$ and form the word $LR^{-1}$.
Let $\mathcal S_{15}\subseteq F_6$ denote the resulting
set of fifteen relator words.  
For unitary evaluations of $L$ and $R$,
\[
 \|LR^*-1\|_2=\|L-R\|_2,
\]
so the equation and its relator have the same defect.
Section~\ref{sec:gate} recovers these integer relations from
the retained relations using Lemma~\ref{csimp:r17:lem:core},
with the three braid relations supplied by Lemma~\ref{lem:braid}.

In the rational group algebra $\Q[K]$, define
\[
 \Delta_K=12-\sum_{i\ne j}(v_{ij}+v_{ij}^{-1}),
 \qquad D_K=\Delta_K/24.
\]
The involution sends each group element to its inverse.
The normalization by $24$ makes $D_K$ correspond to
$\mathcal L$ under the conjugation action associated with
an exact unitary representation.

Ozawa~\cite[Main Theorem]{OzawaT} established the connection
between property $(T)$ and certificates consisting of
Hermitian squares. Netzer and Thom~\cite[Section~2 and
Lemma~2.1]{NT} gave such a certificate for
$\Delta_K^2-\Delta_K/6$.
After division by $24^2$, the rational realization used here is
\begin{equation}\label{eq:Kcertificate}
 D_K^2-\frac1{144}D_K
 =\sum_{\nu=1}^{18326}a_\nu q_\nu^*q_\nu
 \quad\text{in }\Q[K],\qquad
 a_\nu\in\Q_{>0},\quad q_\nu\in\Z[F_6].
\end{equation}
Each $q_\nu$ is a finite integer linear combination of words
in $F_6$, and its image in $\Q[K]$ is understood in this identity.

To use the certificate for approximate matrix solutions, we
must control the errors in identifying words that represent
the same element of $K$. The following lemma records the
required reductions using $\mathcal S_{15}$ and the coefficient
bounds that yield the error $26\eps$.

Expand both sides of \eqref{eq:Kcertificate} in freely reduced
words in $F_6$, and write their difference as
\[
 Z=\sum_u z_u u\in\Q[F_6],
\]
where $z_u$ is the coefficient of the word $u$.
For each element of $K$ represented in this expansion, collect
all words with nonzero coefficients representing that element
into a class $C$.
For each class, let $u_C\in C$ be the word selected in the
certificate data as its representative.
Thus every $u\in C$ represents the same element of $K$ as
$u_C$, and $u_C^{-1}u$ represents the identity in $K$.
All sums over $C$ below range over these classes.
\begin{lemma}\label{lem:Kshort}
Each comparison word $u_C^{-1}u$, $u\in C$, reduces to the
identity using at most thirteen relation replacements from
$\mathcal S_{15}$, allowing inverses. Moreover,
\begin{equation}\label{eq:Kbudget}
 C_K:=2\sum_C\sum_{u\in C\setminus\{u_C\}}|z_u|
 =\frac{238056328850089}{144000000000000}<2.
\end{equation}
Consequently, for any tuple $V$ of six unitaries in a finite
tracial algebra, with $w(V)$ denoting word evaluation,
\begin{equation}\label{eq:Karea}
 \max_C\max_{u\in C}\|(u_C^{-1}u)(V)-1\|_2
 \le13\max_{s\in\mathcal S_{15}}\|s(V)-1\|_2.
\end{equation}
\end{lemma}

\begin{proof}
The certificate data contain $9\,692$ nonzero terms in $2\,928$
classes, giving $6\,764$ nontrivial comparison words.
Under free cancellation, cyclic conjugation, and inversion,
these reduce to $945$ cases with the following numbers
of relation replacements:
\[
\begin{array}{c|rrrrrrrrrrrrr}
 \text{replacements}&1&2&3&4&5&6&7&8&9&10&11&12&13\\\hline
 \text{cases}&15&165&183&195&148&109&78&22&17&8&2&1&2.
\end{array}
\]
The coefficient calculation gives \eqref{eq:Kbudget}.
The rational identity, the zero-sum classes, the coefficient budget,
and each relation replacement have been checked using exact arithmetic.
The certificate data, reduction paths, and verification scripts are
available from the authors upon request.

On unitary evaluations, free cancellation is exact, while
conjugation and inversion preserve the norm of a relator defect.
Each relation replacement therefore contributes at most
$\max_{s\in\mathcal S_{15}}\|s(V)-1\|_2$.
Telescoping over at most thirteen replacements proves
\eqref{eq:Karea}.
\end{proof}

We now apply the lemma to six unitary matrices $V$ whose
relator defects are at most $\eps\ge0$.
Since $Z$ maps to zero in $\Q[K]$, its coefficients sum to zero
in each class. Hence
\[
 Z=\sum_C\sum_{u\in C\setminus\{u_C\}}z_u(u-u_C).
\]
By \eqref{eq:Karea}, $\|u(V)-u_C(V)\|_2\le13\eps$.
Passing to conjugation gives, for every contraction $Y\in M_d$,
\[
 \|\Ad(u(V))(Y)-\Ad(u_C(V))(Y)\|_2
 \le2\|u(V)-u_C(V)\|_2
 \le26\eps.
\]
Thus the residual contributes at most
\[
 26\eps\sum_C\sum_{u\in C\setminus\{u_C\}}|z_u|
 =13C_K\eps\le26\eps
\]
to the absolute value of the quadratic form.
Under this conjugation action, the Hermitian squares remain
positive and the formal expression defining $D_K$ becomes
$\mathcal L$. The residual bound therefore proves
\eqref{eq:Kquadratic}.

\subsection{Spectral correction with explicit tolerances}
\label{subsec:spectralcorrection}

We now use the preceding approximate spectral bounds to construct
an averaging map for the algebra reconstruction in
Section~\ref{sec:algebra}.
Lemma~\ref{lem:spectralfilter} uses powers of $T$ to reduce
commutator energy and bounds the resulting change in a matrix.
Theorem~\ref{thm:input} then corrects the unitaries to obtain
exactly invariant blocks with spectral gaps away from scalar
matrices on each block. The resulting averaging map is close to
a power of $T$ and approximately idempotent in the
$\|\cdot\|_{\infty\to2}$ norm.

The correction follows Alekseev, Liu, and Thom
\cite[proofs of Theorem~4.3 and Corollary~5.1]{ALT}.
We specify the tolerances explicitly because they determine the
final penalty coefficient. All choices are independent of the
matrix dimension.

Fix an integer $m\ge1$ and rational numbers
$0<\lambda<1$, $C\ge0$, and $0<\theta_{\rm sp}<1$,
where $\theta_{\rm sp}$ is the desired accuracy. Set
\[
 \kappa=\lambda/2,\qquad c=\kappa^2/2^{28}.
\]
For a positive rational number $s$, define
$d_2(s)=\min\{j\in\Z_{\ge0}:2^j\ge s\}$.
The power $N$ controls decay on the corrected corners, while $r_I$
is the total unitary-edit tolerance. The constants $K_0,K_b,K_\gamma$
transfer the projection errors into that tolerance; $a_I,t_I$ are
used for coverage and orthogonalization. Set
\begin{equation}\label{eq:inputparameters1}
\begin{gathered}
 \kappa'=\kappa/1024,\qquad
 N=\lceil1/c\rceil d_2(16/\theta_{\rm sp}),\qquad
 r_I=\frac{\theta_{\rm sp}}{16N},\qquad
 C_0=\frac{2^{20}}{\kappa^2},\\
 K_0=4m+\frac{256m}{\kappa'},\qquad
 K_b=72K_0+128m,\qquad
 K_\gamma=\frac{3\cdot2^{19}}{\kappa^2}K_0+112m,\\
 b_I=\frac{r_I^2}{8K_b},\qquad
 a_I=\min\left\{\frac14,\frac{r_I^2}{512m}\right\},\qquad
 t_I=\left\lceil(192/b_I)^2\right\rceil.
\end{gathered}
\end{equation}
The next choices control projection repair through $\gamma,v_I$
and the remaining spectral filtering through $\beta,L,\mu_I,m_{\rm cut}$.
The final entry $\eps_{\rm in}$ is the admissible relation error:
\begin{equation}\label{eq:inputparameters2}
\begin{gathered}
 \gamma=\min\left\{
 \frac14,\frac{\kappa^2}{2^{20}},
 \frac{r_I^2}{8K_\gamma},\frac{b_I}{222},
 \frac{a_Ib_I}{168\cdot2^{8t_I}}
 \right\},\\
 v_I=\min\left\{
 \frac12,\frac{r_I^2}{128mC_0},
 \frac{r_I^2}{128m},\frac{\gamma^8}{3}
 \right\},\qquad
 \alpha=v_I^2,\qquad
 \beta=\min\left\{\frac{\theta_{\rm sp}}{16},\frac{\alpha^2}{64}\right\},\\
 L=\max\left\{N,\left\lceil\frac{2}{\kappa\beta}\right\rceil\right\},
 \qquad
 \mu_I=\frac{\beta^2}{4L^2},\qquad
 \rho_{\rm filt}=\left(\frac{1-\lambda}{1-\kappa}\right)^2,\\
 m_{\rm cut}=\lceil1/(1-\rho_{\rm filt})\rceil d_2(2/(\kappa\mu_I)),\\
 \delta_*=\frac{\kappa\mu_I(1-\kappa)^{2m_{\rm cut}}}2,\qquad
 \eps_{\rm in}=\min\left\{1,\frac{\delta_*}{C+1}\right\}.
\end{gathered}
\end{equation}

All denominators are positive, and these choices use only rational
arithmetic, integer powers, minima, maxima, and ceilings.
Bernoulli's inequality gives the bounds
\begin{equation}\label{eq:filterchoices}
 (1-c)^N,(1-\kappa)^N\le\theta_{\rm sp}/16,\qquad
 (1-\kappa)^L\le\beta/2,\qquad
 \rho_{\rm filt}^{m_{\rm cut}}\le\kappa\mu_I/2.
\end{equation}
We also have
\[
 0<\alpha\le1/4,\qquad 9\alpha\le\gamma^{16},\qquad
 \beta\le\alpha^2/64.
\]

Throughout this subsection, let $v_1,\ldots,v_m\in\mathcal U(d)$,
with $T,\cE$ as in \eqref{eq:energy} and $\mathcal L=I-T$, and assume
\begin{equation}\label{eq:residualhyp}
 \ip{(\mathcal L^2-\lambda\mathcal L)y}{y}\ge-C\eps
 \quad(\|y\|\le1),\qquad 0\le\eps\le\eps_{\rm in}.
\end{equation}
We first estimate the effect of repeated averaging.

\begin{lemma}\label{lem:spectralfilter}
For every contraction $x\in M_d$,
\begin{equation}\label{eq:betafilter}
 \|x-T^Lx\|_2\le\kappa^{-1/2}\sqrt{\cE(x)}+\beta/2,
 \qquad
 \cE(T^Lx)\le\mu_I+(1-\kappa)^{2L}\le\beta^2/2.
\end{equation}
Moreover,
\begin{equation}\label{eq:powersclose}
 \max\bigl\{\|T^N-T^L\|_{\infty\to2},
             \|T^{2N}-T^N\|_{\infty\to2}\bigr\}
 \le\beta/2+\theta_{\rm sp}/16.
\end{equation}
\end{lemma}
\begin{proof}
For a contraction $x$, let $\nu_x$ be its spectral measure for
$\mathcal L$, supported in $[0,1]$ and of mass $\|x\|_2^2\le1$.
Apply \eqref{eq:residualhyp} to the contraction $T^{m_{\rm cut}}x$.
On $[0,\kappa]$, we have $s(\lambda-s)\ge\kappa s$ and
$(1-s)^{2m_{\rm cut}}\ge(1-\kappa)^{2m_{\rm cut}}$.
The contribution from $[\kappa,\lambda]$ is nonpositive,
whereas the positive contribution from $[\lambda,1]$ is at most
$(1-\lambda)^{2m_{\rm cut}}$. Hence
\begin{equation}\label{eq:lowmass}
 \int_{[0,\kappa]}s\,d\nu_x(s)
 \le\frac{C\eps+(1-\lambda)^{2m_{\rm cut}}}
          {\kappa(1-\kappa)^{2m_{\rm cut}}}
 \le\mu_I.
\end{equation}
Split the spectrum at $\kappa$. On the lower interval, use
$1-(1-s)^L\le Ls$ and \eqref{eq:lowmass}; on the upper interval,
use $(1-s)^L\le(1-\kappa)^L$.
These bounds give \eqref{eq:betafilter}.
The same decomposition, together with $L\ge N$ and
\eqref{eq:filterchoices}, gives \eqref{eq:powersclose}.
\end{proof}

We now correct the unitaries and construct the averaging map.
For a nonzero projection $e_i$, write
\[
 \tau_i(Y)=\frac{\tr_d(Y)}{\tr_d(e_i)},\qquad
 \|Y\|_{2,\tau_i}=\tau_i(Y^*Y)^{1/2}
 \quad(Y\in e_iM_de_i)
\]
for the normalized trace and its associated norm on the corner
$e_iM_de_i$.

\begin{theorem}\label{thm:input}
With the parameters above and under \eqref{eq:residualhyp},
the following hold.
\begin{enumerate}
\renewcommand{\labelenumi}{\textup{(\roman{enumi})}}
\item For every contraction $x\in M_d$,
\begin{equation}\label{eq:inputfilter}
 \|x-T^Lx\|_2\le\kappa^{-1/2}\sqrt{\cE(x)}+\theta_{\rm sp},
 \qquad \sqrt{\cE(T^Lx)}\le\theta_{\rm sp}.
\end{equation}
\item There are nonzero mutually orthogonal projections $e_i$
with $\sum_i e_i=1$ and, for each $j$, two unitaries
$V_j^+,V_j^-$ commuting with every $e_i$.
Writing $\sum_{j,\pm}$ for the sum over $1\le j\le m$ and
both signs, we have
\begin{equation}\label{eq:inputedit}
 \sum_{j,\pm}\|V_j^\pm-v_j\|_2^2\le\theta_{\rm sp}^2.
\end{equation}
For every $Y\in e_iM_de_i$,
\begin{equation}\label{eq:cornerspectralgap}
 \frac1{8m}\sum_{j,\pm}\|[V_j^\pm,Y]\|_{2,\tau_i}^2
 \ge c\|Y-\tau_i(Y)e_i\|_{2,\tau_i}^2.
\end{equation}
\item There is a unital completely positive map $F_0:M_d\to M_d$
that preserves the trace, is selfadjoint on $L^2$, and satisfies
$F_0(e_i x e_j)=e_iF_0(x)e_j$ for all $x\in M_d$, with
\begin{equation}\label{eq:inputchannel}
\begin{gathered}
 \|F_0^2-F_0\|_{\infty\to2}\le\theta_{\rm sp},\qquad
 \|F_0-T^L\|_{\infty\to2}\le\theta_{\rm sp},\\
 \max_i\|F_0|_{e_iM_de_i}-\tau_i(\cdot)e_i\|_{2\to2}
 \le\theta_{\rm sp}.
\end{gathered}
\end{equation}
\end{enumerate}
\end{theorem}
\begin{proof}
Part~\textup{(i)} follows from Lemma~\ref{lem:spectralfilter}, since
$\beta\le\theta_{\rm sp}/16$. For part~\textup{(ii)}, we follow the
projection selection and orthogonalization in
\cite[Theorem~4.3]{ALT}, keeping the total error over all corners.
We use the Hilbert-Schmidt estimates for functional calculus
\cite[Lemmas~VII.5.4-VII.5.5]{Bhatia}, the inequality of Araki and
Yamagami \cite[Theorem~1]{AY} for rectangular matrices by zero padding,
and the coarea, cutoff, and polar completion estimates in
\cite[Lemma~2.3, Section~2.4, and (3.12)]{ALT}.
All auxiliary spaces retain the trace denominator $d$; in particular,
$\|Z\|_2^2=d^{-1}\Tr(Z^*Z)$ for a rectangular matrix $Z$.

Apply the projection correction in
\cite[proof of Theorem~2.4]{ALT} to $T^L$.
Lemma~\ref{lem:spectralfilter} and $\beta\le\alpha^2/64$ give a
projection $r_0$ such that
\[
 \tr_d(1-r_0)\le\kappa\beta^2/\alpha+36\beta^2/\alpha^2\le\alpha.
\]
For every projection $p\le r_0$ with
$\cE(p)<\kappa\tr_d(p)/64$, the same correction gives a projection
$f$ satisfying
\[
\begin{gathered}
 \|f-p\|_2^2\le36\kappa^{-1}\cE(p),\qquad
 \cE(f)\le\alpha\tr_d(p),\\
 \tfrac13\tr_d(p)\le\tr_d(f)\le\tfrac53\tr_d(p).
\end{gathered}
\]
We select these projections as in
\cite[proof of Theorem~4.3, Step~1]{ALT}.
For a selfadjoint matrix $S$ and a real interval $\mathcal J$, let $1_{\mathcal J}(S)$
denote its spectral projection; let $\operatorname{pol}(X)$ denote
the partial isometry in the polar decomposition of $X$.
Starting with $S_0=1-r_0$, set
\[
 b_i=1_{[0,\kappa/512]}(S_{i-1}),\qquad
 w_i=\operatorname{pol}(r_0b_i).
\]
Then $w_i^*w_i=b_i$ and $w_iw_i^*\le r_0$.
Choose a nonzero projection $p_i\le w_iw_i^*$ of minimum rank with
$\cE(p_i)<\kappa\tr_d(p_i)/64$, correct it to $f_i$, and put
$S_i=S_{i-1}+f_i$. Stop when no such $p_i$ exists.
The selection terminates and satisfies
\begin{equation}\label{eq:selectionbounds}
\begin{gathered}
 \sum_i\tr_d(p_i)\le8,\qquad \sum_i\tr_d(f_i)\le8/3,\qquad
 A_*:=\cE(1-r_0)+\sum_i\cE(f_i)\le9\alpha\le\gamma^{16},\\
 \tr_d\left(1_{[0,\gamma]}\left(1-r_0+\sum_i f_i\right)\right)
 \le C_0\sqrt\alpha.
\end{gathered}
\end{equation}
For the last estimate, the coarea argument in
\cite[Theorem~4.3, Step~2]{ALT} gives the upper bound
$2^{17}\sqrt{2A_*}/\kappa^2\le C_0\sqrt\alpha$.

To obtain projections that can be made orthogonal, set
\[
 \lambda_0=\gamma^2,\qquad R_i=(\lambda_0+S_i)^{-1},\qquad
 D_i=\lambda_0(R_{i-1}-R_i).
\]
The resolvent estimates of \cite[Lemmas~3.2-3.3]{ALT} give, for every
unitary $u$,
\begin{equation}\label{eq:resolventledger}
\begin{gathered}
 \sum_i\|[u,R_{i-1}]f_i\|_2^2
 \le148\lambda_0^{-3}\left(\|[u,1-r_0]\|_2^2+
                              \sum_i\|[u,f_i]\|_2^2\right),\\
 \sum_i\cE(D_i)\le2696\lambda_0^{-4}A_*,\qquad
 \sum_i\tr_d(S_{i-1}^2D_i)\le3\lambda_0.
\end{gathered}
\end{equation}
For the constant $148$, integration in
\cite[proof of Lemma~3.2]{ALT} gives
$72\lambda_0^{-3}+2\lambda_0^{-2}\le74\lambda_0^{-3}$ for
selfadjoint unitaries; dilation to matrices of twice the size gives
$148$ for general unitaries. For the constant $2696$, factor
\[
 D_i=Y_iY_i^*,\qquad
 Y_i=\sqrt{\lambda_0}R_{i-1}f_i
       (1+f_iR_{i-1}f_i)^{-1/2}.
\]
Here $0\le D_i\le1$, so $\|Y_i\|\le1$, and functional calculus gives
\[
 \|[u,Y_i]\|_2
 \le2\lambda_0^{-3/2}\|[u,f_i]\|_2
       +\tfrac32\lambda_0^{-1/2}\|[u,R_{i-1}]f_i\|_2.
\]
Squaring, summing, and using the product rule for $D_i=Y_iY_i^*$
gives the coefficient $4(8+\tfrac92\cdot148)=2696$.

We use the following orthogonalization estimate
\cite[(3.14)]{ALT}. For a family $q=(q_i)$ of projections with
\[
 \sum_i\tr_d(q_i)\le3,\qquad \sum_i\cE(q_i)\le s,
 \qquad s\ge0,
\]
write $h_+=(h+|h|)/2$ for the positive part of a selfadjoint matrix
$h$, and define
\[
 \operatorname{def}(q)
 =\tr_d\left(\left(1-\sqrt{\sum_iq_i}\right)_+^2\right).
\]
For $0<a\le1/4$ and $t\ge1$, there are partial isometries $Z_i$
with pairwise orthogonal projections $Z_i^*Z_i$ and $Z_iZ_i^*\le q_i$
such that
\begin{equation}\label{eq:orthogenergy}
 \sum_i\cE(Z_i)
 \le2s+\frac{8e^{4t}\sqrt{s}}{a}+\frac{64}{\sqrt{t}},
 \qquad
 \tr_d\left(1-\sum_iZ_i^*Z_i\right)
 \le\operatorname{def}(q)+4a.
\end{equation}
The coarea estimate on $[\gamma,2\gamma]$ and the coverage calculation
in \cite[Lemma~3.3]{ALT} supply projections $q_i$ with
$\rank q_i\le\rank f_i$ and
\[
\begin{gathered}
 \sum_i\cE(q_i)\le37\gamma^2,\qquad
 \sum_i\tr_d(q_i)\le8/3,\qquad
 \sum_i\tr_d(S_{i-1}^2q_i)\le3\gamma,\\
 \operatorname{def}(q)\le\alpha+C_0\sqrt\alpha+7\gamma.
\end{gathered}
\]
Apply \eqref{eq:orthogenergy} with $s=37\gamma^2$, $a=a_I$, and
$t=t_I$. For the resulting $Z_i$, put
\[
 B_*=\sum_i\cE(Z_i),\qquad P_i=Z_i^*Z_i,\qquad
 P_0=1-\sum_iP_i,\qquad Q_i=Z_iZ_i^*.
\]
The choices of $\gamma$ and $t_I$ give
\[
 B_*\le74\gamma^2+56\cdot2^{8t_I}\gamma/a_I+64/\sqrt{t_I}
 \le b_I,
\]
and the partition $(P_i)_{i\ge0}$ satisfies
$\tr_d(P_0)\le\alpha+C_0\sqrt\alpha+7\gamma+4a_I$.
Pinching and polar completion give unitaries $W_j$ commuting with
every $P_i$ such that
\[
 \sum_j\|W_j-v_j\|_2^2\le32mB_*.
\]
Following \cite[(4.12)]{ALT}, define
\[
 \Delta_{j,i}^{\pm}=v_j^{\pm1}Z_i-Z_iW_j^{\pm1},\qquad
 \varepsilon_i^{\rm tr}=\frac1{2m}\sum_{j,\pm}\|\Delta_{j,i}^{\pm}\|_2^2
       +2\tr_d((1-b_i)Q_i).
\]
Let $\cE_W$ denote the energy \eqref{eq:energy} for the tuple $(W_j)$.
Transporting the minimum rank condition to $P_i$ gives
\cite[Theorem~4.3, Step~3, (4.13)]{ALT}
\[
 \cE_W(q)\ge\kappa'\tr_d(q)-\varepsilon_i^{\rm tr}
 \qquad(q\text{ a projection},\ q\le P_i,\
                      \rank q\le\rank P_i/2).
\]
Moreover,
\[
 \sum_i\varepsilon_i^{\rm tr}\le72B_*+3\cdot2^{19}\gamma/\kappa^2.
\]
Indeed, the product rule and $\sum_iP_i\le1$ give
\[
 \sum_{j,i,\pm}\|\Delta_{j,i}^{\pm}\|_2^2
 \le16mB_*+4\sum_j\|W_j-v_j\|_2^2\le144mB_*.
\]
Since $Q_i\le q_i$ and
$1-b_i\le(512/\kappa)^2S_{i-1}^2$, we also have
\[
 \sum_i\tr_d((1-b_i)Q_i)
 \le(512/\kappa)^2\sum_i\tr_d(S_{i-1}^2q_i)
 \le3\cdot2^{18}\gamma/\kappa^2.
\]

Apply \cite[Lemma~4.2]{ALT} to the partition $(P_i)_{i\ge0}$,
the unitary tuple $(W_j)$, and the errors $(\varepsilon_i^{\rm tr})_{i\ge1}$.
This gives corrected unitaries $V_j^\pm$ and retained
subprojections of the $P_i$, with scalar gap
$(\kappa')^2/256=c$ on every nonzero retained corner.
The construction takes $V_j^\pm$ to be the identity on the
complement of the retained projections. This complement
includes both $P_0$ and all portions discarded from the other
$P_i$.

Split this entire complement into projections of rank one,
adjoin them to the nonzero retained projections, and denote
the resulting partition by $(e_i)$. This refinement leaves
the corrected unitaries and their error bounds unchanged.
Every $V_j^\pm$ commutes with every $e_i$, and the scalar gap
holds on each corner: it holds on the retained corners by
the lemma and on the new corners because every matrix in
a corner of rank one is scalar. 
Their total error satisfies
\[
 \sum_{j,\pm}\|V_j^\pm-v_j\|_2^2
 \le K_bB_*+K_\gamma\gamma+16m\alpha+
       16mC_0\sqrt\alpha+64ma_I\le5r_I^2/8.
\]

Each of the five terms is at most $r_I^2/8$ by
\eqref{eq:inputparameters1} -\eqref{eq:inputparameters2}.
Since $r_I\le\theta_{\rm sp}$, this proves part~\textup{(ii)}.

For part~\textup{(iii)}, define the averaging operator of the corrected
tuple and its $N$th power by
\[
 Q_0(x)=\frac{x}{2}
 +\frac1{8m}\sum_{j,\pm}
 \bigl(V_j^\pm x(V_j^\pm)^*
       +(V_j^\pm)^*xV_j^\pm\bigr),
 \qquad F_0=Q_0^N.
\]
Both maps are unital, completely positive, trace preserving,
selfadjoint on $L^2$, and satisfy the block identity in
part~\textup{(iii)}. Since $0\le Q_0\le I$ on $L^2$, the
scalar gap on each corner yields
\[
 \max_i
 \|F_0|_{e_iM_de_i}-\tau_i(\cdot)e_i\|_{2\to2}
 \le(1-c)^N\le\theta_{\rm sp}/16.
\]
The error bound and Cauchy--Schwarz give
\[
 \|Q_0-T\|_{\infty\to2}
 \le\frac1{2m}\sum_{j,\pm}\|V_j^\pm-v_j\|_2
 \le\frac{r_I}{\sqrt{2m}}\le r_I.
\]
Both $Q_0$ and $T$ are contractive in the operator norm and the
$2$-norm, so telescoping gives
\[
 \|Q_0^k-T^k\|_{\infty\to2}\le kr_I\qquad(k\ge1).
\]
Finally, \eqref{eq:powersclose}, $Nr_I=\theta_{\rm sp}/16$, and
$\beta\le\theta_{\rm sp}/16$ imply
\[
 \begin{aligned}
 \|F_0-T^L\|_{\infty\to2}
 &\le Nr_I+\beta/2+\theta_{\rm sp}/16
 \le5\theta_{\rm sp}/32,\\
 \|F_0^2-F_0\|_{\infty\to2}
 &\le3Nr_I+\beta/2+\theta_{\rm sp}/16
 \le9\theta_{\rm sp}/32.
 \end{aligned}
\]
This proves \eqref{eq:inputchannel}.
\end{proof}

\subsection{Reconstruction and approximate commutants}
\label{sec:algebra}

Section~\ref{subsec:spectralcorrection} produces an approximately
idempotent averaging map $F_0$. We now replace it by the conditional
expectation $E_A$ onto an actual matrix subalgebra $A$.
Proposition~\ref{thm:reconstruction} gives this replacement under
the block conditions established in Theorem~\ref{thm:input}.
Combining it with the filtering estimates yields
Theorem~\ref{thm:interface}: every contraction with small commutator
energy is close to $A$, and every contraction in $A$ nearly commutes
with the input unitaries. Both estimates are independent of the
matrix dimension.

For an exactly idempotent unital completely positive map preserving
the trace, its range is a subalgebra and the map is the conditional
expectation onto that range
\cite[Chapters~3 and~4]{Paulsen}; see also
\cite[Theorems~6.1.8 and~6.2.2]{Peterson}.
The following quantitative version uses the reconstruction argument
of Alekseev, Liu, and Thom \cite[proof of Theorem~5.2]{ALT}.

\begin{proposition}\label{thm:reconstruction}
Let $e_i$ be nonzero mutually orthogonal projections in $M_d$ with
$\sum_i e_i=1$. Write
\[
 m_i=\rank e_i,\qquad
 \tau_i(x)=\frac{\Tr(x)}{m_i}\quad(x\in e_iM_de_i).
\]
Suppose $F_0:M_d\to M_d$ is unital, completely positive, preserves
the trace, is selfadjoint on $L^2$, and satisfies
$F_0(e_i x e_j)=e_iF_0(x)e_j$ for all $i,j$ and $x\in M_d$.
If $0<\zeta<1$ and
\[
 \delta_0:=\|F_0^2-F_0\|_{\infty\to2}\le\frac{\zeta^8}{2^{96}},
 \qquad
 \sigma:=\max_i
 \|F_0|_{e_iM_de_i}-\tau_i(\cdot)e_i\|_{2\to2}
 \le\frac{\zeta^8}{2^{96}},
\]
then there is a unital $*$-subalgebra $A\subseteq M_d$ containing
every $e_i$ such that
\[
 \|F_0-E_A\|_{\infty\to2}\le\zeta.
\]
\end{proposition}
\begin{proof}
We follow the reconstruction in \cite[Theorem~5.2]{ALT}.
In this proof, put $\rho_{\rm rec}=(\zeta/4096)^2$, so
$\delta_0,\sigma\le\rho_{\rm rec}^4$ and $0<\rho_{\rm rec}<2^{-24}$.
The block identity makes each rectangular block
$e_iM_de_j$ invariant under $F_0$; denote its restriction by
$\Phi_{ij}$. These restrictions are selfadjoint contractions on
the corresponding Hilbert--Schmidt spaces; write $\spec$ for the spectrum.
The truncation and matching argument in \cite[(5.4)-(5.5)]{ALT}
retains a set of indices with block sum
$p=\sum_{i\text{ retained}}e_i$ such that
\[
 \tr_d(1-p)\le64\delta_0^2/\rho_{\rm rec}^6\le64\rho_{\rm rec}^2,
 \qquad
 \spec(\Phi_{ij})\subseteq[-\rho_{\rm rec},\rho_{\rm rec}]\cup[1-\rho_{\rm rec},1]
\]
for all retained $i,j$.

Using the unnormalized Hilbert--Schmidt norm from Section~\ref{sec:prelim}, set
\[
 \mathcal D(x)=\|x\|_{\HS}^2-\operatorname{Re}\Tr(x^*F_0(x)).
\]
The square root of $\mathcal D$ is a seminorm satisfying
$\mathcal D(x^*)=\mathcal D(x)$. For contractions $x,y$, the product
estimate \cite[(5.7)]{ALT} gives
\[
 \|F_0(xy)-xy\|_{\HS}
 \le\sqrt2\bigl(\mathcal D(x)^{1/2}+\mathcal D(y)^{1/2}\bigr).
\]
Call a nonzero eigenvector of $\Phi_{ij}$ high if its eigenvalue
lies in $[1-\rho_{\rm rec},1]$. For such a vector $v$, put
$m_{\max}=\max(m_i,m_j)$ and $t=\min(m_i,m_j)$, and normalize
$\|v\|_{\HS}^2=m_{\max}$.
For its eigenvalue $\lambda_v\ge1-\rho_{\rm rec}$, the bounds
$\sigma\le\rho_{\rm rec}^4$ and $0<\rho_{\rm rec}<2^{-24}$ imply
\[
 (\lambda_v^2-\sigma)^{-1}
 \le\bigl((1-\rho_{\rm rec})^2-\rho_{\rm rec}^4\bigr)^{-1}
 \le1+3\rho_{\rm rec}.
\]
Also $\lambda_v^2-\sigma\ge1-3\rho_{\rm rec}$.
Thus the fourth moment and rank inequalities in \cite[(5.4)]{ALT} give
\[
 t\ge(1-3\rho_{\rm rec})m_{\max},\qquad
 \Tr|v|^4\le(1+3\rho_{\rm rec})m_{\max}.
\]
Complete the polar factor of $v$ to a partial isometry $u\in e_iM_de_j$
of rank $t$. Applying $(s-1)^2\le(s^2-1)^2$ to the singular values gives
\[
\begin{aligned}
 \|u-v\|_{\HS}^2
 &\le\Tr|v|^4-2m_{\max}+t\le6\rho_{\rm rec} m_{\max},\\
 \mathcal D(u)^{1/2}
 &\le\sqrt{\rho_{\rm rec} m_{\max}}+\sqrt2\|u-v\|_{\HS}
 \le5\sqrt{\rho_{\rm rec} m_{\max}}.
\end{aligned}
\]

Each retained block has at most one high eigendirection.
Indeed, suppose $v,w$ are orthogonal high eigenvectors with the above
normalization, and let $u,z$ be their completed polar factors.
Then $|\Tr(u^*z)|\le2\sqrt6\,\sqrt{\rho_{\rm rec}}\,m_{\max}$.
For $a=u^*z\in e_jM_de_j$, the rank comparison, the hypothesis on
diagonal blocks, and the product estimate give
\[
\begin{aligned}
 \|\tau_j(a)e_j\|_{\HS}&\le8\sqrt{\rho_{\rm rec} m_{\max}},\\
 \|a-\tau_j(a)e_j\|_{\HS}
 &\le10\sqrt2\,\sqrt{\rho_{\rm rec} m_{\max}}+\sigma\sqrt{m_{\max}}
 \le16\sqrt{\rho_{\rm rec} m_{\max}}.
\end{aligned}
\]
But $\|a\|_{\HS}^2\ge2t-m_{\max}\ge(1-6\rho_{\rm rec})m_{\max}$, so these
inequalities imply $1-6\rho_{\rm rec}\le(8+16)^2\rho_{\rm rec}$, or $1\le582\rho_{\rm rec}$,
contradicting $\rho_{\rm rec}<2^{-24}$.

For retained indices, define $i\sim j$ when $\Phi_{ij}$ has a high
eigenvector, and let $u_{ij}$ be its completed polar factor.
This relation is reflexive and symmetric because $F_0(e_i)=e_i$
and $F_0$ preserves adjoints. To prove transitivity, suppose
$i\sim j\sim k$ and put
$a=u_{ij}u_{jk}$ and $M_*=\max(m_i,m_j,m_k)$.
The rank and product estimates give
\[
\begin{aligned}
 \|a\|_{\HS}^2
 &\ge\min(m_i,m_j)+\min(m_j,m_k)-m_j
 \ge(1-6\rho_{\rm rec})M_*,\\
 \|F_0(a)-a\|_{\HS}&\le15\sqrt{\rho_{\rm rec} M_*}.
\end{aligned}
\]
If $i\not\sim k$, spectral separation gives
$\|\Phi_{ik}\|_{2\to2}\le\rho_{\rm rec}$. Consequently,
$(1-\rho_{\rm rec})^2(1-6\rho_{\rm rec})\le225\rho_{\rm rec}$.
The left side is at least $1/8$, giving $1\le1800\rho_{\rm rec}$, again a
contradiction. Thus $\sim$ is an equivalence relation.

In each class choose an index $o$ of least rank. Set $U_o=e_o$,
$U_i=u_{io}$ for $i\ne o$ in that class, and
\[
 q_i=U_iU_i^*,\qquad W_{ij}=U_iU_j^*\quad(i\sim j).
\]
Since $U_i^*U_i=e_o$, the $W_{ij}$ form exact matrix units with
$W_{ii}=q_i$. Moreover, $m_o\ge(1-3\rho_{\rm rec})m_i$ for every $i$ in the class.
For distinct $i\sim j$, the product estimate gives
\[
 \frac{\|\Phi_{ij}(W_{ij})-W_{ij}\|_{\HS}}{\sqrt{m_o}}
 \le\frac{10\sqrt2}{\sqrt{1-3\rho_{\rm rec}}}\sqrt{\rho_{\rm rec}}\le15\sqrt{\rho_{\rm rec}}.
\]
Let $\Pi_{ij}$ be the orthogonal projection onto the high eigenline,
and, for any nonzero matrix $w$, let $P_w$ be the orthogonal projection
onto $\mathbb Cw$, using the Hilbert-Schmidt inner product.
Spectral separation and the distance formula for projections onto
lines imply
\[
 \|\Pi_{ij}-P_{W_{ij}}\|_{2\to2}
 =\left\|(I-\Pi_{ij})\frac{W_{ij}}{\sqrt{m_o}}\right\|_{\HS}
 \le\frac{15\sqrt{\rho_{\rm rec}}}{1-\rho_{\rm rec}}\le16\sqrt{\rho_{\rm rec}}.
\]
Since $\|\Phi_{ij}-\Pi_{ij}\|_{2\to2}\le\rho_{\rm rec}$, it follows that
\[
 \|\Phi_{ij}-P_{W_{ij}}\|_{2\to2}
 \le\rho_{\rm rec}+16\sqrt{\rho_{\rm rec}}\le41\sqrt{\rho_{\rm rec}}.
\]
On diagonal blocks,
\[
 \|\Phi_{ii}-P_{q_i}\|_{2\to2}
 \le\sigma+\sqrt{1-m_o/m_i}
 \le\sigma+\sqrt{3\rho_{\rm rec}}\le41\sqrt{\rho_{\rm rec}};
\]
between distinct classes, $\|\Phi_{ij}\|_{2\to2}\le\rho_{\rm rec}$.

Let $A_0$ be the algebra generated by the $W_{ij}$, with unit
$q=\sum_iq_i$, taking $q_i=0$ on discarded indices. Define
\[
 A=A_0\oplus(1-q)M_d(1-q).
\]
Then $q\le p$, $\tr_d(1-q)\le64\rho_{\rm rec}^2+3\rho_{\rm rec}\le4\rho_{\rm rec}$, and
$e_i=q_i+(e_i-q_i)\in A$ for every $i$.
Let $E_{A_0}:qM_dq\to A_0$ be the conditional expectation preserving
the restricted trace $\tr_d$. Thus
\[
 E_A(x)=E_{A_0}(qxq)+(1-q)x(1-q).
\]
On each retained rectangle, the map $x\mapsto E_{A_0}(qxq)$ is
$P_{W_{ij}}$ within a class and zero between classes.
Orthogonality of the rectangles therefore gives
\[
 \|F_0(pxp)-E_{A_0}(qxq)\|_2
 \le41\sqrt{\rho_{\rm rec}}\,\|pxp\|_2\qquad(x\in M_d),
\]
with no dependence on the number of blocks.
For a contraction $x$, use $L^2$ contractivity and
$\|x-pxp\|_2\le\sqrt{2\tr_d(1-p)}$ to obtain
\[
\begin{aligned}
 \|(F_0-E_A)(x)\|_2
 &\le\|F_0(x-pxp)\|_2
       +\|F_0(pxp)-E_{A_0}(qxq)\|_2\\
 &\quad+\|(1-q)x(1-q)\|_2\\
 &\le\sqrt{2\tr_d(1-p)}+41\sqrt{\rho_{\rm rec}}+\sqrt{\tr_d(1-q)}\\
 &\le\sqrt{128}\rho_{\rm rec}+43\sqrt{\rho_{\rm rec}}\le55\sqrt{\rho_{\rm rec}}\le\zeta.
\end{aligned}
\]
This proves the required bound in the $\infty\to2$ norm.
\end{proof}

We now combine Theorem~\ref{thm:input} and
Proposition~\ref{thm:reconstruction} to obtain the main result of
this section.
Fix integers $d,m\ge1$ and rational numbers
$0<\lambda<1$, $C\ge0$, and $0<\zeta_{\rm alg}<1$.
Set
\[
 \theta_{\rm sp}=\frac{\zeta_{\rm alg}^8}{2^{96}},
\]
and let $\eps_{\rm in}$ be the tolerance defined by
\eqref{eq:inputparameters1}-\eqref{eq:inputparameters2}
for these inputs. It depends only on $m,\lambda,C,\zeta_{\rm alg}$.
For $v_1,\ldots,v_m\in\mathcal U(d)$, use $T,\cE$ from
\eqref{eq:energy} and put $\mathcal L=I-T$.

\begin{theorem}[Approximate commutants]
\label{thm:interface}\label{cor:interface}
Suppose
\[
 \ip{(\mathcal L^2-\lambda\mathcal L)y}{y}\ge-C\eps
 \quad(y\in M_d,\ \|y\|\le1),\qquad
 0\le\eps\le\eps_{\rm in}.
\]
Then there is a unital $*$-subalgebra $A\subseteq M_d$ such that,
for every contraction $x\in M_d$ and every contraction $a\in A$,
\begin{align}
 \|x-E_Ax\|_2
 &\le\sqrt{2/\lambda}\,\sqrt{\cE(x)}+3\zeta_{\rm alg},
 \label{eq:algebrainterface}\\
 \max_{1\le j\le m}\|[v_j,a]\|_2
 &\le6\sqrt m\,\zeta_{\rm alg}.\notag
\end{align}
\end{theorem}
\begin{proof}
Apply Theorem~\ref{thm:input}, followed by
Proposition~\ref{thm:reconstruction} with $\zeta=\zeta_{\rm alg}$.
The choice of $\theta_{\rm sp}$ ensures that the proposition applies,
and the two approximation bounds give
\[
 \|E_A-T^L\|_{\infty\to2}
 \le\|E_A-F_0\|_{\infty\to2}+\|F_0-T^L\|_{\infty\to2}
 \le\zeta_{\rm alg}+\theta_{\rm sp}.
\]
Together with \eqref{eq:inputfilter}, this yields, for every
contraction $x\in M_d$,
\[
 \|x-E_Ax\|_2
 \le\sqrt{2/\lambda}\,\sqrt{\cE(x)}
       +\zeta_{\rm alg}+2\theta_{\rm sp}.
\]
Since $\theta_{\rm sp}\le\zeta_{\rm alg}$, the first estimate follows.

For a contraction $a\in A$, use $E_Aa=a$ and the fact that
$\sqrt{\cE}$ is a seminorm bounded by $\|\cdot\|_2$:
\[
 \sqrt{\cE(a)}
 \le\sqrt{\cE(T^La)}+\|a-T^La\|_2
 \le\zeta_{\rm alg}+2\theta_{\rm sp}\le3\zeta_{\rm alg}.
\]
Finally, \eqref{eq:energy} gives
$\|[v_j,a]\|_2\le2\sqrt m\,\sqrt{\cE(a)}$ for each $j$.
\end{proof}

\section{Quantitative normalization of approximate commutants}
\label{sec:reconstruction}

An exact inclusion between a matrix algebra and one of its unitary
conjugates is an equality, since the two algebras have the same
dimension. We prove a quantitative version under an additional
hypothesis: a unital $*$-subalgebra $B\subseteq M_d$ approximates
the common commutant of the conjugating unitaries, in the sense provided by
Theorem~\ref{thm:interface}. The main result,
Theorem~\ref{thm:normalization}, reverses the near inclusions supplied
by the compression relations, with bounds independent of the matrix
dimension and the number of unitaries. Section~\ref{sec:gate} uses
this reversal in the reflection estimate.

For unital $*$-subalgebras $E,A\subseteq M_d$ and $\nu\ge0$, write
\[
 E\subseteq_\nu A
 \quad\Longleftrightarrow\quad
 \sup_{x\in E,\ \|x\|\le1}\|x-E_Ax\|_2\le\nu.
\]
This near inclusion measures approximation uniformly on the unit
ball; the approximants $E_Ax$ are contractions because $E_A$ is
contractive. See also \cite{Christensen}, \cite[(2.1)]{Thom}, and
\cite[Remark~6.6]{PSS}.
We write $Z(A)=A\cap A'$ for the center of $A$.

We first prove an estimate for reconstructing a partition.
Finite averaging then supplies the required control on the center.
Thom's relative correction and
conditional median argument
\cite[Proposition~3.1, Lemma~4.1, and Theorem~4.2]{Thom}
then compare the matrix blocks and reverse the near inclusions.

\subsection{Central partitions and averaging}

The center $Z(A)$ is the algebra of functions on its minimal
projections, weighted by their traces. Its unital subalgebras
correspond to partitions of these atoms.
The next lemma reconstructs such a partition from an almost
idempotent Markov operator. The following subsection applies it
to $E_{Z(A)}E_B|_{Z(A)}$.

For a finite probability space $(\Omega,\mu)$ with positive atom
weights $\mu(x)$, a Markov operator has the form
\[
 (Pf)(x)=\sum_{y\in\Omega}P(x,y)f(y),\qquad
 P(x,y)\ge0,\qquad \sum_{y\in\Omega}P(x,y)=1.
\]
For a partition $\mathcal C$ of $\Omega$, its conditional expectation
averages on each part:
\[
 (E_{\mathcal C}f)(x)
 =\frac1{\mu(C)}\sum_{y\in C}\mu(y)f(y)
 \qquad(x\in C\in\mathcal C).
\]
Its range is the algebra of functions constant on each part.
All function norms below use the measure $\mu$.

\begin{lemma}\label{thm:commrecon}
Let $P$ be a Markov operator on $(\Omega,\mu)$, selfadjoint on
$L^2(\mu)$ with $0\le P\le I$. If $0<\zeta<1$ and
\[
 \delta_0:=\|P^2-P\|_{\infty\to2}\le\zeta^{12}/2^{42},
\]
then there is a partition $\mathcal C$ such that
\[
 \|P-E_{\mathcal C}\|_{\infty\to2}\le\zeta.
\]
The bound is independent of the number and weights of the atoms.
\end{lemma}
\begin{proof}
Put
\[
 \rho_{\rm cut}=\zeta^4/2^{12},\qquad
 r_{\rm cut}=\rho_{\rm cut}^2/64.
\]
Then $\delta_0\le\rho_{\rm cut}^3/64$.
Selfadjointness gives symmetric weights
$c(x,y)=\mu(x)P(x,y)$. Define the boundary of a set $S$ and the
energy of a function $f$ by
\[
 b(S)=\sum_{x\in S,\ y\notin S}c(x,y),\qquad
 \mathcal D(f)=\frac12\sum_{x,y\in\Omega}c(x,y)|f(x)-f(y)|^2
             =\langle(I-P)f,f\rangle.
\]
Write $1_S$ for the indicator of $S$.
Choose a maximal disjoint family of sets satisfying
\[
 \|(P-P^2)1_S\|_2>r_{\rm cut}\sqrt{\mu(S)},
\]
and let $\Omega_{\rm bad}$ be their union. Averaging the squared norm of
$(P-P^2)\sum_S\pm1_S$ over independent signs gives
\[
 \mu(\Omega_{\rm bad})\le\delta_0^2/r_{\rm cut}^2\le\rho_{\rm cut}^2.
\]
By maximality, no subset of $\Omega\setminus \Omega_{\rm bad}$ satisfies the
strict inequality above.

For $S\subseteq\Omega\setminus \Omega_{\rm bad}$ with $b(S)<\mu(S)/36$, put
$f=P1_S$. Since $0\le P\le I$,
\[
 \|f-1_S\|_2^2\le b(S),\qquad
 \mathcal D(f)\le r_{\rm cut}\mu(S).
\]
For $U_t=\{x:f(x)>t\}$ and $1/3\le t\le2/3$, writing
$\triangle$ for symmetric difference, we have
\[
 \mu(U_t\mathbin{\triangle}S)
 \le9\|f-1_S\|_2^2\le9b(S)<\mu(S)/4.
\]
Also $0\le f\le1$ and $\int f\,d\mu=\mu(S)$, so
$\mu\{f\ge1/3\}\le3\mu(S)$.
Coarea and Cauchy--Schwarz on edges meeting this set give
\[
 \int_{1/3}^{2/3}b(U_t)\,dt
 \le\sqrt{3\mu(S)\mathcal D(f)}.
\]
Thus some $U=U_t$ satisfies
\[
 b(U)\le3\sqrt3\,\sqrt{r_{\rm cut}}\mu(S)
 \le6\sqrt{r_{\rm cut}}\mu(S)
 \le\rho_{\rm cut}\mu(U).
\]

Starting with $R=\Omega\setminus \Omega_{\rm bad}$, choose a nonempty $S\subseteq R$
of least mass satisfying $b(S)<\mu(S)/36$. Form $U$ as above, remove
$C=U\cap R$, and repeat. If no such $S$ exists, keep the nonempty
remainder as the last part. Since $S\subseteq R$,
$C\mathbin{\triangle}S\subseteq U\mathbin{\triangle}S$, and hence
\[
 \tfrac34\mu(S)\le\mu(C)\le\tfrac54\mu(S),\qquad
 \mu(U)\le\tfrac53\mu(C).
\]
Every nonempty $Q\subseteq C$ with $\mu(Q)\le\mu(C)/2$ has
$\mu(Q)\le5\mu(S)/8<\mu(S)$, so minimality gives
$b(Q)\ge\mu(Q)/36$. The stopping condition gives the same bound
on the final remainder. Add $\Omega_{\rm bad}$ as an exceptional part if it is
nonempty, and denote the resulting partition by $\mathcal C$.
Counting each edge when its endpoints are first separated and
using $b(\Omega_{\rm bad})\le\mu(\Omega_{\rm bad})$ gives
\begin{equation}\label{eq:totalboundary}
 \sum_{C\in\mathcal C}b(C)
 \le2b(\Omega_{\rm bad})+2\sum_Ub(U)
 \le2\rho_{\rm cut}^2+\tfrac{10}{3}\rho_{\rm cut}
 <6\rho_{\rm cut}.
\end{equation}

For a real function $f$, let $g$ equal a median of $f$ on each
ordinary part and zero on $\Omega_{\rm bad}$, with $\|g\|_\infty\le\|f\|_\infty$.
Thus, in each ordinary part, both $\{f>g\}$ and $\{f<g\}$ have
at most half its mass. Integrating the preceding expansion bound
over the positive and negative level sets of $f-g$ gives
\[
\begin{aligned}
 \|f-g\|_1
 &\le\|f\|_\infty\rho_{\rm cut}
 +36\left(\sqrt{\mathcal D(f)/2}
                  +12\|f\|_\infty\rho_{\rm cut}\right),\\
 \|f-g\|_2^2&\le2\|f\|_\infty\|f-g\|_1.
\end{aligned}
\]
Here Cauchy--Schwarz bounds the contribution from edges within parts
by $\sqrt{\mathcal D(f)/2}$; edges between parts contribute at most
$2\|f\|_\infty\sum_Cb(C)$, and the exceptional part contributes
at most $\|f\|_\infty\mu(\Omega_{\rm bad})$.
For a real contraction $x$, take $f=Px$. Its energy is at most
$\delta_0$, so its squared distance to the partition algebra is at most
\[
 866\rho_{\rm cut}+72\sqrt{\delta_0/2}\le875\rho_{\rm cut}.
\]
Applying this to real and imaginary parts gives
\[
 \|(I-E_{\mathcal C})P\|_{\infty\to2}
 \le44\sqrt{\rho_{\rm cut}}.
\]
Conversely, \eqref{eq:totalboundary} and $(I-P)^2\le I-P$ give
\[
 \|(I-P)E_{\mathcal C}\|_{\infty\to2}
 \le4\sqrt{\rho_{\rm cut}}.
\]
Finally, pair the identity
\[
 (P-E_{\mathcal C})^2
 =(P^2-P)+(I-E_{\mathcal C})P+(I-P)E_{\mathcal C}
\]
with an arbitrary contraction. Since $P-E_{\mathcal C}$ is
selfadjoint, this yields
\[
 \|P-E_{\mathcal C}\|_{\infty\to2}^2
 \le\delta_0+48\sqrt{\rho_{\rm cut}}
 \le49\zeta^2/64<\zeta^2.\qedhere
\]
\end{proof}

\subsection{Finite mixing and invariance of the center}

We first approximate $E_B$ by finite powers of an averaging map,
then use this approximation to control its action on $Z(A)$.

\begin{lemma}\label{lem:correctedmixing}
Let $B\subseteq M_d$ be a unital $*$-subalgebra, put $P=E_B$,
and let $V_1,\ldots,V_m\in\mathcal U(d)$, where $m\ge1$ is an integer.
Define
\[
 \cE_V(x)=\frac1{4m}\sum_{j=1}^m\|[V_j,x]\|_2^2.
\]
Suppose $0<k\le1$, $\varepsilon_{\rm alg}\ge0$, and
\[
 \|x-Px\|_2\le k^{-1/2}\sqrt{\cE_V(x)}+\varepsilon_{\rm alg}
 \quad(\|x\|\le1),\qquad
 \max_j\|[V_j,b]\|_2\le\varepsilon_{\rm alg}
 \quad(b\in B_1).
\]
Let $S$ be the averaging map below, and let $S^*$ denote its
adjoint on $L^2(M_d,\tr_d)$:
\[
 S=\frac12\left(I+\frac1m\sum_j\Ad(V_j)\right),\qquad
 S^*=\frac12\left(I+\frac1m\sum_j\Ad(V_j^*)\right).
\]
For each integer $n\ge1$, set
\[
 e_n=(1-k/2)^{n/2}
       +\sqrt2\bigl(n+1+\sqrt{2/k}\bigr)\varepsilon_{\rm alg}.
\]
Then
\begin{equation}\label{eq:correctedmixing}
 \max\{\|S^n-P\|_{\infty\to2},
         \|(S^*)^n-P\|_{\infty\to2}\}\le e_n.
\end{equation}
\end{lemma}
\begin{proof}
Haar averaging over $\mathcal U(B)$ gives
$\|V_j-E_{B'}(V_j)\|_2\le\varepsilon_{\rm alg}$.
By \eqref{eq:unitaryrounding}, choose $U_j\in\mathcal U(B')$ with
$\|U_j-V_j\|_2\le\sqrt2\varepsilon_{\rm alg}$.
Let $\cE_U$ be the energy defined above with $U_j$ in place of $V_j$,
and put
\[
 \widetilde S=\frac12\left(I+\frac1m\sum_j\Ad(U_j)\right).
\]
Each $\Ad(U_j)$ and its adjoint fix $B$ pointwise, so
\begin{equation}\label{eq:exactfixed}
 P\widetilde S=\widetilde SP=P,\qquad
 \|S-\widetilde S\|_{\infty\to2}
 \le\sqrt2\varepsilon_{\rm alg}.
\end{equation}
For contractions $x$, the triangle inequality for the tuple of
commutators yields
\[
 \bigl|\sqrt{\cE_U(x)}-\sqrt{\cE_V(x)}\bigr|
 \le\sqrt2\varepsilon_{\rm alg}.
\]
Thus, with $b_{\rm mix}=(1+\sqrt{2/k})\varepsilon_{\rm alg}$,
\begin{equation}\label{eq:correctedgap}
 \|x-Px\|_2\le k^{-1/2}\sqrt{\cE_U(x)}+b_{\rm mix}
 \qquad(\|x\|\le1).
\end{equation}

The average $m^{-1}\sum_j\Ad(U_j)$ is an $L^2$ contraction, hence
\[
 \|y\|_2^2-\|\widetilde Sy\|_2^2\ge\cE_U(y).
\]
For a contraction $y$, \eqref{eq:correctedgap} also gives
$\cE_U(y)\ge(k/2)\|(I-P)y\|_2^2-kb_{\rm mix}^2$.
Using \eqref{eq:exactfixed}, we obtain
\[
 \|(I-P)\widetilde Sy\|_2^2
 \le(1-k/2)\|(I-P)y\|_2^2+kb_{\rm mix}^2.
\]
Since $\widetilde S$ preserves contractions, iteration gives
\[
 \|(I-P)\widetilde S^nx\|_2^2
 \le(1-k/2)^n+2b_{\rm mix}^2\qquad(\|x\|\le1).
\]
Consequently,
\[
 \|\widetilde S^n-P\|_{\infty\to2}
 \le(1-k/2)^{n/2}+\sqrt2b_{\rm mix}.
\]
Both $S$ and $\widetilde S$ are contractive in the operator norm
and the $2$-norm. Telescoping \eqref{eq:exactfixed} therefore gives
\[
 \|S^n-\widetilde S^n\|_{\infty\to2}
 \le n\sqrt2\varepsilon_{\rm alg}.
\]
Adding the last two bounds proves the estimate for $S^n$.
Replacing $V_j,U_j$ by their adjoints preserves both energies
and the correction bounds, proving the estimate for $(S^*)^n$.
\end{proof}

When $\varepsilon_{\rm alg}>0$, the error term
$n\varepsilon_{\rm alg}$ grows with $n$, so we choose a finite
averaging power below.

\begin{lemma}\label{lem:centeraveraging}
Let $A,B\subseteq M_d$ be unital $*$-subalgebras, let $m,n\ge1$
be integers, and let $\eta\ge0$. Suppose
$V_1,\ldots,V_m\in\mathcal U(d)$ satisfy
$V_j^*AV_j\subseteq_\eta A$ for every $j$.
For the averaging map
\[
 S=\frac12\left(I+\frac1m\sum_j\Ad(V_j)\right),
\]
suppose that, for some $e_n\ge0$,
\[
 \max\{\|S^n-E_B\|_{\infty\to2},
         \|(S^*)^n-E_B\|_{\infty\to2}\}\le e_n.
\]
Then
\begin{equation}\label{eq:centerbound}
 \|(I-E_{Z(A)})E_B|_{Z(A)}\|_{\infty\to2}
 \le2e_n+3n\eta.
\end{equation}
\end{lemma}
\begin{proof}
Every word $w$ of length $\ell\le n$ in the $V_j$, without inverses,
satisfies $w^*Aw\subseteq_{\ell\eta}A$.
Indeed, iterate the assumed inclusions, using conditional
expectations to choose contraction approximants at each step.
For $z\in Z(A)_1$ and $a\in A_1$, it follows that
\[
 [wzw^*,a]=w[z,w^*aw]w^*,\qquad
 \|[wzw^*,a]\|_2\le2\ell\eta.
\]
The map $S^n$ is a convex combination of these conjugations,
including the identity. Thus $\|[S^nz,a]\|_2\le2n\eta$.
Haar averaging over $\mathcal U(A)$ gives
\begin{equation}\label{eq:forwardcommutant}
 \|(I-E_{A'})S^nz\|_2\le2n\eta.
\end{equation}
The map $(S^*)^n$ is the corresponding average of $z\mapsto w^*zw$.
The word inclusions therefore give
\begin{equation}\label{eq:backwardalgebra}
 \|(I-E_A)(S^*)^nz\|_2\le n\eta.
\end{equation}
The mixing hypotheses imply
\[
 \|(I-E_{A'})E_Bz\|_2\le e_n+2n\eta,
 \qquad
 \|(I-E_A)E_Bz\|_2\le e_n+n\eta.
\]
Haar averaging and the bimodularity of $E_A$ show that
$E_AE_{A'}=E_{A'}E_A=E_{Z(A)}$.
Apply
\[
 I-E_{Z(A)}=(I-E_A)+E_A(I-E_{A'})
\]
and the $L^2$ contractivity of $E_A$ to the last two bounds.
\end{proof}

Thus forward averaging controls the distance to $A'$, while
backward averaging controls the distance to $A$.
Both follow from the assumed inclusions $V_j^*AV_j\subseteq_\eta A$.
The Haar expectation identities used here also appear in
\cite[proof of Lemma~4.1]{Thom}.

To obtain a central partition, let $A\subseteq M_d$ be a unital
$*$-subalgebra and let $\eta\ge0$. Suppose the hypotheses of
Lemma~\ref{lem:correctedmixing} hold and
$V_j^*AV_j\subseteq_\eta A$ for every $j$.
Choose an integer $n\ge1$ and $0<\zeta<1$ such that, with $e_n$
defined in that lemma,
\[
 2e_n+3n\eta\le b:=\zeta^{12}/2^{42}.
\]
Lemma~\ref{lem:centeraveraging} bounds the distance of
$E_B(Z(A)_1)$ to $Z(A)$ by $b$.
The compression $E_{Z(A)}E_B|_{Z(A)}$ is a Markov operator,
selfadjoint on $L^2(Z(A),\tr_d)$. As a compression of the orthogonal
projection $E_B$, it lies between $0$ and $I$ in the Hilbert-space
operator order. Its difference from its square is
\[
 E_{Z(A)}E_B(I-E_{Z(A)})E_B|_{Z(A)},
\]
whose $\infty\to2$ norm is at most $b$ by \eqref{eq:centerbound}.
Apply Lemma~\ref{thm:commrecon} to $Z(A)$, with the atom weights
given by the normalized traces of its minimal projections.
It gives a unital abelian subalgebra $C\subseteq Z(A)$ satisfying
\begin{equation}\label{eq:centralpartition}
 \|E_B|_{Z(A)}-E_C|_{Z(A)}\|_{\infty\to2}
 \le\zeta+b\le2\zeta.
\end{equation}

\subsection{Normalization of near inclusions}

We first record the finite correction estimates from
\cite[proof of Proposition~3.1]{Thom}, with relaxed constants.
They replace a near inclusion by commuting representations on a
nearby finite-dimensional space, using the unnormalized Hilbert--Schmidt
norm from Section~\ref{sec:prelim}.

\begin{corollary}\label{lem:dilation}
Suppose $E,A\subseteq M_d$ are unital $*$-subalgebras,
$E\subseteq_\nu A$, $0\le\nu\le1/8$, and
$C\subseteq E\cap A$ is a unital $*$-subalgebra.
There are a finite-dimensional Hilbert space $\mathcal K$, an isometry
$V:\C^d\to\mathcal K$, a projection $R\in B(\mathcal K)$ of rank $D$,
and commuting unital $*$-representations
\[
 \pi_E:E\to B(R\mathcal K),\qquad
 \pi_{A'}:A'\to B(R\mathcal K),
\]
where $B(R\mathcal K)$ is the algebra of operators on $R\mathcal K$.
Write $P=VV^*$ and regard operators on $R\mathcal K$ as operators on
$\mathcal K$ by extending them by zero on $(R\mathcal K)^\perp$.
There are partial isometries $W,Z:\C^d\to R\mathcal K$ such that
\begin{align}
 \|R-P\|_{\HS}&\le6\nu\sqrt d,& |D-d|&\le32\nu^2d,\label{eq:dilrank}\\
 \|W-V\|_{\HS}&\le12\nu\sqrt d,& d-\rank W&\le32\nu^2d,\label{eq:dilW}\\
 \|Z-V\|_{\HS}&\le16\nu\sqrt d,& d-\rank Z&\le64\nu^2d,\label{eq:dilZ}
\end{align}
Moreover,
\begin{equation}\label{eq:intertwining}
 Wx=\pi_E(x)W\ (x\in E),\qquad
 Zy=\pi_{A'}(y)Z\ (y\in A'),\qquad
 Zc=\pi_E(c)Z\ (c\in C).
\end{equation}
For contractions, these intertwiners give the transport bounds
\begin{equation}\label{eq:diltransport}
 \begin{aligned}
 \|\pi_E(x)-VxV^*\|_{\HS}&\le32\nu\sqrt d
       &&(x\in E,\ \|x\|\le1),\\
 \|\widehat x-VxV^*\|_{\HS}&\le64\nu\sqrt d
       &&(Zx=\widehat xZ,\ Zx^*=\widehat x^*Z),
 \end{aligned}
\end{equation}
where the second line assumes $x\in M_d$ and
$\widehat x\in B(R\mathcal K)$ are contractions.
It also applies to bounded joint functional calculi preserved by $Z$.
\end{corollary}
\begin{proof}
Apply Thom's relative correction to $E\subseteq_\nu A$ with
common subalgebra $C$. Let $\mathsf T$ be the averaged exact $E$-intertwiner
from that construction, and let $\operatorname{pol}(\mathsf T)$ denote the
partial isometry in its polar decomposition, zero on $\ker \mathsf T$.
Set $W=\operatorname{pol}(\mathsf T)$. The construction supplies $Z$ and
the projection, rank, and intertwining estimates, leaving only
the norm estimate for $W$ in \eqref{eq:dilW}.
For that estimate, the singular-value argument for polar approximation
\cite[Theorem~2.2 and Corollary~2.3]{Higham} gives
\[
 \|\operatorname{pol}(\mathsf T)-V\|_{\HS}
 \le2\|\mathsf T-V\|_{\HS}
 \le4\sqrt2\,\nu\sqrt d.
\]
The polar factor has the same rank as $\mathsf T$ and preserves its
exact intertwining identities, giving \eqref{eq:intertwining} for $W$.

For $W$, suppose $x,\widehat x$ are contractions with
$Wx=\widehat xW$ and $Wx^*=\widehat x^*W$.
The range projection $WW^*$ commutes with $\widehat x$, and
\[
 \|\widehat x-WxW^*\|_{\HS}\le\sqrt{D-\rank W},\qquad
 \|WxW^*-VxV^*\|_{\HS}\le2\|W-V\|_{\HS}.
\]
The same estimates hold with $Z$ in place of $W$ under the
corresponding intertwining identities.
Their sum, with \eqref{eq:dilrank}--\eqref{eq:dilZ}, proves
\eqref{eq:diltransport}. For commuting selfadjoint operators,
intertwining also preserves their joint functional calculus,
giving the last assertion.
\end{proof}

We next compare the matrix block sizes with their multiplicities,
following \cite[proof of Theorem~4.2]{Thom}.
Fix a unital $*$-subalgebra $C\subseteq Z(A)$ and
$\widetilde V\in\mathcal U(d)\cap C'$ such that
\[
 E=\widetilde V^*A\widetilde V\subseteq_\nu A,\qquad
 0\le\nu\le\bar\nu\le1/8,\qquad \bar\nu>0.
\]
Then $C\subseteq Z(A)\cap Z(E)$.
In their respective orthogonal decompositions, write
\[
 A=\bigoplus_j(M_{b_j}\otimes I_{s_j}),\qquad
 E=\bigoplus_i(M_{a_i}\otimes I_{r_i}),
\]
where $b_j,a_i$ are matrix block sizes and $s_j,r_i$ their
multiplicities. Let $f_j,e_i$ be the corresponding minimal central
projections and set
\[
 \Delta_A=\sum_j\frac{b_j}{s_j}f_j,\qquad
 \Delta_E=\sum_i\frac{a_i}{r_i}e_i=\widetilde V^*\Delta_A\widetilde V.
\]
Choose the positive invertible element $M_C\in C$ so that
\[
 F_A=\Delta_A(\Delta_A+M_C)^{-1},\qquad E_C(F_A)=\tfrac12 I,
 \qquad F_E=\widetilde V^*F_A\widetilde V.
\]
This conditional median exists uniquely: on each minimal projection
of $C$, the conditional average $E_C(F_A)$ decreases continuously
and strictly from $1$ to $0$ as the positive scalar value of $M_C$
increases from $0$ to infinity.
Since $\widetilde V$ commutes with $M_C$, also
$F_E=\Delta_E(\Delta_E+M_C)^{-1}$.

Apply Corollary~\ref{lem:dilation}.
Let $r_i^\sharp$ be the multiplicity of $M_{a_i}$ in $\pi_E(E)$,
and let $b_j^\sharp$ be the multiplicity of $M_{s_j}$ in
$\pi_{A'}(A')$; these multiplicities may be zero.
The $E$-intertwiner $W$ leaves at most $d+D-2\rank W$
unmatched dimensions. For the $A'$-intertwiner $Z$, the corresponding
bound is $d+D-2\rank Z$.
Thus \eqref{eq:dilrank}--\eqref{eq:dilZ} give, as in the estimates
preceding \cite[(3.3)]{Thom},
\begin{equation}\label{eq:multedit}
 \sum_i a_i|r_i^\sharp-r_i|\le160\nu^2d,\qquad
 \sum_j s_j|b_j^\sharp-b_j|\le160\nu^2d.
\end{equation}
Set $A^\sharp=\pi_{A'}(A')'$, with the commutant taken in
$B(R\mathcal K)$. The commuting representations give
$\pi_E(E)\subseteq A^\sharp$.
For this inclusion, let $m_{ij}$ be the multiplicity of the
$M_{a_i}$ block in the $M_{b_j^\sharp}$ block of $A^\sharp$.
Then \cite[(4.2)--(4.3)]{Thom} gives
\begin{equation}\label{eq:branching}
 b_j^\sharp=\sum_i m_{ij}a_i,\qquad
 r_i^\sharp=\sum_j m_{ij}s_j,\qquad
 \rank\bigl(\pi_E(e_i)\pi_{A'}(f_j)\bigr)=a_im_{ij}s_j.
\end{equation}
On $R\mathcal K$, define
\[
 \Delta^- =\sum_{r_i^\sharp>0}\frac{a_i}{r_i^\sharp}\pi_E(e_i),\qquad
 \Delta^+ =\sum_{b_j^\sharp>0}\frac{b_j^\sharp}{s_j}\pi_{A'}(f_j),\qquad
 \widehat M=\pi_E(M_C).
\]
These operators commute and are positive invertible on $R\mathcal K$.
Equation~\eqref{eq:branching} gives $\Delta^-\le\Delta^+$, hence
\begin{equation}\label{eq:liftedtransforms}
 0<\widehat F^-:=\Delta^-(\Delta^-+\widehat M)^{-1}
 \le\widehat F^+:=\Delta^+(\Delta^++\widehat M)^{-1}<R.
\end{equation}

\begin{corollary}\label{lem:medianmatching}
In this setup,
\begin{equation}\label{eq:medianinvariance}
 \|\widetilde V^*F_A\widetilde V-F_A\|_2\le128\sqrt{\bar\nu}.
\end{equation}
\end{corollary}
\begin{proof}
The scalar comparisons in \cite[(4.5)--(4.6)]{Thom}, together with
\eqref{eq:multedit}, bound each squared Hilbert--Schmidt error caused
by changing the multiplicities by $160\nu^2d$.
Adding the transport bounds \eqref{eq:diltransport} yields
\begin{equation}\label{eq:transformtransport}
 \|\widehat F^- -VF_EV^*\|_{\HS}\le80\nu\sqrt d,\qquad
 \|\widehat F^+ -VF_AV^*\|_{\HS}\le80\nu\sqrt d.
\end{equation}
Since $\Tr F_E=\Tr F_A$, Cauchy--Schwarz on
$R\mathcal K+P\mathcal K$, whose dimension is at most
$D+d\le5d/2$, gives
\[
 0\le\Tr(\widehat F^+-\widehat F^-)\le320\nu d.
\]
The difference is a positive contraction by
\eqref{eq:liftedtransforms}, so its squared Hilbert--Schmidt norm
is at most its trace. Thus
\[
 \|F_E-F_A\|_2\le160\nu+\sqrt{320\nu}
 \le128\sqrt{\bar\nu}.\qedhere
\]
\end{proof}

\begin{corollary}\label{lem:medianreverse}
In the same setup,
\begin{equation}\label{eq:reversenear}
 \sup_{a\in A,\ \|a\|\le1}\|a-E_E(a)\|_2
 \le64\bar\nu+
 8\left(\frac{\max_\pm\|\widehat F^\pm-R/2\|_{\HS}}{\sqrt d}\right)^{1/2}.
\end{equation}
Here $E_E$ is the trace-preserving conditional expectation onto $E$,
so $E_E(a)$ is a contraction whenever $a$ is.
\end{corollary}
\begin{proof}
For the exact inclusion $\pi_E(E)\subseteq A^\sharp$, set
$\Theta=\Delta^+(\Delta^-)^{-1}\ge R$.
Thom's block-ratio estimate \cite[Lemma~4.1]{Thom}, written with
the unnormalized trace on $R\mathcal K$, gives
\[
 \sup_{x\in A^\sharp,\ \|x\|\le1}
 \|x-E_{\pi_E(E)}(x)\|_{\HS}^2
 \le2\Tr(R-\Theta^{-1}).
\]
All inverses here are taken on $R\mathcal K$.
The commuting transforms satisfy
\[
 \Theta^{-1}
 =\frac{\widehat F^-(R-\widehat F^+)}
        {\widehat F^+(R-\widehat F^-)}.
\]
Apply the scalar inequality
\[
 0\le1-\frac{a(1-b)}{b(1-a)}
 \le8\bigl(|a-\tfrac12|+|b-\tfrac12|\bigr)
 \qquad(0<a\le b<1).
\]
Since $D\le3d/2$, Cauchy--Schwarz gives
\[
 \Tr(R-\Theta^{-1})\le32\sqrt d\,
             \max_\pm\|\widehat F^\pm-R/2\|_{\HS}.
\]
Consequently the corrected reverse error, divided by $\sqrt d$,
is bounded by the second term in \eqref{eq:reversenear}.

For a contraction $a\in A$, the intertwining identities imply
$ZaZ^*\in A^\sharp$, and
\[
 \|ZaZ^*-VaV^*\|_{\HS}\le2\|Z-V\|_{\HS}\le32\nu\sqrt d.
\]
Its conditional-expectation approximant in $\pi_E(E)$ lifts to
a contraction $e\in E$; choose zero on any block killed by $\pi_E$.
Equation~\eqref{eq:diltransport} gives
$\|\pi_E(e)-VeV^*\|_{\HS}\le32\nu\sqrt d$.
These two transport errors add $64\nu$ after division by $\sqrt d$.
Finally, $E_E$ is the orthogonal projection onto $E$ in $L^2$, so
$\|a-E_E(a)\|_2\le\|a-e\|_2$.
Using $\nu\le\bar\nu$ proves \eqref{eq:reversenear}.
\end{proof}

We now combine the preceding estimates. Fix
\begin{equation}\label{eq:normalizationparameters}
 k=2^{-9},\qquad \varepsilon_{\rm alg}=2^{-1664},\qquad
 \eta=2^{-1642}.
\end{equation}
For a unitary tuple $V_1,\ldots,V_m$, let $\cE_V$ be the commutator
energy defined in Lemma~\ref{lem:correctedmixing}.

\begin{theorem}\label{thm:normalization}
Let $d,m\ge1$ be integers, let $A,B\subseteq M_d$ be unital $*$-subalgebras,
and let $V_1,\ldots,V_m\in\mathcal U(d)$.
Suppose
\begin{align}
 \|x-E_Bx\|_2
 &\le k^{-1/2}\sqrt{\cE_V(x)}+\varepsilon_{\rm alg}
 &&(x\in M_d,\ \|x\|\le1),\label{eq:normalizationhyp}\\
 \max_{1\le j\le m}\|[V_j,b]\|_2
 &\le\varepsilon_{\rm alg}
 &&(b\in B,\ \|b\|\le1).\notag
\end{align}
If $V_j^*AV_j\subseteq_\eta A$ for every $j$, then
\begin{equation}\label{eq:normalizationconclusion}
 V_jAV_j^*\subseteq_{2^{-20}}A\qquad(1\le j\le m).
\end{equation}
\end{theorem}
\begin{proof}
Choose
\begin{equation}\label{eq:normalizationledger}
 \zeta=2^{-128},\qquad
 b=\zeta^{12}/2^{42}=2^{-1578},\qquad n=2^{22}.
\end{equation}
Since $\sqrt{2/k}=32$ and $\sqrt2(n+33)<2^{24}$,
Lemma~\ref{lem:correctedmixing} gives
\[
 (1-k/2)^{n/2}\le2^{-2048},\qquad
 e_n\le2^{-2048}+2^{-1640}.
\]
Hence
\[
 2e_n+3n\eta
 \le2^{-2047}+2^{-1639}+3\cdot2^{-1620}<b.
\]
The central partition construction \eqref{eq:centralpartition}
therefore gives a unital $*$-subalgebra $C\subseteq Z(A)$ with
\[
 \|E_B|_{Z(A)}-E_C|_{Z(A)}\|_{\infty\to2}\le2\zeta.
\]
For a contraction $c\in C$, we have $E_Cc=c$ and
$\|c-E_Bc\|_2\le2\zeta$. Thus
\[
 \|[V_j,c]\|_2
 \le\|[V_j,E_Bc]\|_2+2\|c-E_Bc\|_2
 \le\varepsilon_{\rm alg}+4\zeta\le5\zeta.
\]
Haar averaging over $\mathcal U(C)$ and
\eqref{eq:unitaryrounding} give $\widetilde V_j\in\mathcal U(C')$ satisfying
\[
 \|\widetilde V_j-V_j\|_2\le5\sqrt2\zeta<8\zeta.
\]
Put $\bar\nu=256\zeta=2^{-120}$.
Since $\eta+16\zeta\le\bar\nu$, we obtain
$\widetilde V_j^*A\widetilde V_j\subseteq_{\bar\nu}A$ for every $j$.

Choose the conditional median $M_C$ of $\Delta_A$ once for the
whole tuple, and put $F_A=\Delta_A(\Delta_A+M_C)^{-1}$.
Thus $F_A\in Z(A)$, $0\le F_A\le1$, and $E_CF_A=I/2$.
Corollary~\ref{lem:medianmatching}, followed by the unitary correction,
gives
\[
 \|V_j^*F_AV_j-F_A\|_2
 \le128\sqrt{\bar\nu}+16\zeta\le2^{-52}.
\]
Consequently $\sqrt{\cE_V(F_A)}\le2^{-53}$.
Apply \eqref{eq:normalizationhyp} and the central partition bound:
\[
 \|F_A-I/2\|_2
 \le k^{-1/2}2^{-53}+\varepsilon_{\rm alg}+2\zeta<2^{-48}.
\]

Fix $j$ and apply Corollary~\ref{lem:dilation} to
$E=\widetilde V_j^*A\widetilde V_j\subseteq_{\bar\nu}A$, with the same $C$.
Write $V,R,P$ for its correction data and let $\widehat F^\pm$
be the associated transforms in \eqref{eq:liftedtransforms}.
Since $F_E=\widetilde V_j^*F_A\widetilde V_j$, its distance to $I/2$ equals that of $F_A$.
Equations~\eqref{eq:transformtransport} and~\eqref{eq:dilrank} give
\[
\begin{aligned}
 \frac{\max_\pm\|\widehat F^\pm-R/2\|_{\HS}}{\sqrt d}
 &\le80\bar\nu+2^{-48}
          +\frac{\|P-R\|_{\HS}}{2\sqrt d}\\
 &\le83\bar\nu+2^{-48}\le2^{-47}.
\end{aligned}
\]
Corollary~\ref{lem:medianreverse} now gives
$A\subseteq_{64\bar\nu+8\sqrt{2^{-47}}}E$.
Conjugating by $\widetilde V_j$ gives
$\widetilde V_jA\widetilde V_j^*\subseteq_{64\bar\nu+8\sqrt{2^{-47}}}A$.
Replacing $\widetilde V_j$ by $V_j$ adds at most $16\zeta$ to this error. Thus
\begin{equation}\label{eq:reverseledger}
 V_jAV_j^*\subseteq_{2^{-20}}A,\qquad
 64\bar\nu+8\sqrt{2^{-47}}+16\zeta<2^{-20}.
\end{equation}
For the last inequality, use $\sqrt{2^{-47}}\le3\cdot2^{-25}$ and
$2^{-114}+3\cdot2^{-22}+2^{-124}<2^{-20}$.
\end{proof}

\section{Reflection rigidity and trace separation}\label{sec:gate}

This section proves that $f$ has normalized trace at least $3/4$
on every tuple of six selfadjoint matrix contractions, completing
Theorem~\ref{thm:main}. We first combine the spectral estimates of
Section~\ref{sec:certificates} with the normalization theorem of
Section~\ref{sec:reconstruction} to show that approximate unitary
solutions of the 97 full relations force the commutator in the
signal to be small. We then recover these relations from the
fifteen retained equations in \eqref{csimp:r3:fifteen}.

For arbitrary selfadjoint contractions, a small total penalty allows
us to round the variables to symmetries and apply this reflection
estimate. A large total penalty gives the required lower bound
directly. Together with the negative group evaluation from
Section~\ref{sec:polynomial}, this rules out a representation of $f$
as an element of the contraction quadratic module plus a sum of
additive commutators.

\subsection{The reflection estimate}
We use $k,\varepsilon_{\rm alg},\eta$ from
\eqref{eq:normalizationparameters}, put $\delta=2^{-J}$ with $J$ as in
\eqref{eq:globalconstants}, and set
\begin{equation}\label{eq:outerparameters}
 \zeta_{\rm alg}=2^{-1700},\qquad
 \theta_{\rm sp}=\zeta_{\rm alg}^8/2^{96}=2^{-13696}.
\end{equation}
The next lemma shows that this relation tolerance is small enough
for both spectral inputs of Section~\ref{sec:certificates}.

\begin{lemma}\label{lem:tower}
For either input $(m,\lambda,C)$ in \eqref{eq:actualinputs},
\[
 \delta\le\eps_{\rm in}(m,\lambda,C,\theta_{\rm sp}).
\]
\end{lemma}
\begin{proof}
For both inputs, $m\le2^{14}$, $2^{-9}\le\kappa=\lambda/2\le1/4$,
and $C\le2^{80}$.
Substitute \eqref{eq:outerparameters} into
\eqref{eq:inputparameters1}--\eqref{eq:inputparameters2}.
Since $d_2(16/\theta_{\rm sp})=13700<2^{14}$, we obtain
\begin{center}
\begin{tabular}{ll@{\qquad}ll}
\toprule
Quantity & Bound & Quantity & Bound\\\midrule
$c$&$\ge2^{-46}$&$N$&$\le2^{60}$\\
$r_I$&$\ge2^{-13760}$&$C_0$&$\le2^{38}$\\
$\kappa'$&$\ge2^{-19}$&$K_0$&$\le2^{42}$\\
$K_b$&$\le2^{50}$&$K_\gamma$&$\le2^{82}$\\
$b_I$&$\ge2^{-27573}$&$a_I$&$\ge2^{-27543}$\\
$t_I$&$\le2^{55163}$&&\\\bottomrule
\end{tabular}
\end{center}
Put $T_0=2^{55200}$.
The last term in the minimum defining $\gamma$ is at least
$2^{-(55124+8\cdot2^{55163})}$, and
\[
 55124+8\cdot2^{55163}=55124+2^{55166}<16T_0.
\]
The other terms are also at least $2^{-16T_0}$. Successive substitution
therefore gives
\[
 \gamma\ge2^{-16T_0},\quad v_I\ge2^{-256T_0},\quad
 \alpha\ge2^{-512T_0},\quad\beta\ge2^{-2048T_0},\quad
 L\le2^{4096T_0},\quad\mu_I\ge2^{-16384T_0}.
\]
For the parameter $\rho_{\rm filt}$ in \eqref{eq:inputparameters2},
$1-\rho_{\rm filt}=\kappa(2-3\kappa)/(1-\kappa)^2\ge\kappa\ge2^{-9}$, so
\[
 d_2(2/(\kappa\mu_I))\le16384T_0+10\le2^{15}T_0,
 \qquad m_{\rm cut}\le2^{24}T_0.
\]
Using $1-\kappa\ge1/2$ and $C+1\le2^{81}$ in the definition of
$\eps_{\rm in}$ now yields
\begin{equation}\label{eq:newtolerance}
 \varepsilon_{\rm in}\ge2^{-(3\cdot2^{24}T_0+81)}
 \ge2^{-2^{55226}}=\delta.
\end{equation}
The resulting algebra errors also satisfy
\[
 3\zeta_{\rm alg}\le\varepsilon_{\rm alg},\qquad
 6\sqrt m\,\zeta_{\rm alg}\le\varepsilon_{\rm alg},
\]
as required by \eqref{eq:normalizationparameters}.
\end{proof}

The defect of a full relation is the normalized $2$-norm of its
left side minus its right side at the given matrix values.
We now apply the preceding estimates to the 97 relations of
Section~\ref{subsec:fullrelations}.

\begin{theorem}\label{thm:gate}\label{lem:gateslack}
Let $d\ge1$, and suppose the entries in the full system have unitary
values in $M_d$, with $v$ a symmetry. If, for some $0<\eps<\delta$, every full relation
has defect strictly less than $\eps$, then, with $u=U_{12}$ and $h=z_2$,
\begin{equation}\label{eq:gate}
 \|[h,uvu^*]\|_2\le g_0:=5\cdot2^{-20}.
\end{equation}
\end{theorem}
\begin{proof}
Lemma~\ref{lem:tower} allows us to apply
Proposition~\ref{prop:Htransfer}, \eqref{eq:Kquadratic}, and
Theorem~\ref{thm:interface} to the two inputs
\[
 (m,\lambda,C)=(375,1/96,2^{80}),\qquad(6,1/144,26).
\]
Using $\kappa\ge k$ and the algebra error bounds above, we obtain
unital $*$-subalgebras $A,B\subseteq M_d$ with
\[
\begin{aligned}
 \|a-E_Aa\|_2&\le k^{-1/2}\sqrt{\mathcal E_H(a)}
                          +\varepsilon_{\rm alg},\\
 \|a-E_Ba\|_2&\le k^{-1/2}\sqrt{\mathcal E_K(a)}
                          +\varepsilon_{\rm alg}
 \qquad(\|a\|\le1).
\end{aligned}
\]
The unit balls of $A$ and $B$ commute, within $\varepsilon_{\rm alg}$,
with the five basic entries $x,y,z_1,z_2,z_3$ and the six action
unitaries, respectively. Here we used
$3\zeta_{\rm alg}\le\varepsilon_{\rm alg}$ and
$6\sqrt{375}\,\zeta_{\rm alg}\le\varepsilon_{\rm alg}$.

The crossing relations put $d_i$ and $e_i$ within $\eps$ of
$\gc{x}{z_i}$ and $\gc{y}{z_i}$, respectively. Their commutators
with $A_1$ therefore have norm at most
$4\varepsilon_{\rm alg}+2\eps$; those of each target $\pi_{ab}$ have
norm at most $16\varepsilon_{\rm alg}+8\eps$.
For $V=U_{ab}$, the compression relations give
\[
 \max_{g\in\{x,y,z_1,z_2,z_3\}}\|[V^*aV,g]\|_2
 \le16\varepsilon_{\rm alg}+10\eps
 \le28\varepsilon_{\rm alg}
 \qquad(a\in A_1),
\]
since $\eps<\delta<\varepsilon_{\rm alg}$.
The words in $\mathscr L_H$ have length at most twelve, hence
\[
\begin{aligned}
 \sqrt{\mathcal E_H(V^*aV)}&\le168\varepsilon_{\rm alg},\\
 \|V^*aV-E_A(V^*aV)\|_2
 &\le168\varepsilon_{\rm alg}/\sqrt k+\varepsilon_{\rm alg}
 \le\eta.
\end{aligned}
\]
Thus $V^*AV\subseteq_\eta A$ for every action unitary $V$.
Theorem~\ref{thm:normalization}, applied to $A,B$ and the six
$U_{ab}$, gives the reverse inclusions
$VAV^*\subseteq_{2^{-20}}A$.

The centralization relations and the same averaging list give
\[
 \|v-E_Av\|_2
 \le6\eps/\sqrt k+\varepsilon_{\rm alg}
 \le512\varepsilon_{\rm alg}/k.
\]
Since $E_Av\in A_1$, the reverse inclusion for $u=U_{12}$ puts
$uvu^*$ within $512\varepsilon_{\rm alg}/k+2^{-20}$ of $A_1$.
Every element of $A_1$ commutes with $h=z_2$ within
$\varepsilon_{\rm alg}$, so
\[
 \|[h,uvu^*]\|_2
 \le\varepsilon_{\rm alg}+1024\varepsilon_{\rm alg}/k+2^{-19}
 \le5\cdot2^{-20}=g_0.\qedhere
\]
\end{proof}

For a unitary $h$ and a symmetry $W$, the identity
\[
 \tr_d(-1+h^*WhW+Wh^*Wh)=1-\|[h,W]\|_2^2
\]
converts the reflection estimate into a lower bound for the signal.
We will use it after establishing the unitary and symmetry conditions
for the rounded variables.

\subsection{Finite transport and recovery}

We first record three word estimates used to recover the 97 full
relations. All defects are measured in the normalized $2$-norm.
Indices in the indexed relations below range over $\{1,2,3\}$.

\begin{lemma}\label{lem:reciprocity}
Let $a,b,c\in\mathcal U(d)$. Then
\[
 \|[\gc{a}{c},b]\|_2
 \le2\|[b,c]\|_2+2\|[\gc{a}{b},a]\|_2
       +\|[\gc{a}{b},c]\|_2.
\]
\end{lemma}
\begin{proof}
The identity $aba^*=\gc{a}{b}b$ and the product commutator rule give
\[
 \|[b,aca^*]\|_2
 \le\|[b,c]\|_2+2\|[\gc{a}{b},a]\|_2
       +\|[\gc{a}{b},c]\|_2.
\]
Indeed, expand $[\gc{a}{b}b,aca^*]$ and use
$\|[aba^*,aca^*]\|_2=\|[b,c]\|_2$.
Expanding $[\gc{a}{c},b]$ with $\gc{a}{c}=(aca^*)c^*$ adds
at most $\|[b,c]\|_2$.
\end{proof}

The next lemma recovers the three braid relations from the other
twelve integer relations.

\begin{lemma}\label{lem:braid}
Let $U_{ab}\in\mathcal U(d)$ for $a\ne b$, and let $\iota\ge0$.
Suppose the six row and column commutations and the six Steinberg
relations in \eqref{eq:s2:integer-relations} have defects at most
$\iota$. Then
\[
 \|U_{ij}U_{ji}^*U_{ij}-U_{ji}^*U_{ij}U_{ji}^*\|_2
 \le27\iota\qquad(i<j).
\]
\end{lemma}
\begin{proof}
Fix distinct $i,j,k$ and put
\[
 a=U_{ij},\quad b=U_{ji},\quad p=U_{ik},\quad q=U_{jk},\quad
 r=U_{ki},\quad s=U_{kj},\qquad w=ab^*a.
\]
The assumed relations give
\[
\begin{gathered}
 \|apa^*-p\|_2\le\iota,\qquad
 \|aqa^*-pq\|_2\le\iota,\qquad
 \|b^*pb-q^*p\|_2\le2\iota,\\
 \|asa^*-s\|_2\le\iota,\qquad
 \|b^*sb-rs\|_2\le2\iota,\qquad
 \|ara^*-s^*r\|_2\le\iota,\qquad
 \|[r,s]\|_2\le\iota.
\end{gathered}
\]
For the two bounds involving $b^*$, rearrange $bp=qpb$ and
$sb=rbs$, respectively, and use one row or column commutation.
The bound for $ara^*$ follows by rearranging $ra=sar$.
Substitution into $wpw^*$ and $wsw^*$ now gives
\[
 \|wpw^*-q^*\|_2\le(1+2+2)\iota=5\iota,
 \qquad
 \|wsw^*-r\|_2\le(1+2+2+1)\iota\le7\iota.
\]
Since $\|a-\gc{p}{s}\|_2\le\iota$, replacing the four factors
in the conjugated commutator yields
\[
 \|waw^*-\gc{q^*}{r}\|_2
 \le(1+2\cdot5+2\cdot7)\iota=25\iota.
\]
The identity
$\gc{q^*}{r}=q^*\gc{q}{r}^*q$, together with the defects of
$qr=brq$ and $bq=qb$, gives
$\|\gc{q^*}{r}-b^*\|_2\le2\iota$.
Thus $\|waw^*-b^*\|_2\le27\iota$.
Multiplication on the right by $wa^*$ shows that this is precisely
the required bound for $\|ab^*a-b^*ab^*\|_2$.
\end{proof}

The following five-step calculation transports the remaining target.

\begin{lemma}\label{lem:integertransport}
Let the six $U_{ij}$ be unitaries whose twelve nonbraid integer
relations in \eqref{eq:s2:integer-relations} have defects at most
$\iota\ge0$. Put $K_{23}=U_{23}^*U_{32}$. Then
\begin{equation}\label{eq:integerconjugacy}
 \|K_{23}U_{21}K_{23}^*-U_{31}\|_2\le5\iota.
\end{equation}
\end{lemma}
\begin{proof}
Use the following five steps:
\[
\begin{aligned}
U_{23}^*U_{32}U_{21}U_{32}^*U_{23}
 &\longrightarrow U_{23}^*U_{31}U_{21}U_{23}\\
 &\longrightarrow U_{23}^*U_{31}U_{23}U_{21}\\
 &\longrightarrow U_{23}^*U_{21}^*U_{23}U_{31}U_{21}\\
 &\longrightarrow U_{21}^*U_{31}U_{21}\\
 &\longrightarrow U_{31}.
\end{aligned}
\]
Each step has error at most $\iota$. In order, the steps use
$U_{32}U_{21}=U_{31}U_{21}U_{32}$,
$[U_{23},U_{21}]=0$,
$U_{23}U_{31}=U_{21}U_{31}U_{23}$,
$[U_{23},U_{21}]=0$, and $[U_{21},U_{31}]=0$,
with their stated defects. Telescoping proves the bound.
\end{proof}

We now recover the full relations from the symmetry dictionary.
First take unitaries $x,a,t,z$ and a symmetry $q$ satisfying
$qtq=t^*$ and $qzq=z^*$, and form the words in
\eqref{csimp:r17:eq:dict}. At this stage $x$ is independent;
the additional representation $x=szs$ is introduced below.
Use the full readouts $y,c,d_i,e_i,U_{ij}$ in
\eqref{eq:s2:full-readouts}, evaluated at this independent $x$.
In particular, $d_1=d$. When $t^6=1$, we have
$r^3=1$, $v^2=1$, and $[v,r]=[v,q]=0$.
The intermediate system consists of seventeen equations:
\begin{align}
 &t^6=x^5=z^5=1;\label{csimp:r17:eq:orders}\\
 &[t,x]=[a,x]=0;\label{csimp:r17:eq:basiccomm}\\
 &[z,z_2]=[c,x]=[d,x]=[d,z]=[d,z_2]=0;\label{csimp:r17:eq:roots}\\
 &[a,b]=[b,a_+]=0,\qquad aa_+=ba_+a;\label{csimp:r17:eq:integer}\\
 &[a,z]=[a,z_3]=0,\qquad az_2a^*=p;\label{csimp:r17:eq:compress}\\
 &[v,z]=0.\label{csimp:r17:eq:central}
\end{align}
These equations recover the 97 full tests; the polynomial uses only
the fifteen equations \eqref{csimp:r3:fifteen}.
We will also use the exact identities
\begin{equation}\label{csimp:r17:eq:pflip}
 \ell p\ell=p^*,\qquad qp^*q=r^*pr,
\end{equation}
which follow from $\ell^2=1$ and
$p=d^*\ell d^*\ell d\ell d\ell$.
They supply the compression in the reverse direction.

\begin{lemma}\label{csimp:r17:lem:core}
Let $x,a,t,z$ be unitary and $q$ a symmetry, with
$qtq=t^*$, $qzq=z^*$, and $t^6=1$ exactly.
Form the readouts in \eqref{csimp:r17:eq:dict} and
\eqref{eq:s2:full-readouts}.
If the remaining sixteen equations
\eqref{csimp:r17:eq:orders}--\eqref{csimp:r17:eq:central}
have $2$-norm defects at most $e\ge0$, then all 97 full-system
defects are at most $2^{13}e$.
\end{lemma}
\begin{proof}
Conjugation by $r$ cycles $(z_1,z_2,z_3)$ and the two action
triples $(a,a_+,U_{31})$ and $(b,U_{21},U_{32})$ exactly.
Conjugation by $q$ swaps indices $2$ and $3$ in the action entries,
inverts $z_i$ with the same index permutation, and interchanges $x,y$.
Thus the two commutations in \eqref{csimp:r17:eq:integer} give
all six row and column commutations by conjugation with $r$.
The crossing gives three crossings by $r$ and the other three by $q$,
all with defects at most $e$.
Lemma~\ref{lem:braid} bounds the three braid defects by $27e$.

Since $r=t^2$ and $v=t^3$, the bound $\|[t,x]\|_2\le e$ gives
\[
 \|[r,x]\|_2,\ \|[r,y]\|_2\le2e,\qquad
 \|[v,x]\|_2,\ \|[v,y]\|_2\le3e.
\]
The commutator product rule then gives
\[
 \|[a_+,x]\|_2,\ \|[U_{31},x]\|_2\le5e,\qquad
 \|[b,x]\|_2\le14e,\qquad
 \|[U_{21},x]\|_2,\ \|[U_{32},x]\|_2\le18e.
\]
Here the estimate for $b$ uses
$\|b-\gc{a}{a_+}\|_2\le e$.
Conjugating by $q$ gives the corresponding bounds with $y$.
Every action therefore commutes with $x,y$ within $18e$.

The definitions of $c,d_i,e_i$ make all seven crossing equations exact.
The relation $[z_1,z_2]=0$ and its cyclic conjugates bound every
$[z_i,z_j]$ by $e$. Conjugating $d_i$ by $r$ or $r^*$ changes it
from the corresponding $d_j$ by at most $4e$.
Transporting the retained commutations consequently gives
\[
 \|[d_i,x]\|_2\le13e,\qquad \|[d_i,z_i]\|_2\le9e.
\]
For the other order of the indices, Lemma~\ref{lem:reciprocity}
with $(a,b,c)=(x,z_1,z_2)$ gives
\[
 \|[d_2,z_1]\|_2
 \le2\|[z_1,z_2]\|_2+2\|[d_1,x]\|_2+\|[d_1,z_2]\|_2
 \le5e.
\]
Cyclic transport now bounds every $[d_i,z_j]$ by $13e$,
and the product rule bounds every $[d_i,d_j]$ by $52e$.

The exact identities
\[
 qd_1q=z_1^*e_1^*z_1,\qquad
 qd_2q=z_3^*e_3^*z_3,\qquad
 qd_3q=z_2^*e_2^*z_2
\]
and the same-index bounds place $e_1,e_3,e_2$ within $9e$ of
$qd_1^*q,qd_2^*q,qd_3^*q$, respectively.
Hence $[e_i,y]$ and $[e_i,z_j]$ have norms at most $31e$,
and the product rule bounds $[e_i,e_j]$ by $124e$.
Since $qcq=c^*$, the bound on $[c,x]$ also bounds $[c,y]$.
All five basic fifth-power defects are at most $e$, by conjugation
from $x^5$ and $z^5$. Thus the 47 tests involving these elements
have defects at most $124e$.

For the compression equations, define the temporary target words
\[
 p_{12}=az_2a^*,\qquad p_{23}=rp_{12}r^*,\qquad p_{13}=r^*p_{12}r,
 \qquad p_{ji}=p_{ij}\quad(i<j).
\]
The readouts and $r^3=1$ give the exact identities
\[
 p_{12}=U_{12}z_2U_{12}^*,\qquad
 p_{23}=U_{23}z_3U_{23}^*,\qquad
 p_{13}=U_{31}z_1U_{31}^*.
\]
These targets will be compared with the prescribed words
$\pi_{ij}=d_i^*e_jd_ie_j^*$ for $i<j$.
The commutations $[U_{ij},z_k]=0$ for $k\ne j$ follow from
\eqref{csimp:r17:eq:compress} by conjugation with $r,q$,
with defects at most $e$.
Since $\|p_{12}-p\|_2\le e$, the identities
\eqref{csimp:r17:eq:pflip} give
\[
 \|bz_3b^*-r^*p_{12}r\|_2\le2e.
\]
Its cyclic conjugates give the other two reverse compressions.
The three forward compressions are exact by the choice of $p_{ij}$.
Thus all six compressions to temporary targets have defects at most $2e$.

The retained estimate $\|[d,z]\|_2\le e$ gives
$\|qd_1^*q-e_1\|_2\le e$, and conjugation by $r$ transports
$e_1$ to $e_2$ within $4e$.
Therefore $\|\ell d^*\ell-e_2\|_2\le5e$, and
\[
 \|p_{12}-\pi_{12}\|_2
 =\|p_{12}-d_1^*e_2d_1e_2^*\|_2\le11e.
\]
Conjugating this comparison by $r$ adds at most $16e$, so
$\|p_{23}-\pi_{23}\|_2\le27e$.
For $K_{23}=U_{23}^*U_{32}$ from Lemma~\ref{lem:integertransport},
the reverse compression to $p_{23}$ and the exact forward compression give
\[
 \|K_{23}z_2K_{23}^*-z_3\|_2\le2e.
\]
The two fixed-element commutations for $z_1$ also give
$\|K_{23}z_1K_{23}^*-z_1\|_2\le2e$.
Moreover, $\|[K_{23},x]\|_2,\|[K_{23},y]\|_2\le36e$.
Since the crossings defining $d_1$ and $e_2$ are exact, editing
the two occurrences of each factor yields
\[
 \|K_{23}d_1K_{23}^*-d_1\|_2,\quad
 \|K_{23}e_2K_{23}^*-e_3\|_2
 \le(2\cdot36+2\cdot2)e=76e.
\]
Lemma~\ref{lem:integertransport}, with $\iota=e$, gives
$\|K_{23}U_{21}K_{23}^*-U_{31}\|_2\le5e$.
Now replace $p_{12}$ by $U_{21}z_1U_{21}^*$ at cost at most $2e$,
use the bound for $z_1$, and the exact identity $p_{13}=U_{31}z_1U_{31}^*$:
\[
 \|K_{23}p_{12}K_{23}^*-p_{13}\|_2
 \le(2+2\cdot5+2)e=14e.
\]
Transporting the seed $\|p_{12}-\pi_{12}\|_2\le11e$ therefore gives
\[
 \|p_{13}-\pi_{13}\|_2
 \le(14+11+4\cdot76)e=329e.
\]
Adding the reverse compression error gives at most $331e$.
The other prescribed target pairs cost at most $13e$ and $29e$.
Finally, $v$ commutes with $q,r$ exactly, so the five centralization
defects are at most $3e$.
The root, integer, fixed-element, and centralization defects are bounded
by $124e$, $27e$, $18e$, and $3e$, respectively. Thus every full
residual is at most $331e\le2^{13}e$, as claimed.
\end{proof}

We now impose $x=szs$ to recover three of the intermediate tests.
Let $s,q$ be symmetries and $z$ a unitary satisfying
\begin{equation}\label{csimp:r3:rootaction}
 qzq=z^*,\qquad(qs)^3=1,\qquad x=szs,\qquad y=qxq.
\end{equation}
Then $sxs=z$ and $sys=y^*$, the latter following from
$sqs=qsq$ and $qzq=z^*$.
As above, write $d=\gc{x}{z}$ and $c=\gc{x}{y}$.

\begin{lemma}\label{csimp:r3:recoverthree}
Under \eqref{csimp:r3:rootaction},
\[
 \|x^5-1\|_2=\|z^5-1\|_2,\qquad
 \|[d,z]\|_2=\|[d,x]\|_2,\qquad
 \|[c,x]\|_2\le3\|[d,x]\|_2.
\]
Suppose also that $a,t$ are unitaries with $qtq=t^*$ and $t^6=1$,
and form the readouts and action entries as in
Lemma~\ref{csimp:r17:lem:core}.
If the fifteen equations in \eqref{csimp:r3:fifteen} have defects
at most $e\ge0$, then the intermediate seventeen equations have
defects at most $3e$, and the 97 full defects are at most
$3\cdot2^{13}e$.
\end{lemma}
\begin{proof}
Conjugation by $s$ gives the first equality, and
$sds=d^*$ gives the second.
Set $j=qsq=sqs$ and $f_0=jdj$. Then
\[
 jxj=x^*,\qquad jyj=z^*,\qquad jzj=y^*,\qquad
 f_0=\gc{x^*}{y^*}=(xy)^*c(xy).
\]
It follows that
\[
 \|[f_0,x]\|_2=\|[d,x]\|_2,\qquad
 \|[f_0,y]\|_2=\|[d,z]\|_2=\|[d,x]\|_2.
\]
Conjugating by $xy$ and applying the product rule now gives
\[
 \|[c,x]\|_2
 =\|[f_0,y^*xy]\|_2
 \le\|[f_0,x]\|_2+2\|[f_0,y]\|_2
 =3\|[d,x]\|_2.
\]
These estimates recover $x^5=1$, $[d,z]=0$, and $[c,x]=0$
from the fifteen retained tests. The extra retained equation
$(qs)^3=1$ supplies the symmetry used here, and
Lemma~\ref{csimp:r17:lem:core} gives the full-system bound.
\end{proof}

\subsection{Rounding and matrix positivity}

The relation penalties control both the distance to a tuple of
symmetries and the errors in the retained relations. We first record
these two bounds, then make the finite-order relations exact.

\begin{lemma}\label{lem:coverage}
Let $A=(A_1,\ldots,A_N)$ be a tuple of selfadjoint contractions in a
finite tracial algebra with normalized trace $\tau$. Let
$w_1,\ldots,w_m$ be words in $X_1,\ldots,X_N$ such that every
variable occurs in at least one word, and put
\[
 q=\sum_{\rho=1}^m\tau\bigl(2-w_\rho(A)-w_\rho(A)^*\bigr).
\]
Define $\widehat A_j=\sgn(A_j)$ by functional calculus, where
$\sgn(t)=1$ for $t\ge0$ and $\sgn(t)=-1$ for $t<0$.
Then $q\ge0$, each $\widehat A_j$ is a symmetry, and
\begin{equation}\label{eq:coverbounds}
 \|A_j-\widehat A_j\|_2\le\sqrt{q/2},\qquad
 \|1-w_\rho(A)\|_2\le\sqrt q.
\end{equation}
If every word has length at most $L$, then, with
$\widehat A=(\widehat A_1,\ldots,\widehat A_N)$,
\begin{equation}\label{eq:nativeedit}
 \|1-w_\rho(\widehat A)\|_2
 \le(1+L/\sqrt2)\sqrt q.
\end{equation}
\end{lemma}
\begin{proof}
For any contraction $w$,
\begin{equation}\label{eq:defectidentity}
 2-w-w^*=(1-w)^*(1-w)+(1-w^*w)\ge(1-w)^*(1-w).
\end{equation}
Thus every summand defining $q$ is nonnegative, and the second
bound in \eqref{eq:coverbounds} follows.
Choose a word $w=w_\rho(A)$ containing $A_j$. Cyclically rotate
its trace to $\tau(A_jB)$, where $\|B\|\le1$. Tracial H\"older gives
\[
 \operatorname{Re}\tau(w)\le|\tau(A_jB)|\le\tau(|A_j|).
\]
Consequently,
\[
 \|A_j-\widehat A_j\|_2^2
 =\tau\bigl((1-|A_j|)^2\bigr)
 \le\tau(1-|A_j|)
 \le1-\operatorname{Re}\tau(w)\le q/2.
\]
Telescoping each word under the substitutions
$A_j\mapsto\widehat A_j$ proves \eqref{eq:nativeedit}.
\end{proof}

\begin{lemma}\label{csimp:r3:roundorder}
Let $U\in\mathcal U(d)$ and let $Q\in M_d$ be a symmetry with $QUQ=U^*$.
For $n=3$ or $n=6$, there is a unitary $\widehat U$ such that
\[
 \widehat U^n=1,\qquad Q\widehat UQ=\widehat U^*,\qquad
 \|U-\widehat U\|_2\le\|U^n-1\|_2.
\]
\end{lemma}
\begin{proof}
For $n=6$, choose a nearest sixth root on the unit circle
equivariantly under complex conjugation; the real points $1,-1$
are already sixth roots. If the argument error is
$\theta\in[-\pi/6,\pi/6]$, then
$2\sin(|\theta|/2)\le2\sin(6|\theta|/2)$.
For $n=3$, choose nearest cubic roots equivariantly, except that
$-1$ is sent to $1$. Away from $-1$, the argument error lies in
$[-\pi/3,\pi/3]$ and satisfies
$2\sin(|\theta|/2)\le2\sin(3|\theta|/2)$.
At $-1$, the distance to the chosen root and $|(-1)^3-1|$ both
equal $2$.

In either case this defines a Borel function $\varphi_n$ satisfying
\[
 \varphi_n(z)^n=1,\qquad
 \varphi_n(\bar z)=\overline{\varphi_n(z)},\qquad
 |z-\varphi_n(z)|\le|z^n-1|\quad(|z|=1).
\]
Functional calculus with $\widehat U=\varphi_n(U)$ proves the
claims. The special choice at $-1$ preserves conjugation symmetry
for odd order.
\end{proof}

\begin{theorem}[Matrix positivity]\label{thm:matrixpositivity}
For every integer $d\ge1$ and every tuple $A=(A_1,\ldots,A_6)$
of selfadjoint contractions in $M_d$, the polynomial $f$ in
\eqref{eq:nativefinal} satisfies
\[
 \tr_d f(A)\ge\frac34.
\]
\end{theorem}
\begin{proof}
Let
\[
 \rho=\sum_{j=1}^{15}\tr_d\bigl(2-u_j(A)-u_j(A)^*\bigr),
 \qquad
 \rho_*:=\frac4M=\frac{\delta^2}{2^{52}},
\]
where $\delta=2^{-J}$ and $J,M$ are fixed in
\eqref{eq:globalconstants}. Each $u_j(A)$ is a contraction, so
$\rho\ge0$. If $\rho\ge\rho_*$, the two signal words also
evaluate to contractions, and hence
\[
 \tr_d f(A)\ge-3+M\rho\ge1.
\]

Suppose now that $0\le\rho<\rho_*$. All six variables occur in
the penalty words. Lemma~\ref{lem:coverage} therefore gives
$\widehat A_j=\sgn(A_j)$ with
\[
 \|A_j-\widehat A_j\|_2\le\sqrt{\rho/2},\qquad
 \|1-u_i(A)\|_2\le\sqrt\rho.
\]
On this tuple put $q=\widehat A_6$, $s=\widehat A_1$, and
$t=\widehat A_4q$. Since $u_1=t^6$ and $u_2=(qs)^3$ have
lengths $12$ and $6$, respectively, \eqref{eq:nativeedit} gives
\[
 \|t^6-1\|_2\le10\sqrt\rho,\qquad
 \|(qs)^3-1\|_2\le6\sqrt\rho.
\]
Apply Lemma~\ref{csimp:r3:roundorder} to $t$ and $qs$, keeping
$q$ fixed. Write their corrected values as $\widetilde t$ and
$q\widetilde s$. Then
$\widetilde t^6=(q\widetilde s)^3=1$, and both
$\widetilde t q$ and $\widetilde s$ are symmetries.
Define the final tuple $\widetilde A$ by replacing
$\widehat A_4$ with $\widetilde t q$ and $\widehat A_1$ with
$\widetilde s$, leaving the other entries unchanged. Thus
\[
 \|A_j-\widetilde A_j\|_2\le11\sqrt\rho
 \qquad(1\le j\le6).
\]

Every penalty word has length at most $72$. Telescoping gives
\[
 \|1-u_i(\widetilde A)\|_2
 \le(1+72\cdot11)\sqrt\rho
 =793\sqrt\rho\le2^{10}\sqrt\rho.
\]
On a tuple of symmetries, the cancellations defining $u_i$
preserve the corresponding relation defect: deleting adjacent equal
letters preserves the evaluated word, and removing matching
endpoints conjugates it by a unitary. The retained relations
therefore have these same bounds at $\widetilde A$.
Lemmas~\ref{csimp:r3:recoverthree} and~\ref{csimp:r17:lem:core}
bound all 97 full-system defects by
\[
 3\cdot2^{13}\cdot793\sqrt\rho
 \le2^{25}\sqrt\rho<\delta/2.
\]
Choose $\eps$ strictly between the largest full defect and $\delta$.
Theorem~\ref{thm:gate} applies, including when $\rho=0$.

At $\widetilde A$, the word $z_2$ is unitary and $W$ is a
symmetry. The reflection estimate and the signal trace identity
give
\[
 \tr_d\bigl(-1+\omega(\widetilde A)
                   +\omega(\widetilde A)^*\bigr)
 \ge1-g_0^2,\qquad g_0=5\cdot2^{-20}.
\]
Each signal word has length $36$, so their combined trace change
between $A$ and $\widetilde A$ is at most $792\sqrt\rho$.
Consequently,
\[
 \tr_d f(A)\ge1-g_0^2-792\sqrt\rho+M\rho\ge\frac34.
\]
Indeed, $\delta\le1$ and $\rho<\rho_*$ imply
\[
 g_0^2+792\sqrt\rho
 \le25\cdot2^{-40}+792\cdot2^{-26}<\frac14.\qedhere
\]
\end{proof}

The explicit words in \eqref{eq:signalword}--\eqref{eq:penaltywords}
and their adjoints give 32 distinct nonconstant monomials.
The two signal words have length $36$, the penalty words have
maximum length $72$, and the total expanded letter count is $716$.
Since the constant coefficient is $30M-1\ne0$, the Hermitian
polynomial $f$ has integer coefficients, degree $72$, and exactly
33 monomials. These facts, Theorem~\ref{thm:matrixpositivity},
and the negative evaluation in
Proposition~\ref{prop:s2:negative-evaluation} prove
Theorem~\ref{thm:main}.

\begin{remark}
For a word $w=X_{i_1}\cdots X_{i_\ell}$, put
$t_j=X_{i_{j+1}}\cdots X_{i_\ell}$, with $t_\ell=1$.
The telescoping identity
\[
 1-w^*w=\sum_{j=1}^\ell t_j^*(1-X_{i_j}^2)t_j
\]
together with \eqref{eq:defectidentity} shows that every penalty
$2-u_j-u_j^*$ belongs to the contraction quadratic module
$\mathcal M_6$ defined in \eqref{eq:cubemoduleintro}.
\end{remark}

\begin{corollary}\label{cor:module}
For $0\le\eps<1$,
\[
 f+\eps\notin\mathcal M_6+
 [\mathcal A_6,\mathcal A_6]_{\rm lin},
\]
where $\mathcal A_6=\C\langle X_1,\ldots,X_6\rangle$ and
$[\mathcal A_6,\mathcal A_6]_{\rm lin}$ denotes the linear span
of additive commutators.
\end{corollary}
\begin{proof}
At the contraction tuple of
Proposition~\ref{prop:s2:negative-evaluation}, every summand of
$\mathcal M_6$ has nonnegative trace and every additive commutator
has trace zero, whereas $\tau_\Lambda(f+\eps)=-1+\eps<0$.
\end{proof}

\subsection*{Acknowledgements}
We thank Andreas Thom for suggesting how to obtain an explicit polynomial
from the group double in his work with G\'abor Kun, and for his advice on
the two-variable encoding, proof simplification and attribution.
We thank Florin R\u{a}dulescu for his detailed note on a quartic reduction
of our example and for drawing our attention to his earlier fourth-moment
and analytic Hilbert-seventeenth formulations. We also thank Sizhuo Yan,
Jianting Yang, Tianshi Yu and Hao Liang for valuable discussions and
encouragement, and the MechMath Agent Team~\cite{cao2026mechmathagentteamllm}
for their support.
The authors are supported by the National Key R\&D Program of China
2023YFA1009401 and the Strategic Priority Research Program of the
Chinese Academy of Sciences under Grant XDA0480501.

\subsection*{Declaration of Generative AI}
Generative AI tools assisted in exploring proof strategies, checking
calculations,  revising the exposition, and drafting the paper. The authors take responsibility
for the mathematical content, references, and final manuscript.

\appendix

\section{A quartic reduction following R\u{a}dulescu}
\label{subsec:quartic}\label{sec:quartic}

We construct a Hermitian quartic in 37 selfadjoint variables whose
normalized trace is at least $3/4$ on all selfadjoint matrix tuples, without norm
restrictions, and equals $-1$ at a specified contraction tuple.
Auxiliary product coordinates make the required relations quadratic,
and a Hermitian quadratic selector encodes the fifth power relation.

\begin{remark}\label{rem:fourthmoments}
The reduction to moments of orders three and four goes back to
R\u{a}dulescu~\cite{Rad99}; see also~\cite[Section~2]{CapRad}.
His analytic noncommutative Hilbert-seventeenth formulation is a
separate result~\cite{Rad04}.
In a private note of 2 October 2026, Florin R\u{a}dulescu applied a
prefix product reduction to our earlier example, obtaining 1,577
selfadjoint variables and a matrix trace lower bound of $1/4$.
The dictionary with 37 variables, the symmetries of the basic elements, and the sparse
weights below are subsequent refinements.
For operators $U,V,Z$ in a finite tracial algebra with trace $\tau$,
the basic identity is
\[
 \tau((Z-UV)^*(Z-UV))
 =\tau(Z^*Z)-2\operatorname{Re}\tau(Z^*UV)+\tau(V^*U^*UV).
\]
It records a product through moments of order at most four.
Our construction shares products using the group relations and applies
the reflection estimate in Theorem~\ref{thm:gate} directly.
\end{remark}

The following selector replaces the fifth power relation by a
quadratic expression and a symmetry coordinate.

\begin{lemma}\label{cs:Hbound}
In a finite tracial algebra with normalized trace $\tau$, let $g$ be
unitary and $s_g$ a symmetry, and use
$\|x\|_2=\tau(x^*x)^{1/2}$.
Define
\[
 H_g=2g^2+2(g^*)^2+2g+2g^*-3-5s_g.
\]
Then, without any commutation assumption on $g$ and $s_g$,
\[
 \|g^5-1\|_2^2\le\|H_g\|_2^2+\tfrac52\|H_g\|_2.
\]
If $g^5=1$, the choice
$s_g=\frac25\sum_{j=0}^4g^j-1$ is a symmetry and gives $H_g=0$.
\end{lemma}
\begin{proof}
Put $P=(1+s_g)/2$ and $A=1+g+g^*+g^2+(g^*)^2$.
Since $P$ is a projection and $A$ is selfadjoint,
\[
 H_g=2(A-5P),\qquad
 \|H_g\|_2^2/4
 =\tau((1-P)A^2)+\tau(P(A-5)^2).
\]
For $|z|=1$, put $t=|z-1|^2$ and $a(t)=t^2-5t+5$.
Then $1+z+\bar z+z^2+\bar z^2=a(t)$ and
$|z^5-1|^2=t\,a(t)^2$.
For $0\le t\le4$,
\[
 t\,a(t)^2\le4a(t)^2,\qquad
 t\,a(t)^2\le5|a(t)-5|.
\]
For the second inequality, use $|a(t)-5|=t(5-t)$ and
\[
 5(5-t)-a(t)^2=t(-t^3+10t^2-35t+45).
\]
Writing $v=4-t$, the last factor is
$v(v-1)^2+2v+1\ge0$.
Functional calculus and Cauchy--Schwarz now give
\[
 \|g^5-1\|_2^2
 \le4\tau((1-P)A^2)+5\tau(P|A-5|)
 \le\|H_g\|_2^2+\tfrac52\|H_g\|_2.
\]
When $g^5=1$, the average $\frac15\sum_{j=0}^4g^j$ is the
spectral projection for the eigenvalue one, proving the last assertion.
\end{proof}

\subsection{Quadratic relations and tracial majorization}
\label{cs:dictionary}\label{cs:quartic}

Let $T_1,\ldots,T_{37}$ be free selfadjoint variables.
Introduce the fifteen complex coordinate symbols, in this order,
\begin{equation}\label{csimp:r3:complexcoords}
 x,z,z_2,z_3,a,b,a_+,d,e_2,p,C_d,C_p,t,r,C_a,
\end{equation}
with the $j$th symbol abbreviating $T_{2j-1}+iT_{2j}$.
The seven remaining coordinate symbols are selfadjoint:
\begin{equation}\label{csimp:r3:realcoords}
 (q,v,\ell,s,s_z,W,j)=(T_{31},\ldots,T_{37}).
\end{equation}
Thus these 22 symbols represent 37 primitive selfadjoint variables.
On matrix inputs, the first fifteen symbols are arbitrary matrices
and the last seven are arbitrary selfadjoint matrices.
No norm bounds or unitary conditions are assumed.

The independent coordinates $C_d,C_p,C_a$ will record the products
$xz$, $d^*e_2$, and $ba_+$, respectively.
Retaining $C_p$ and $C_a$ keeps the equations $C_p d=pe_2$ and
$aa_+=C_a a$ quadratic. The thirty-one binomial equations are
\begin{align*}
&r=t^2,\quad v=tr,\quad qt=t^*q,\quad qz=z^*q,\quad bq=qa,\\
&a_+r=ra,\quad z_2r=rz,\quad rz_3=zr,\quad\ell=rq,\\
&C_d=xz,\quad C_dx^*=dz,\quad e_2\ell=\ell d^*,\\
&C_p=d^*e_2,\quad C_p d=pe_2,\quad az_2=pa,\quad vz=zv,\\
&tx=xt,\quad ax=xa,\quad zz_2=z_2z,\quad dx=xd,\quad dz_2=z_2d,\\
&ab=ba,\quad ba_+=a_+b,\quad C_a=ba_+,\quad aa_+=C_a a,\\
&az=za,\quad az_3=z_3a,\quad Wa=av,\quad xs=sz,\\
&jq=qs,\qquad js=sq.
\end{align*}
The last two equations recover the braid $qsq=sqs$ after rounding,
without a separate complex coordinate for $qs$.
Since $z=T_3+iT_4$ and $s_z=T_{35}$, define the Hermitian selector
\begin{equation}\label{eq:refinedselector}
 \widetilde H_z=T_3^2-7T_4^2+4T_3-5T_{35}.
\end{equation}
The expression $H_z$ from Lemma~\ref{cs:Hbound}, evaluated at these
coordinates, satisfies the polynomial identity
\begin{equation}\label{eq:selectorcomparison}
 \widetilde H_z
 =H_z-3\left(\frac{z^*z+zz^*}{2}-1\right).
\end{equation}
The two selectors therefore agree whenever $z$ is unitary, without
requiring it to commute with $s_z$.

For single letters $A,B$ from this dictionary or their adjoints, put
\[
 \mathcal N(A)=A^*A,\qquad
 \mathcal N(AB)=\tfrac12\bigl((A^*A)^2+(B^*B)^2\bigr).
\]
For each displayed equation $L=R$, define
\[
 \mathcal P(L,R)
 =\mathcal N(L)+\mathcal N(R)-L^*R-R^*L.
\]
These expressions majorize the residual squares under the trace,
as proved below. With $J$ fixed in \eqref{eq:globalconstants}, set
\begin{equation}\label{csimp:r3:quarticdef}
\begin{split}
 Q_\flat={}&\sum_{L=R}\mathcal P(L,R)
       +\frac{13}{7}\widetilde H_z^2
       +\frac12\sum_{u\in\{C_d,C_p,r,C_a,v,\ell\}}(1-u^*u)^2
       +\frac{325}{14}(1-s_z^2)^2,\\
 F_\flat={}&-1+z_2^*Wz_2W+Wz_2^*Wz_2+\Omega_\flat Q_\flat,
       \qquad \Omega_\flat=7\cdot2^{4J+126}.
\end{split}
\end{equation}
The first sum runs over all thirty-one equations.
The last seven squares in $Q_\flat$ penalize failures of unitarity;
for the selfadjoint symbols, unitarity means being a symmetry.
The following identity provides the additional control needed for
the other coordinates.

\begin{lemma}\label{csimp:r1:lem:major}\label{csimp:r1:majorsection}
For arbitrary matrix values $A,B\in M_d$,
\begin{equation}\label{csimp:r1:eq:major}
 \tr_d\bigl(\mathcal N(AB)-(AB)^*(AB)\bigr)
 =\tfrac12\|A^*A-BB^*\|_2^2.
\end{equation}
Consequently, $Q_\flat$ is cyclically equivalent to the sum of its
seven weighted unitarity squares, $(13/7)\widetilde H_z^2$, the 31
residual squares $(L-R)^*(L-R)$, and
$\tfrac12(A^*A-BB^*)^2$ for every occurrence of a product $AB$ on
either side of the equations, counting repetitions.
\end{lemma}
\begin{proof}
Cyclicity gives
$\tr_d((AB)^*(AB))=\tr_d(A^*ABB^*)$ and
$\tr_d((B^*B)^2)=\tr_d((BB^*)^2)$.
Expanding the square proves \eqref{csimp:r1:eq:major}.
The same expansion is an identity modulo commutators; summing it
proves the assertion about $Q_\flat$.
The additional squares vanish when the factors are unitary.
\end{proof}

In particular, $\tr_dQ_\flat\ge0$ on every selfadjoint matrix tuple.
We call the additional terms balance squares.
For square matrix values $A,B$, the identity
$\|1-A^*A\|_2=\|1-AA^*\|_2$ gives
\[
 \|1-A^*A\|_2
 \le\|1-B^*B\|_2+\|A^*A-BB^*\|_2.
\]
Thus these squares transfer unitarity control between factors.
If $\bar A$ is a unitary completion of the polar factor of $A$, then
\[
 \|A-\bar A\|_2\le\|1-A^*A\|_2,\qquad
 \|A-\bar A\|_4^2\le\|1-A^*A\|_2.
\]
These inequalities follow on singular values from
$|s-1|\le|s^2-1|$ and $|s-1|^4\le(s^2-1)^2$ for $s\ge0$.
For selfadjoint $A$, use sign rounding, sending a zero eigenvalue
to one. See~\cite[Chapters~III--IV]{Bhatia} and
\cite[Theorem~2.2]{Higham}.

The weights also reduce the number of monomials.
For $u\in\{C_d,C_p,r,C_a,v,\ell\}$, the quadratic diagonal term
cancels because
\[
 u^*u+\tfrac12(1-u^*u)^2
 =\tfrac12+\tfrac12(u^*u)^2.
\]
The weight $325/14$ cancels the $s_z^2$ contribution from
$(13/7)\widetilde H_z^2$.
The sum of the coefficients in the cyclic rotation class of
$T_3^2T_4^2$ is $26-14(13/7)=0$.
After cyclic collection, the coefficients of $T_3$, $T_4^2$, and
$T_{35}$ also vanish. No separate unitarity square for $z$ is needed,
as the next lemma shows.

\subsection{Rounding and recovery on unrestricted tuples}

We now evaluate the coordinates at arbitrary selfadjoint matrices
$T_1,\ldots,T_{37}$ of a common size.

Round each of the fifteen complex dictionary symbols to a unitary,
and each of the seven selfadjoint dictionary symbols to a symmetry.
Bars denote these independent roundings of the 22 symbols.
In particular, $e_2$ remains an independent coordinate; its relation
to the word $\ell d^*\ell$ is recovered later.

\begin{lemma}\label{csimp:r3:polar}
Put $\rho=\tr_dQ_\flat\ge0$.
Every raw binomial residual has $2$-norm at most $\sqrt\rho$, and
\[
 \|\widetilde H_z\|_2\le\sqrt{7/13}\sqrt\rho.
\]
Each dictionary symbol $u$ in
\eqref{csimp:r3:complexcoords}--\eqref{csimp:r3:realcoords} satisfies
\[
 \|1-u^*u\|_2\le5\sqrt\rho.
\]
Its independent polar or sign rounding $\bar u$ satisfies
\[
 \|u-\bar u\|_2\le5\sqrt\rho,\qquad
 \|u-\bar u\|_4\le\sqrt5\,\rho^{1/4}.
\]
After rounding, all thirty-one binomial residuals and the original
selector $H_{\bar z}$, using $\bar s_z$ as its symmetry coordinate,
have $2$-norm at most $128\sqrt\rho$.
\end{lemma}
\begin{proof}
Lemma~\ref{csimp:r1:lem:major} expresses $Q_\flat$, modulo
commutators, as a sum of weighted squares with nonnegative traces.
This proves the raw residual and selector bounds.
The seven coordinates with explicit unitarity penalties have
defects at most $\sqrt{2\rho}$.
Each balance square changes this bound by at most
$\sqrt{2\rho}$ between its two factors.
The displayed equations supply the following links:
\[
\begin{array}{c|l}
\text{coordinate with a unitarity penalty}&\text{neighbors used}\\\hline
C_d&x\\
C_p&d\\
r&a,a_+,q,t,z,z_2,z_3\\
\ell&e_2
\end{array}
\]
The second links are $x$ to $s$, $a$ to $W,b,p$, and $q$ to $j$.
Thus every symbol is at most two links from a penalized coordinate,
giving a defect at most $3\sqrt{2\rho}\le5\sqrt\rho$.
The singular value inequalities above give both rounding bounds.

For a product $AB$, expand
\[
 AB-\bar A\bar B
 =(A-\bar A)\bar B+\bar A(B-\bar B)
       +(A-\bar A)(B-\bar B).
\]
The first two terms are bounded in $2$-norm by $5\sqrt\rho$ each;
Schatten H\"older and the $4$-norm bounds give the same bound for
the third. Hence every quadratic monomial changes by at most
$15\sqrt\rho$, and each rounded binomial residual is at most
$31\sqrt\rho$.

By \eqref{eq:selectorcomparison}, the original selector on the raw
tuple satisfies
\[
 \|H_z\|_2
 \le\bigl(\sqrt{7/13}+15\bigr)\sqrt\rho\le16\sqrt\rho.
\]
Its quadratic, linear, and symmetry terms change under rounding
by at most $(60+20+25)\sqrt\rho=105\sqrt\rho$.
Therefore
$\|H_{\bar z}\|_2\le121\sqrt\rho\le128\sqrt\rho$.
These nonstrict bounds also cover $\rho=0$.
\end{proof}

\begin{lemma}\label{cs:fourfactor}
Let $B_i,\widehat B_i,E_i\in M_d$ satisfy
$B_i=\widehat B_i+E_i$ for $1\le i\le4$.
If $\alpha,\beta\ge0$ and
$\|\widehat B_i\|\le1$, $\|E_i\|_2\le\alpha$,
$\|E_i\|_4\le\beta$, then
\[
 \bigl|\tr_d(B_1B_2B_3B_4
          -\widehat B_1\widehat B_2\widehat B_3\widehat B_4)\bigr|
 \le4\alpha+6\alpha^2+4\alpha\beta^2+\beta^4.
\]
\end{lemma}
\begin{proof}
Expand by the number of error factors.
The four terms with one error cost at most $\alpha$ each.
Schatten H\"older with exponents $(2,2)$, $(2,4,4)$, and
$(4,4,4,4)$ bounds each term with two, three, and four errors by
$\alpha^2$, $\alpha\beta^2$, and $\beta^4$, respectively.
There are six, four, and one such terms.
\end{proof}

\subsection{The Hermitian cyclic representative and trace separation}
Write $p\cyc q$ when $p-q$ is a sum of additive commutators.
Equivalently, $p$ and $q$ have the same total coefficient in each
class of words under cyclic rotation
\cite[Definition~1.2 and Remark~1.3]{KS}.

Order the selfadjoint letters as $T_1<\cdots<T_{37}$.
For a Hermitian polynomial $p$, define $\HC(p)$ by collecting these
class coefficients as follows. For two distinct classes paired by
reversal, take the lexicographically least word $w$ in their union
and assign the two class coefficients to $w,w^*$.
For a class fixed by reversal, assign its coefficient to its least
palindromic rotation, if one exists; otherwise assign half its
coefficient to its least rotation $w$ and half to $w^*$.
Keep the constant coefficient unchanged.\label{cs:cyclic}

This rule is linear over $\R$, preserves Hermitian symmetry and class
coefficients, does not increase degree, and satisfies $\HC(p)\cyc p$.
It also preserves coefficients in $\Z[i]$: in a class fixed by reversal
with no palindrome, the words form distinct adjoint pairs, whose
coefficients have even integer sum. Thus the division by two remains
in $\Z[i]$. Define
\begin{equation}\label{cs:cyclicquartic}
 F_4=\HC(F_\flat)\cyc F_\flat.
\end{equation}
The two polynomials therefore have identical traces under every
substitution. The presentation of $Q_\flat$ as residual, unitarity, selector,
and balance squares is modulo commutators, as in
Lemma~\ref{csimp:r1:lem:major}; the individual majorants need not be
Hermitian squares.

For the negative evaluation, use the group readouts of
Section~\ref{subsec:coordinates} and set
\[
 e_2=\ell d^*\ell,\qquad C_d=xz,\qquad C_p=d^*e_2,\qquad
 C_a=ba_+,\qquad j=qsq=sqs,\qquad
 s_z=\frac25\sum_{k=0}^4z^k-1.
\]
Let $\mathbf a_4\in L(\Lambda)^{37}$ be the corresponding tuple in
\eqref{csimp:r3:complexcoords}--\eqref{csimp:r3:realcoords}:
each complex group readout $U$ supplies
$\operatorname{Re}U=(U+U^*)/2$ and
$\operatorname{Im}U=(U-U^*)/(2i)$, followed by the seven symmetry
coordinates. All entries of $\mathbf a_4$ are selfadjoint contractions.
Let $F_2$ denote the free group on two generators.

\begin{theorem}\label{thm:quartic}\label{csimp:r3:quarticthm}
The Hermitian polynomials $F_\flat,F_4\in
\Z[i]\langle T_1,\ldots,T_{37}\rangle$ have degree four and
914 and 387 distinct monomials, respectively.
For every $d\ge1$ and every tuple $A=(A_1,\ldots,A_{37})$ of
selfadjoint matrices in $M_d$, without norm restrictions,
\[
 \tr_d F_4(A)=\tr_d F_\flat(A)\ge\frac34.
\]
At the specified contraction tuple $\mathbf a_4$,
\[
 \tau_\Lambda F_4(\mathbf a_4)
 =\tau_\Lambda F_\flat(\mathbf a_4)=-1.
\]
The same negative evaluation holds in the separable $\mathrm{II}_1$ factor
$L(\Lambda*F_2)$. Before expanding the complex dictionary symbols into
the $T_j$, the sum of the absolute values of the coefficients of
$F_\flat$ is at most $3+(4930/7)\Omega_\flat$.
\end{theorem}
\begin{proof}
Fix $A$ as in the statement, and evaluate all dictionary symbols at
$A$. Put $\rho=\tr_d Q_\flat(A)\ge0$ and $e=128\sqrt\rho$.
Lemma~\ref{csimp:r3:polar} rounds the 22 dictionary symbols to
unitaries or symmetries $\bar u$, with all binomial residuals and
$H_{\bar z}$ of norm at most $e$.

First make the reflection and finite order identities exact.
For a unitary $U$,
\[
 \|U-\operatorname{sgn}(\operatorname{Re}U)\|_2
 \le\|U-U^*\|_2,
\]
where the sign at zero is chosen to be $1$.
Apply this to $\bar t\bar q$ and $\bar z\bar q$, then multiply on
the right by $\bar q$. The resulting unitaries $t',\widehat z$
are exactly reversed by $\bar q$ and differ from $\bar t,\bar z$
by at most $e$.
The equations for $r,v$ give $\|\bar t^3-\bar v\|_2\le2e$,
so $\|(t')^6-1\|_2\le10e$.
Lemma~\ref{csimp:r3:roundorder} therefore gives
\[
 \widehat t^6=1,\qquad
 \bar q\widehat t\bar q=\widehat t^*,\qquad
 \|\widehat t-\bar t\|_2\le11e.
\]
The two equations for $j$ give
$\|\bar q\bar s\bar q-\bar s\bar q\bar s\|_2\le2e$, hence
$\|(\bar q\bar s)^3-1\|_2\le2e$.
Applying the same lemma with order three produces a symmetry
$\widehat s$ such that
$(\bar q\widehat s)^3=1$ and
$\|\widehat s-\bar s\|_2\le2e$.

Set $\widehat q=\bar q$, $\widehat a=\bar a$, and
$\widehat x=\widehat s\widehat z\widehat s$, and form the remaining
core readouts using \eqref{csimp:r17:eq:dict}.
For comparison with the independent dictionary symbol $e_2$, set
$\widehat e_2=\widehat\ell\widehat d^*\widehat\ell$; this is the
auxiliary short word in the compression target; the full system entry
is recovered separately. Successive use of the quadratic equations gives the following
bounds for $\|\bar u-\widehat u\|_2$, in multiples of $e$:
\[
\begin{array}{c|rrrrrrrrrrrr}
u&b&r&v&z_2&z_3&a_+&d&\ell&e_2&p&W&x\\\hline
\text{bound}&1&23&35&48&48&47&16&24&65&164&36&6
\end{array}
\]
For example, $xs=sz$ gives the bound $6e$ for $x$.
The two equations defining $d$ then cost
$(2+2\cdot6+2\cdot1)e=16e$; the bounds for $e_2$ and $p$ are
$(1+2\cdot24+16)e=65e$ and
$(2+2\cdot16+2\cdot65)e=164e$.
All retained tests other than the power relations have defects at
most $213e$, the largest being the compression bound
$(48+1+164)e$.

For $0\le\rho\le1$, Lemma~\ref{cs:Hbound} gives
\[
 \|\bar z^5-1\|_2^2
 \le e^2+\tfrac52e
 \le16704\sqrt\rho\le2^{16}\sqrt\rho.
\]
The edit to $\widehat z$ therefore gives
$\|\widehat z^5-1\|_2\le(256+640)\rho^{1/4}
\le2^{10}\rho^{1/4}$.
Thus all fifteen retained defects are at most
$2^{15}\rho^{1/4}$, and Lemma~\ref{csimp:r3:recoverthree}
bounds the 97 full defects by $2^{30}\rho^{1/4}$.
Recall $\delta=2^{-J}$. If
\[
 \rho<\rho_*:=\frac4{\Omega_\flat}
 =\frac{\delta^4}{7\cdot2^{124}},
\]
these full defects are less than
$\delta/(2\,7^{1/4})<\delta/2$.
Choose $\eps>0$ strictly between their maximum and $\delta$.
Theorem~\ref{thm:gate} applies, also when $\rho=0$, and the core
signal has trace at least $1-g_0^2$, where $g_0=5\cdot2^{-20}$.

Lemma~\ref{cs:fourfactor}, with
$\alpha=5\sqrt\rho$ and $\beta^2=5\sqrt\rho$, bounds the total
change of the two signal words under polar rounding by
\[
 40\sqrt\rho+550\rho\le64\sqrt\rho+1024\rho.
\]
The subsequent core edits cost at most
$2(2\cdot48+2\cdot36)e=336e$ by the table. Consequently,
\[
 \tr_d F_\flat(A)
 \ge1-g_0^2-2^{16}\sqrt\rho-1024\rho+\Omega_\flat\rho
 \ge\frac34\qquad(\rho<\rho_*).
\]
The last inequality uses $\rho_*<2^{-124}$ and
$g_0^2+2^{-46}+2^{-114}<1/4$.
For $\rho\ge\rho_*$, the rounded signal has trace at least
$-3$, so
\[
 \tr_d F_\flat(A)
 \ge-3+(\Omega_\flat-1024)\rho-64\sqrt\rho
 \ge-3+\tfrac{15}{16}\Omega_\flat\rho\ge\frac34.
\]
Here $64/\sqrt\rho\le32\sqrt{\Omega_\flat}$, and
$1024+32\sqrt{\Omega_\flat}\le\Omega_\flat/16$.

At $\mathbf a_4$, every binomial residual, unitarity defect, selector,
and balance square vanishes, so the signal gives trace $-1$.
The inclusion into $L(\Lambda*F_2)$ preserves the canonical trace.
Free product normal form gives infinite conjugacy classes for every
nonidentity element of $\Lambda*F_2$; as this group is countable,
its von Neumann algebra is a separable $\mathrm{II}_1$ factor.

In the formal alphabet of the complex symbols, their adjoints, and
the selfadjoint symbols, the coefficient sums contributed to
$Q_\flat$ are at most $124$ from the majorants, $12$ from the six
terms of weight $1/2$, $650/7$ from the $s_z$ unitarity penalty, and
$(13/7)16^2=3328/7$ from the selector. Their sum is $4930/7$;
the selector bound uses
$\|\widetilde H_z\|_{\mathrm{coeff},1}\le16$ from
\eqref{eq:selectorcomparison}, where this norm is the sum of absolute
coefficients. Expanding in the primitive letters gives 914 distinct
monomials in $F_\flat$ and 387 in $F_4$, both of degree four.
All coefficients of $14Q_\flat$ lie in $\Z[i]$.
Since $\Omega_\flat$ is divisible by $14$ and $\HC$ preserves
$\Z[i]$, both polynomials have the asserted coefficients.
\end{proof}

\begin{corollary}\label{cor:quarticmodule}
For $0\le\eps<1$, both $F_4+\eps$ and $F_\flat+\eps$ lie outside
$\mathcal M_{37}+[\mathcal A_{37},\mathcal A_{37}]_{\rm lin}$,
where $\mathcal A_{37}$ is the free $*$-algebra and $\mathcal M_{37}$
is its contraction quadratic module from \eqref{eq:cubemoduleintro}.
\end{corollary}
\begin{proof}
Evaluation at $\mathbf a_4$ has trace $-1+\eps<0$, whereas every
element of the displayed set has nonnegative trace there.
\end{proof}

\begin{remark}\label{cor:analyticquartic}
The same evaluation separates $F_4+\eps$, for $0\le\eps<1$, from
weak limits of sums of Hermitian squares modulo the weakly closed
space of cyclic differences in the analytic free $*$-algebra of
\cite{Rad04}. Indeed, its defining seminorms are
$\|p\|_R=\sum_w|c_w|R^{|w|}$ for $p=\sum_wc_ww$, where $|w|$
is the length of $w$, and
\[
 |\tau_\Lambda(p(\mathbf a_4))|\le\|p\|_{R=1}.
\]
This continuous tracial functional is nonnegative on squares and
their weak closure, vanishes on the closed cyclic subspace, and
takes value $-1+\eps$ on $F_4+\eps$.
Thus the example refutes the matrix positivity implication in
\cite[Corollary~1.2]{Rad04}, not the theorem there assuming
positivity in all $\mathrm{II}_1$ algebras.
\end{remark}

\begin{proposition}\label{cs:degreeoptimal}
A Hermitian polynomial of degree at most three with nonnegative trace
on all unrestricted selfadjoint matrix tuples is cyclically equivalent
to a sum of squares of real affine linear polynomials.
Thus degree four is minimal for trace separation without norm restrictions.
\end{proposition}
\begin{proof}
Write $p=p_3+p_2+p_1+p_0$, where $p_j$ is homogeneous of degree $j$.
For every selfadjoint matrix tuple $A$, nonnegativity of
$\tr_d p(tA)$ for all $t\in\R$ forces $\tr_d p_3(A)=0$.
The trace-zero criterion \cite[Theorem~2.1]{KS} gives $p_3\cyc0$.
Modulo commutators, $X_iX_j$ and $X_jX_i$ agree; Hermitian symmetry
therefore identifies the remainder with the symmetric lift of a real
commutative quadratic polynomial. Scalar evaluation makes this
quadratic nonnegative everywhere, so it is a sum of squares of real
affine linear forms. Lifting those forms proves the claim.
\end{proof}
The scaling argument uses unrestricted tuples and gives no degree
optimality assertion on a contraction domain.

\section{A counterexample in two selfadjoint variables}\label{app:fewvariables}

We encode the unrestricted quartic from Appendix~\ref{sec:quartic}
in two selfadjoint variables. Polynomial readouts recover its input
symbols in one corner of a matrix algebra with twelve blocks.
A further penalty controls the error when this corner is only
approximate. The resulting polynomial has integer coefficients and
degree at most $48$, and separates matrix contractions from
contractions in a finite tracial algebra.

\subsection{The quartic input and polynomial readouts}

Use $F_\flat$ from Theorem~\ref{thm:quartic}, before cyclic
collection. Its fifteen complex symbols, in order, are
\[
 x,z,z_2,z_3,a,b,a_+,d,e_2,p,C_d,C_p,t,r,C_a,
\]
and its seven selfadjoint symbols are
\[
 q,v,\ell,s,s_z,W,j.
\]
Each complex symbol represents two selfadjoint coordinates.
Thus the quartic theorem applies to arbitrary complex matrix values
of the first fifteen symbols and arbitrary selfadjoint values of
the last seven, without norm restrictions.
Its normalized trace is at least $3/4$ on these matrix inputs and
equals $-1$ at the specified substitution in
$L(\Lambda*F_2)$, with canonical trace $\tau$, from
Appendix~\ref{sec:quartic}. At that substitution the complex
symbols are unitaries and the selfadjoint symbols are symmetries.

In this formal alphabet, the selector is
\begin{equation}\label{enc:formalselector}
 \widetilde H_z=2z^2+2(z^*)^2+2z+2z^*
              -\frac32(z^*z+zz^*)-5s_z.
\end{equation}
Substituting $z=T_3+iT_4$ and $s_z=T_{35}$ gives exactly
$T_3^2-7T_4^2+4T_3-5T_{35}$.
The formal penalty coefficients have denominators dividing $28$:
the selector square has denominator dividing $4$ and weight $13/7$,
while the remaining weights have denominators dividing $14$.
Since
\[
 \Omega_\flat=7\cdot2^{4J+126},\qquad J=2^{55226},
\]
is divisible by $28$, $F_\flat$ has real integer coefficients
in the formal symbols and their adjoints. Its expansion in the
selfadjoint coordinates $T_i$ instead has coefficients in $\Z[i]$.
Returning to the formal alphabet consisting of the fifteen complex
symbols, their adjoints, and the seven selfadjoint symbols, we have
\begin{equation}\label{enc:coefficient}
 c_F=3+\frac{4930}{7}\Omega_\flat\in\Z_{>0},\qquad
 \|F_\flat\|_{\mathrm{coeff},1}\le c_F.
\end{equation}

For two free selfadjoint variables $X,Y$, define
\begin{equation}\label{enc:spectral}
 e=-\frac{(4X-1)(4X-2)(4X-3)}6,\qquad
 \chi(X)=4X(4X-1)(4X-2)(4X-3).
\end{equation}
When $\chi(X)=0$ at a selfadjoint matrix $X$, its spectrum lies
in $\{0,1/4,1/2,3/4\}$, and $e(X)$ is the projection onto its
zero eigenspace. The readouts below will encode the quartic in
this corner.

For integer exponents $a,b\in\{0,1,2\}$, set
\begin{equation}\label{enc:paths}
 R_{(a)}=2^{16}Y(4X)^a,\qquad
 R_{(a,b)}=\bigl(2^{16}R_{(a)}Y-(1^a+2^a+3^a)\bigr)(4X)^b.
\end{equation}
Use the eight exponent pairs
\[
 (\beta^{(1)},\ldots,\beta^{(8)})
 =((0,0),(0,1),(1,0),(0,2),(1,1),(2,0),(1,2),(2,1)),
\]
and define $C_{ik}=2^{16}R_{\beta^{(i)}}Y
R_{\beta^{(k)}}^*$ for $1\le i,k\le8$.
Assign the fifteen complex symbols, in their displayed order, to
\begin{equation}\label{enc:complexreadouts}
 (C_{12},C_{13},C_{14},C_{15},C_{16},
 C_{23},C_{24},C_{25},C_{26},C_{34},C_{35},C_{36},C_{45},C_{46},C_{56}),
\end{equation}
and assign $(q,v,\ell,s,s_z,W,j)$ to
\begin{equation}\label{enc:realreadouts}
 (C_{11},C_{22},C_{33},C_{44},C_{55},C_{66},C_{17}+C_{71}).
\end{equation}
Write $B_u$ for the readout of a symbol $u$, and put
$B_{u^*}=B_u^*$. Since $C_{ki}=C_{ik}^*$, the seven selfadjoint
symbols have selfadjoint readouts at every selfadjoint pair
$(X,Y)$. The readouts need not be contractions; the unrestricted
positivity of $F_\flat$ is used for precisely this reason.

For $i=1,\ldots,6$, the polynomials $R_{\beta^{(i)}}$ have degrees
$2,3,3,4,4,4$, respectively, and $R_{\beta^{(7)}}$ has degree $5$.
Hence
\begin{equation}\label{enc:degree}
 \deg B_u+\deg e\le12
\end{equation}
for every symbol; for $u=j$, this sum is at most $11$.

\subsection{The encoded polynomial and a specified substitution}

For words in the formal input symbols and their adjoints, set
\[
 \mathcal T(1)=e,\qquad
 \mathcal T(b_1\cdots b_k)=B_{b_1}e\cdots B_{b_k}e
 \quad(k\ge1),
\]
and extend $\mathcal T$ linearly. Define
\begin{equation}\label{enc:formula}
\begin{split}
 G_F&=\tfrac12\bigl(\mathcal T(F_\flat)+\mathcal T(F_\flat)^*\bigr),
 \qquad K_F=2^{440}c_F,\\
 \Psi_F&=81\bigl(1+24G_F+(24K_F\chi(X))^2\bigr).
\end{split}
\end{equation}
The map $\mathcal T$ is a word encoding rule. When $e$ is a
projection, cyclicity of trace identifies each encoded word with
the corresponding product of compressed readouts $eB_ue$.
No multiplicativity of $\mathcal T$ is required on arbitrary inputs.

For the negative pair in $M_{12}(L(\Lambda*F_2))$, index the blocks by a root
$\varnothing$, three vertices $a\in\{1,2,3\}$, and eight leaves
$ (a,b)\in\{1,2,3\}^2\setminus\{(3,3)\}$.
Join the root to each vertex $a$ and join $a$ to each leaf
$(a,b)$. Define $X_\circ$ to be diagonal, with entries $0$ at
the root, $a/4$ at vertex $a$, and $b/4$ at leaf $(a,b)$,
each multiplied by the algebra's identity.

Order the leaves lexicographically and form the integer matrix
\[
 \mathsf V_{i,(a,b)}=a^{\beta^{(i)}_1}b^{\beta^{(i)}_2}
 \qquad(1\le i\le8).
\]
Its determinant is $16$, and the maximum absolute row and column
sums of $\mathsf V^{-1}$ are $30$ and $33$, respectively.
For each formal symbol $u$, let $u_\circ$ be its
value in the negative substitution of Appendix~\ref{sec:quartic}.
Define a Hermitian block matrix $\mathsf H\in M_8(L(\Lambda*F_2))$
by putting the fifteen complex values in the positions
$1\le i<k\le6$, in the order of
\eqref{enc:complexreadouts}, and their adjoints in the opposite
positions. Put $q_\circ,v_\circ,\ell_\circ,s_\circ,(s_z)_\circ,W_\circ$
on diagonal positions $1$ through $6$, put $j_\circ/2$ in
positions $(1,7)$ and $(7,1)$, and set all other entries to zero.
Let
\[
 A_{\rm leaf}=\mathsf V^{-1}\mathsf H(\mathsf V^{-1})^{\mathsf T},
 \qquad
 \mathsf V A_{\rm leaf}\mathsf V^{\mathsf T}=\mathsf H.
\]
Finally, define the selfadjoint $Y_\circ$ to have $2^{-16}I$
on both orientations of every tree edge and
$2^{-16}A_{\rm leaf}$ on the leaf block, with all other entries
zero.

\subsection{Trace bounds and minimality}

\begin{theorem}\label{thm:twovariables}
The polynomial $\Psi_F$ in \eqref{enc:formula} is Hermitian,
has integer coefficients and degree at most $48$, and satisfies
\[
 \tr_d\Psi_F(X,Y)\ge243/4
\]
for every $d\ge1$ and every pair of selfadjoint matrix contractions
$X,Y\in M_d$. The specified pair $(X_\circ,Y_\circ)$ consists
of contractions and satisfies
\[
 (\tr_{12}\otimes\tau)\Psi_F(X_\circ,Y_\circ)=-81.
\]
\end{theorem}
\begin{proof}
First suppose that $\chi(X)=0$, and abbreviate $e=e(X)$.
Then $e$ is a projection. Cyclicity identifies $\tr_dG_F(X,Y)$
with the ambient trace of $F_\flat$ evaluated in $eM_de$ at
the symbols $eB_u(X,Y)e$. The seven selfadjoint symbols remain
selfadjoint. If $r=\operatorname{rank}e>0$, the quartic theorem
gives
\[
 \tr_dG_F(X,Y)\ge\frac{3r}{4d}\ge0;
\]
if $e=0$, this trace is zero. No norm bounds on the compressed
readouts are needed.

For the coefficient norms in the primitive alphabet $X,Y$,
directly from the readout formulas we have
\[
 \|R_{(a)}\|_{\mathrm{coeff},1}\le2^{21},\quad
 \|R_{(a,b)}\|_{\mathrm{coeff},1}\le2^{42},\quad
 \|B_u\|_{\mathrm{coeff},1}\le2^{101},\quad
 \|e\|_{\mathrm{coeff},1}=35<2^6.
\]
The bound for $B_u$ includes the sum defining $B_j$.
Since $F_\flat$ has formal degree four,
\[
 \deg G_F\le48,\qquad
 \|G_F\|_{\mathrm{coeff},1}\le c_F2^{428}.
\]
For $t\in[-1,1]$, choose a nearest
$\widehat t\in\{0,1/4,1/2,3/4\}$. The three factors of
$\chi(t)$ corresponding to the other roots have absolute value
at least $1/2$, so $|t-\widehat t|\le2|\chi(t)|$.
Functional calculus therefore gives a contraction $\widehat X$
with $\chi(\widehat X)=0$ and
$\|X-\widehat X\|_2\le2\|\chi(X)\|_2$.
Telescoping the words of $G_F$, with $Y$ unchanged, yields
\[
 \tr_dG_F(X,Y)
 \ge-2\cdot48\,c_F2^{428}\|\chi(X)\|_2
 \ge-K_F\|\chi(X)\|_2.
\]
As $\chi(X)$ is selfadjoint, completing the scalar square gives
\[
 81^{-1}\tr_d\Psi_F(X,Y)
 \ge1-24K_F\|\chi(X)\|_2
       +(24K_F\|\chi(X)\|_2)^2\ge3/4.
\]

The readouts have integer coefficients, and
$e\in\frac13\Z[X]$. Each encoded word contains at most four
factors $e$, including the image $\mathcal T(1)=e$ of the
constant. Thus $81=3^4$ clears these denominators, and the even
coefficient $24$ clears the symmetrization denominator.
The real integrality of the formal input was established above.
This proves the integer coefficient claim; Hermitian symmetry
follows from \eqref{enc:formula}, and the squared penalty has
degree eight, so the total degree is at most $48$.

For the specified pair, $\|X_\circ\|=3/4$ and
$\chi(X_\circ)=0$. Every block of $\mathsf H$ has norm at most
one. The inverse matrix row and column bounds therefore give
the following bound on the block rows of $Y_\circ$:
\[
 \|Y_\circ\|
 \le2^{-16}\max\{4,1+30\cdot8\cdot33\}
 =7921\cdot2^{-16}<1.
\]
Here $4$ bounds the degree of each nonleaf vertex, and the
additional $1$ in a leaf row accounts for its edge to its parent.
Thus $(X_\circ,Y_\circ)$ is a contraction pair.

At this pair, the path recurrence \eqref{enc:paths} gives
\[
 eR_{(a,b)}=
 \sum_{\substack{1\le u,v\le3\\(u,v)\ne(3,3)}}
 u^av^b E_{\varnothing,(u,v)}\otimes1,
\]
where $E_{\varnothing,(u,v)}$ is the matrix unit from the leaf
$(u,v)$ to the root. The subtraction in \eqref{enc:paths}
removes the return to the root. It follows that $eC_{ik}e$
has root block $(\mathsf V A_{\rm leaf}\mathsf V^{\mathsf T})_{ik}
=\mathsf H_{ik}$. Consequently every input symbol is recovered
in this corner, including $j$ by adding its two halves.
The corner trace identity used above gives
\[
 (\tr_{12}\otimes\tau)G_F(X_\circ,Y_\circ)=-1/12,
 \qquad
 (\tr_{12}\otimes\tau)\Psi_F(X_\circ,Y_\circ)
 =81(1-24/12)=-81.\qedhere
\]
\end{proof}

Two selfadjoint variables are minimal for this separation on
contractions. Indeed, a Hermitian polynomial in one selfadjoint
variable is a real polynomial. Nonnegative scalar evaluations
make it nonnegative on $[-1,1]$, and continuous functional
calculus then gives a nonnegative value at every selfadjoint
contraction in a tracial algebra.

\Needspace{0.85\textheight}
\section{The expanded polynomial}\label{app:nativeexpanded}
For the selfadjoint variables $X_1,\ldots,X_6$, the polynomial
\eqref{eq:nativefinal} has the following expansion, where
$J=2^{55226}$ and $M=2^{2J+54}$ as in \eqref{eq:globalconstants}.
Powers of words denote repeated concatenation, and adjoints reverse
the order of the letters. The display lists the two signal words
and the fifteen penalty words with their reversals; the constant
term is $30M-1$.

\begingroup\small\setlength{\jot}{1pt}
\begin{equation}\label{csimp:r3:expanded33}
\begin{aligned}
f={}&30M-1\\
&+X_{4}X_{6}X_{4}X_{5}(X_{6}X_{4})^{2}X_{2}X_{3}(X_{4}X_{6})^{3}X_{3}X_{2}(X_{4}X_{6})^{2}X_{5}X_{4}X_{6}X_{4}X_{2}X_{3}(X_{4}X_{6})^{3}X_{3}X_{2}\\
&+X_{2}X_{3}(X_{6}X_{4})^{3}X_{3}X_{2}X_{4}X_{6}X_{4}X_{5}(X_{6}X_{4})^{2}X_{2}X_{3}(X_{6}X_{4})^{3}X_{3}X_{2}(X_{4}X_{6})^{2}X_{5}X_{4}X_{6}X_{4}\\
&-M\bigl((X_{4}X_{6})^{6}+(X_{6}X_{4})^{6}\\
&\qquad+X_{6}X_{1}X_{6}X_{1}X_{6}X_{1}+X_{1}X_{6}X_{1}X_{6}X_{1}X_{6}\\
&\qquad+X_{4}X_{6}X_{1}X_{5}X_{6}X_{1}X_{6}X_{4}X_{1}X_{6}X_{5}X_{1}+X_{1}X_{5}X_{6}X_{1}X_{4}X_{6}X_{1}X_{6}X_{5}X_{1}X_{6}X_{4}\\
&\qquad+X_{2}X_{3}X_{1}X_{5}X_{6}X_{1}X_{3}X_{2}X_{1}X_{6}X_{5}X_{1}+X_{1}X_{5}X_{6}X_{1}X_{2}X_{3}X_{1}X_{6}X_{5}X_{1}X_{3}X_{2}\\
&\qquad+X_{5}(X_{6}X_{4})^{2}(X_{6}X_{5}X_{4}X_{6}X_{4})^{2}X_{5}(X_{6}X_{4})^{2}\\
&\qquad+(X_{4}X_{6})^{2}X_{5}X_{4}X_{6}X_{4}(X_{5}(X_{6}X_{4})^{2})^{2}X_{6}X_{5}\\
&\qquad+X_{1}X_{5}X_{6}X_{1}X_{6}X_{5}X_{1}X_{6}X_{5}X_{1}(X_{5}X_{6}X_{1})^{3}X_{6}X_{5}X_{1}X_{6}X_{5}\\
&\qquad+X_{5}X_{6}X_{1}X_{5}X_{6}X_{1}(X_{6}X_{5}X_{1})^{3}X_{5}X_{6}X_{1}X_{5}X_{6}X_{1}X_{6}X_{5}X_{1}\\
&\qquad+X_{1}X_{5}X_{6}X_{1}X_{5}X_{6}X_{1}X_{6}X_{5}X_{1}(X_{6}X_{5}X_{4}X_{6}X_{4})^{2}X_{5}X_{6}X_{1}X_{5}X_{6}X_{1}X_{6}X_{5}X_{1}X_{6}X_{5}X_{1}X_{4}\\
&\qquad\quad{}\cdot X_{6}X_{4}X_{5}(X_{6}X_{4})^{2}\\
&\qquad+(X_{4}X_{6})^{2}X_{5}X_{4}X_{6}X_{4}X_{1}X_{5}X_{6}X_{1}X_{5}X_{6}X_{1}X_{6}X_{5}X_{1}X_{6}X_{5}(X_{4}X_{6}X_{4}X_{5}X_{6})^{2}X_{1}X_{5}X_{6}\\
&\qquad\quad{}\cdot X_{1}X_{6}X_{5}X_{1}X_{6}X_{5}X_{1}\\
&\qquad+(X_{5}X_{6})^{5}+(X_{6}X_{5})^{5}\\
&\qquad+X_{2}X_{3}X_{6}X_{2}X_{3}X_{6}X_{3}X_{2}X_{6}X_{3}X_{2}X_{6}+X_{6}X_{2}X_{3}X_{6}X_{2}X_{3}X_{6}X_{3}X_{2}X_{6}X_{3}X_{2}\\
&\qquad+(X_{6}X_{2}X_{3}(X_{6}X_{4})^{2})^{2}(X_{6}X_{3}X_{2}(X_{6}X_{4})^{2})^{2}\\
&\qquad+((X_{4}X_{6})^{2}X_{2}X_{3}X_{6})^{2}((X_{4}X_{6})^{2}X_{3}X_{2}X_{6})^{2}\\
&\qquad+X_{2}X_{3}(X_{4}X_{6})^{2}X_{2}X_{3}(X_{6}X_{4})^{2}X_{3}X_{2}((X_{4}X_{6})^{2}X_{3}X_{2}X_{6})^{2}\\
&\qquad+(X_{6}X_{2}X_{3}(X_{6}X_{4})^{2})^{2}X_{2}X_{3}(X_{4}X_{6})^{2}X_{3}X_{2}(X_{6}X_{4})^{2}X_{3}X_{2}\\
&\qquad+X_{2}X_{3}X_{5}X_{6}X_{3}X_{2}X_{6}X_{5}+X_{5}X_{6}X_{2}X_{3}X_{6}X_{5}X_{3}X_{2}\\
&\qquad+X_{2}X_{3}(X_{6}X_{4})^{2}X_{5}(X_{6}X_{4})^{2}X_{6}X_{3}X_{2}(X_{6}X_{4})^{2}X_{6}X_{5}(X_{4}X_{6})^{2}\\
&\qquad+(X_{6}X_{4})^{2}X_{5}(X_{6}X_{4})^{2}X_{6}X_{2}X_{3}(X_{6}X_{4})^{2}X_{6}X_{5}(X_{4}X_{6})^{2}X_{3}X_{2}\\
&\qquad+X_{2}X_{3}(X_{4}X_{6})^{2}X_{5}X_{4}X_{6}X_{4}X_{3}X_{2}(X_{4}X_{6}X_{4}X_{5}X_{6}X_{1}X_{5}X_{6}X_{1}X_{6}X_{5}X_{1}X_{6}X_{5}X_{1})^{2}\\
&\qquad\quad{}\cdot (X_{4}X_{6}X_{4}X_{1}X_{5}X_{6}X_{1}X_{5}X_{6}X_{1}X_{6}X_{5}X_{1}X_{6}X_{5})^{2}\\
&\qquad+(X_{5}X_{6}X_{1}X_{5}X_{6}X_{1}X_{6}X_{5}X_{1}X_{6}X_{5}X_{1}X_{4}X_{6}X_{4})^{2}\\
&\qquad\quad{}\cdot (X_{1}X_{5}X_{6}X_{1}X_{5}X_{6}X_{1}X_{6}X_{5}X_{1}X_{6}X_{5}X_{4}X_{6}X_{4})^{2}X_{2}X_{3}X_{4}X_{6}X_{4}X_{5}(X_{6}X_{4})^{2}X_{3}X_{2}\\
&\qquad+((X_{4}X_{6})^{3}X_{5})^{2}+(X_{5}(X_{6}X_{4})^{3})^{2}\bigr).
\end{aligned}
\end{equation}
\endgroup

\bibliographystyle{amsplain}
\bibliography{reference}

\end{document}